\documentclass[twocolumn]{article}

\usepackage[T1]{fontenc}
\usepackage[utf8]{inputenc}
\usepackage{lmodern}

\usepackage{amsmath,amssymb,amsthm}

\usepackage{graphicx}
\usepackage{tikz}
\usepackage{tikz-cd}
\usepackage{pgfplots}
\pgfplotsset{compat=1.18}

\usepackage{booktabs,array,multirow,colortbl}
\usepackage{subcaption}
\usepackage{adjustbox}
\usepackage{enumitem}
\usepackage{float}
\usepackage{geometry}
\usepackage{xcolor}
\usepackage{csquotes}
\usepackage{fancyhdr}
\usepackage[backend=biber,style=numeric,sorting=nyt]{biblatex}
\usepackage[colorlinks=true,linkcolor=blue,citecolor=blue,urlcolor=blue]{hyperref} % ganz zum Schluss

\newtheorem{theorem}{Theorem}[section] % Nummeriert nach den Abschnitten
\newtheorem{lemma}[theorem]{Lemma} 
\newtheorem{proposition}{Proposition}[section]
\newtheorem{example}{Example}[section]    % Lemma-Umgebung teilt sich die
\newtheorem{assumption}{Assumption}[section]
\newtheorem{remark}{Remark}[section]
\newtheorem{corollary}{Corollary}[section]
\newtheorem{definition}[theorem]{Definition} % Definition teilt ebenfalls die Nummerierung

\title{\textbf{Torsion in Persistent Homology and Neural Networks}}
\author{
	\textbf{Maria Walch} \thanks{\texttt{walch@rptu.de, maria.walch@iav.de, maria.walch@igw.uni-heidelberg.de}} \\
	RPTU Kaiserslautern-Landau \\
	IAV GmbH \\
	Heidelberg University
}

\date{}

\begin{document}
	
	\maketitle
	
	\begin{abstract}
\textbf{We explore the role of torsion in hybrid deep learning models that incorporate topological data analysis, focusing on autoencoders. While most TDA tools use field coefficients, this conceals torsional features present in integer homology. We show that torsion can be lost during encoding, altered in the latent space, and in many cases, not reconstructed by standard decoders.
Using both synthetic and high-dimensional data, we evaluate torsion sensitivity to perturbations and assess its recoverability across several autoencoder architectures. Our findings reveal key limitations of field-based approaches and underline the need for architectures or loss terms that preserve torsional information for robust data representation.}
	\end{abstract}
	
\section{Introduction}

Topological data analysis (TDA) has become increasingly relevant across scientific disciplines, from cancer image analysis to theoretical particle physics \cite{Rabadán2020, thomas2022topological}. In parallel, several efficient and accessible TDA libraries have been developed, such as Ripser, GUDHI, and giotto-tda \cite{Bauer2021Ripser,giottotda,gudhi:urm}. With the growing dominance of neural networks in data-driven applications, TDA techniques have also found their way into machine learning pipelines, particularly in unsupervised settings using autoencoders.\\

\noindent Autoencoders have proven effective for nonlinear dimensionality reduction and unsupervised representation learning in a broad application field \cite{Finke2021Autoencoders, Badsha2020, fengbreast, hooliver}. Recent work has sought to enhance their topological sensitivity by incorporating TDA-inspired components — either through loss functions that reward topological similarity between input and output \cite{moor2021topologicalautoencoders, trofimov2023learningtopologypreservingdatarepresentations}, or through architectural modifications tailored to preserve topological structure \cite{carrière2020perslayneuralnetworklayer}.\\

\noindent A central topological feature that remains largely overlooked in this context is torsion. While torsion is often associated with abstract examples like the Möbius strip or Klein bottle, it also arises in applied contexts, such as tumor growth models \cite{amster2024topological}, Section 9.5 and the patch space of natural images \cite{Carlsson2008LocalBehavior}. Furthermore, recent work \cite{obayashi2023field, kahle2018cohenlenstraheuristicstorsionhomology} suggests that torsional features are likely to occur in high-dimensional data — the kind routinely encountered in deep learning applications.\\

\noindent However, torsion is invisible over field coefficients and can only be detected via integer (persistent) homology, as implied by the universal coefficient theorem (\cite{hatcher}, Theorem 3A.3 for homology). Since most hybrid approaches in deep learning rely on field coefficients for computational tractability, they inherently neglect torsional information. This omission is typically unacknowledged, and it is often assumed that the choice of coefficient ring has little or no impact on downstream tasks, explainability or interpretability.\\

\noindent In this work, we challenge that assumption. We show — both computationally and mathematically — that torsion is a fragile but significant topological feature whose presence exposes fundamental limitations of current TDA-enhanced autoencoders. Even architectures that explicitly incorporate TDA loss functions are unable to capture or reconstruct torsion reliably.

 \subsection{Contributions}
 
 Torsion remains a rarely addressed phenomenon in both computational and theoretical topology. Yet it plays a key role in dynamical systems theory, where it serves as a distinguishing feature — for instance, separating models of tumor growth from chaotic attractors such as the Rössler system \cite{amster2024topological}. Despite its relevance, torsion is often treated as a theoretical curiosity and is largely absent from practical applications of topological data analysis (TDA), particularly in deep learning.\\
 \noindent Autoencoders, as nonlinear tools for representation learning, offer a promising interface with TDA. They enable topological computations in low-dimensional latent spaces, where standard algorithms would otherwise be computationally infeasible. However, the extent to which these models preserve or distort torsional features remains unknown.\\
 \noindent This work is the first to systematically explore torsion phenomena in the context of neural networks. We identify several structural and computational limitations and highlight challenges for future hybrid architectures.\\
 
 \noindent From a \textbf{mathematical perspective}, we show:
 \begin{itemize}
 	\item how torsion is typically lost due to dimensionality reduction in the encoder (Section~\ref{sec:dimred}),
 	\item how torsion that is preserved in the latent space leads to changes in persistence diagrams over field coefficients (Section~\ref{sec:torsreplatsp}),
 	\item and which decoder architectures fail to reconstruct torsion — and in which cases reconstruction becomes possible (Section~\ref{sec:lindec}).
 \end{itemize}
 
\noindent  In particular, our results imply that, even when working over a coefficient field, one cannot expect a trained autoencoder to induce isomorphisms between the homology groups of the input and latent spaces when torsional phenomena are involved. This challenges the assumptions underlying many TDA-inspired loss functions, which often aim to align field-based homological invariants between input, output, or latent representations \cite{moor2021topologicalautoencoders, trofimov2023learningtopologypreservingdatarepresentations}.\\

\noindent From a \textbf{computational perspective}, we demonstrate:
 \begin{itemize}
 	\item that torsion is highly unstable under local perturbations of the input data — small changes such as the displacement or removal of a single point can eliminate torsion entirely, and the prime order of torsion classes is sensitive to such modifications,
 	\item and that torsion reconstruction is not guaranteed, even for architectures explicitly designed to preserve topological features.
 \end{itemize}
 
\noindent  To support our claims, we perform experiments on both synthetic and high-dimensional datasets. For illustration, we use the double and triple loop point clouds embedded in three-dimensional Euclidean space (\cite{obayashi2023field}, Figure~4(a) and 4(e)), as well as randomly sampled high-dimensional point clouds. Torsion is detected using the algorithm presented in \cite{obayashi2023field}, implemented in the \texttt{homcloud} persistent homology software package \cite{homcloud}.

 \subsection{Related Work}
 Zomorodian, who introduced the so-called standard algorithm for computing persistent homology over an arbitrary coefficient ring in \cite{zomorodian2005computing}, was the first to address torsion phenomena from a computational perspective within topological data analysis. This line of inquiry was further developed by Obayashi \cite{obayashi2023field}, who examined the impact of torsion phenomena on invariants computed over a coefficient field in persistent homology, both from a computational and a mathematical standpoint.\\
 Computational results also indicate that torsion is a non-negligible phenomenon in high-dimensional data sets. Building on the model developed by Erd\H{o}s and R\'enyi for generating one-dimensional random simplicial complexes \cite{erdos1959on, erdos1960evolution}, Linial and Meshulam introduced a model for generating higher-dimensional random simplicial complexes \cite{linial2006homological, meshulam2009homological}. As shown by Kahle et al.~\cite{kahle2018cohenlenstraheuristicstorsionhomology}, there is almost always a moment in this random complex process at which a \enquote{burst of torsion} appears, marking a critical topological transition.

\section{Persistent Homology}
We define persistent homology with regard to its application to machine learning settings. In this context, we understand a data set $\mathcal{D}=\{x_i \in \mathbb{R}^d \; | \; 1 \leq i \leq n, \; n \in \mathbb{N} \}$ as a finite (Euclidean) metric space, also called a \textit{point cloud}. On this data set we construct a finite geometric simplicial complex -- a topological space constructed by gluing together dimensional simplices (points, line segments, triangles, tetrahedrons, etc.). We impose a total order 
\begin{equation}
	\label{filtration}
	K: \emptyset = K_0 \subset K_1 \subset ... \subset K_N
\end{equation}
on the set of simplices called a \textit{filtration}. In the following, we work predominantly with \textit{Vietoris--Rips filtrations} (in short: \enquote{Rips filtrations}).
\begin{definition}[Vietoris--Rips Filtration]
	\label{def:vrfiltr}
	Given a point cloud  ${\mathcal{D} = \{x_i \in \mathbb{R}^d \; | \; 1 \leq i \leq n \}}$ and positive real numbers \( \delta_0 < \delta_1 < \cdots < \delta_m \) with \( \delta_i \in \mathbb{R}_{>0} \), the Vietoris--Rips filtration is the nested sequence of simplicial complexes
	\begin{equation}
		\emptyset = V_{\delta_0}(\mathcal{D}) \subseteq V_{\delta_1}(\mathcal{D}) \subseteq \cdots \subseteq V_{\delta_m}(\mathcal{D})
	\end{equation}
	where a $k$-simplex in $V_{\delta_i}(\mathcal{D})$ exists if all pairwise distances between its $k+1$ vertices are $\leq \delta_i$.\end{definition}
\noindent Definition~\ref{def:vrfiltr} introduces the notion of an abstract filtered simplicial complex, a purely combinatorial object that can be endowed with the structure of a topological space via \emph{geometric realization} (\cite{edelharerbook}, Section III.1). For a Vietoris--Rips filtration constructed on a finite metric space $M \subset \mathbb{R}^d$, the dimension of its geometric realization may exceed $d$. In the following, we restrict attention to geometric filtered simplicial complexes that are embedded in $\mathbb{R}^d$, thereby excluding phenomena such as Vietoris--Rips bubbles~\cite{vrplanarpoint}. Our objects of study are thus finite geometric filtered simplicial complexes $K$ constructed from finite metric spaces $M \subset \mathbb{R}^d$.
\noindent The $p$-th persistent homology of a finite-dimensional filtered simplicial complex with coefficients in a commutative ring $R$ is given by
\begin{equation}
	\label{eq:persfunc}
	PH_p(K;R): H_p(K_0;R) \rightarrow ... \rightarrow H_p(K_N;R),
\end{equation}
where $H_p(K_i,R)$ denotes the $p$-th simplicial homology group for a subcomplex $K_i$, $0 \leq i \leq N$, in the filtration. This sequence captures how topological features such as connected components, loops, and voids appear and disappear across different scales. The image of the map \( H_p(K_i; R) \to H_p(K_j; R) \) is referred to as a \textit{persistent homology group} and identifies features that persist across the filtration interval $[i,j)$.\newline
\noindent Over \( R = \mathbb{Z} \), the simplicial homology groups \( H_p(K, \mathbb{Z}) \) decompose, by the structure theorem for finitely generated abelian groups \cite{dummit2003abstract, gallian2021abstract}, into a direct sum of a free part and a finite torsion subgroup.

\subsection{Field Coefficients and Persistence Diagrams}
For field coefficients, the torsion subgroup vanishes, and persistent homology is fully determined by the ranks of the free parts of the persistent homology groups (\textit{field Betti numbers}). As a result, its topological invariants admit a simple representation: namely, the persistence intervals \( [i, j) \) over which a feature exists within the filtration encode all relevant information. These intervals can be visualized either as points in a \textit{persistence diagram}—where the \( x \)-axis denotes the birth time of a feature and the \( y \)-axis its death time—or as horizontal bars in a so-called \textit{barcode diagram}. In the following, we denote by \( \mathrm{dgm}_p(X; R) \) the persistence diagram in homological dimension \( p \), with coefficients in the ring \( R \), of a topological space \( X \) obtained as the topological realization of a filtration of a finite metric space.
\\
\noindent In computational data analysis, persistent homology is highly popular because it is stable under noise and sufficiently small perturbations of the data. In other words, the metric distance (\textit{bottleneck distance}) between the persistence diagrams of two topological spaces that have a small Hausdorff distance is also small \cite{skraba2023wassersteinstabilitypersistencediagrams}. On the other hand, a common criticism for persistent homology is that individual outliers can have significant effects — they can even be considered \enquote{deadly} \cite{chazal2014robusttopologicalinferencedistance} in the sense that they significantly change the field Betti numbers of the underlying topological space. 

\subsection{Torsion}
Torsion coefficients annihilate specific elements of the homology group via multiplication. Specifically, for each element \( z \in H_p(K; \mathbb{Z}) \) of the torsion subgroup, there exists an integer \( q^{k} \), for $q$ prime, such that
\[
q^{k} z = 0,
\]
where $q^k$ is a \textit{torsion coefficient} and $q$ its \textit{prime divisor}. We say that the integral homology group of $K$ for homology dimension $p$ exhibits \textit{q-torsion} \cite{hatcher}. 
\begin{remark}
	\label{rem:torsfiltsimpcomp}
	In a \textbf{filtered} simplicial complex, torsion may occur either in the absolute homology group of a subcomplex \( K_i \), or in a relative homology group \( H_*(K_i, K_j) \), where \( K_j \subseteq K_i \) are subcomplexes appearing in the filtration of \( K \). An instance where torsion arises exclusively in relative homology, but not in any absolute homology group of the filtration, is given in \cite{obayashi2023field}, Example~1.5.
\end{remark}
\noindent For our subsequent discussion, we need the following result from \cite{obayashi2023field}.
\begin{theorem}[\cite{obayashi2023field}, Theorem 1.6]
	\label{th:obay}
	Let $X$ be a topological space obtained as the topological realization of a filtration of a finite metric space, $k$ a coefficient field and $p$ a homology dimension. The persistence diagram $\mathrm{dgm}_p(X;k)$ is independent of the choice of coefficient field, if the relative integral homology group $H_p(X_i,X_j;\mathbb{Z})$ is free for any indices $0 \leq i <j \leq N$ in the filtration of $X$ and $H_{p-1}(X_i;\mathbb{Z})$ is free for any $0 \leq i \leq N$.
\end{theorem}

\section{Neural Networks}
Neural networks represent a class of machine learning models composed of interconnected computational units, inspired by the architecture of biological neurons. These units, or \textit{neurons}, are typically organized into \textit{layers}. A network with more than three layers is called \textit{deep}. We consider a neural network \( f \) as a composition of \( N \in \mathbb{N} \) sequentially connected layers \( f_j \), \( 1 \leq j \leq N \):
\begin{equation}
	\label{eq:network}
	f(x_{i,nn}) = f_N \circ f_{N-1} \circ \dots \circ f_2 \circ f_1(x_{i,nn}).
\end{equation}
The composition \( f_N \circ f_{N-1} \circ \dots \circ f_1 \) represents the \textit{forward propagation} of an input \( x_{i,nn} \) through the network's layers, yielding the final output. Thus such a network is also known as a \textit{feedforward network}.\newline
\noindent Let $\mathcal{D}=\{x_i \in \mathbb{R}^d \; | \; 1 \leq i \leq n, \; n \in \mathbb{N} \}$ be a data set serving as the \textit{input} to the neural network. The output \( o_{ijk} \) of neuron \( k \) in layer \( j \) for an input \( x_i \in \mathbb{R}^v \) is computed as:
\begin{equation}
	\label{eq:neurcomp}
	o_{ijk} = \sigma \left( W_{jk} \cdot x_i + b_{jk} \right),
\end{equation}
where \( \sigma: \mathbb{R} \to \mathbb{R} \) denotes a typically nonlinear \textit{activation function}. The term \( W_{jk} \in \mathbb{R}^v \) represents the \textit{weight vector} associated with neuron \( k \) in layer \( j \), and \( x_i \in \mathbb{R}^v \) is the input to this layer. Finally, \( b_{jk} \in \mathbb{R} \) denotes the \textit{bias} term associated with neuron \( k \) in layer \( j \). 
\begin{assumption}[Euclidicity]
	\label{ass:euclid}
For a given layer \( f_j \) containing \( M \in \mathbb{N} \) neurons, the output for an input \( x_{i,l} \in \mathbb{R}^v \) \textit{to the layer} is assumed to be a vector \( f_j(x_{i,l}) \in \mathbb{R}^M \), defined as:
\begin{equation}
	\label{eq:layerout}
	f_j(x_{i,l}) = [o_{ij1}, o_{ij2},..., o_{ijM}],
\end{equation}
where each \( o_{ijk} \in \mathbb{R} \) for \( 1 \leq k \leq M \) denotes the output of the \( k \)-th neuron in layer \( j \) for the input \( x_{i,l} \).
\end{assumption} 
\noindent We denote the term in Eq.~\eqref{eq:layerout} as the \textit{layer state}. 

\noindent The weights and biases are \textit{parameters} of the neuronal network that are \textit{optimized} during an iterative process called \textit{training}. The difference (also called \textit{loss}) between an output $y_i=f(x_{i,nn})$ associated with a specific input $x_{i,nn}$ to a neuronal network is measured with respect to a pre-selected metric distance (the \textit{error metric} or \textit{loss term}). This is selected in a way that a minimum loss would indicate the network to have learned a desirable characteristic of the input data. By default, the mean squared error (MSE) is selected in applications, defined as:  

\[
\text{MSE} = \frac{1}{n} \sum_{i=1}^{n} (y_i - \hat{y}_i)^2
\]

\noindent where \( y_i \) are the true outputs, \( \hat{y}_i \) are the predicted outputs and $n$ denotes the cardinality of samples in the input data.

\section{Autoencoders}
Autoencoders \cite{rumelhart} constitute a network architecture initially introduced for dimensionality reduction \cite{GoodBengCour16}. Via the use of nonlinear activation functions, autoencoders represent a nonlinear dimensionality reduction method.
\begin{definition}[Autoencoder]
	\label{def:AE}
	Let $\mathcal{D} \subset \mathbb{R}^d$ be an input data set. Under Assumption~\ref{ass:euclid}, an autoencoder is a neural network comprising two parameterized functions (we assume each given in the form of Eq.~\eqref{eq:network}), namely, the \textbf{encoder}
	\begin{equation}
		\label{enceq}
		f_{\text{enc}}(\cdot;\theta_{\text{enc}}): \mathbb{R}^d \rightarrow \mathbb{R}^s
	\end{equation} 
	and the \textbf{decoder}
	\begin{equation}
		\label{deceq}
		f_{\text{dec}}(\cdot;\theta_{\text{dec}}): \mathbb{R}^s \rightarrow \mathbb{R}^d,
	\end{equation}
	where $\theta_{\text{enc}}$ denote the encoder's parameters and $\theta_{\text{dec}}$ denote the decoder's parameters respectively.
	Together they define the autoencoder as a function
	\begin{equation}
		\begin{split}
			f_{\text{auto}}&: \mathbb{R}^d \to \mathbb{R}^d \\	
			f_{\text{auto}} &:= f_{\text{dec}} \circ f_{\text{enc}}.
		\end{split}
	\end{equation} 
	The output of the encoder is called the \textbf{latent space}.
\end{definition}
\noindent The defining characteristic of an autoencoder is that the latent layer contains significantly fewer neurons than the input and output layers, i.e. $s<<d$ in Definition~\ref{def:AE}.\\
\noindent If the encoder does not exhibit stochasticity, the encoder is surjective. However, autoencoders are not injective and thus not bijective if not explicitly constructed so.\\
\noindent With regard to their broad applicability and the need for non-model driven ways for data analysis, there exists a wide variation of autoencoder architectures. We call an autoencoder trained using the MSE loss \textit{vanilla}. In Appendix~\ref{app:AE}, we present two modifications of the loss term inspired from TDA, that we use for subsequent experiments in this work.

\section{Torsion and Stability}
In an autoencoder, the data undergoes (non-linear) transformations through the weight matrices and activation functions. This raises two questions:
\begin{enumerate}[label=\Roman*.]
	\item How do these transformations affect the torsion subgroups in the input data?
	\item Whether and how does the encoder remove torsion from the data so that it no longer exists in the latent space?
	
\end{enumerate}

\noindent Two factors are crucial for answering these questions: (1) the dimensionality reduction and (2) the shift of the data. Specifically, based on Assumption~\ref{ass:euclid}, the transformation induced by the weight matrix and the activation function is understood as a shift of the data points in Euclidean space.

\subsection{Data Shift}
Torsion is sensitive to small changes and perturbations in the underlying data. This becomes apparent already in the following minimal example using a Möbius strip.
\begin{example} [Möbius Strip]
	Figure~\ref{fig:moebiusstrip} exhibits an illustration of the minimal triangulation of the Möbius strip. Clearly, the removal of $a_3$ results in the removal of the non-orientability and thus the relative torsion in the Möbius strip structure, despite still exhibiting an $H_1$--cycle due to the remaining connection at $a_0$. 
\end{example}

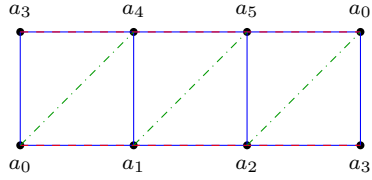
\begin{figure}[ht]
	\centering
	\begin{tikzpicture}[scale=1.5, every node/.style={font=\footnotesize}]
		% Punkte definieren
		\coordinate (a0) at (0,0);
		\coordinate (a1) at (1,0);
		\coordinate (a2) at (2,0);
		\coordinate (a3) at (3,0);
		
		\coordinate (a3t) at (0,1);
		\coordinate (a4) at (1,1);
		\coordinate (a5) at (2,1);
		\coordinate (a0t) at (3,1);
		
		% Punkte markieren
		\foreach \name in {a0,a1,a2,a3,a3t,a4,a5,a0t}
		\fill (\name) circle (1pt);
		
		% Labels unten
		\node[below=2pt] at (a0) {$a_0$};
		\node[below=2pt] at (a1) {$a_1$};
		\node[below=2pt] at (a2) {$a_2$};
		\node[below=2pt] at (a3) {$a_3$};
		
		% Labels oben (ohne _top)
		\node[above=2pt] at (a3t) {$a_3$};
		\node[above=2pt] at (a4) {$a_4$};
		\node[above=2pt] at (a5) {$a_5$};
		\node[above=2pt] at (a0t) {$a_0$};
		
		% Vertikale Rechtecke (blau)
		\draw[blue] (a0) -- (a1);
		\draw[blue] (a1) -- (a2);
		\draw[blue] (a2) -- (a3);
		\draw[blue] (a3t) -- (a0t);
		\draw[blue] (a0) -- (a3t);
		\draw[blue] (a1) -- (a4);
		\draw[blue] (a2) -- (a5);
		\draw[blue] (a3) -- (a0t);
		
		% Horizontale Verbindungen (rot gestrichelt)
		\draw[red, dashed] (a0) -- (a1) -- (a2) -- (a3);
		\draw[red, dashed] (a3t) -- (a4) -- (a5) -- (a0t);
		
		% Diagonale Verbindungen (grün gepunktet)
		\draw[green!60!black, dash dot] (a0) -- (a4);
		\draw[green!60!black, dash dot] (a1) -- (a5);
		\draw[green!60!black, dash dot] (a2) -- (a0t);
		
	\end{tikzpicture}
	\caption{Triangulation of the Möbius strip. Orientation induced by the ordering $a_0 < a_1 < a_2 < a_3 < a_4 < a_5$. }
	\label{fig:moebiusstrip}
\end{figure}

\noindent To get a sense of how easily torsion can be removed from a data structure by shifting just a few points, we used the algorithm developed in \cite{obayashi2023field} to identify the first simplex that contributed to torsion in the triple loop point cloud (taken from \cite{obayashi2023field}, Fig. 4 (b)). We added Gaussian noise to its vertices and then used the algorithm again to identify the first torsional simplex. We repeated this process until no torsion could be identified in the data. This happened after only $10$ out of $600$ data points were shifted. The results are depicted in Figure~\ref{fig:tors3loop} and summarized in Table~\ref{tab:comparison}.

\begin{figure}[ht]
	\centering
	
	\includegraphics[width=0.5\linewidth]{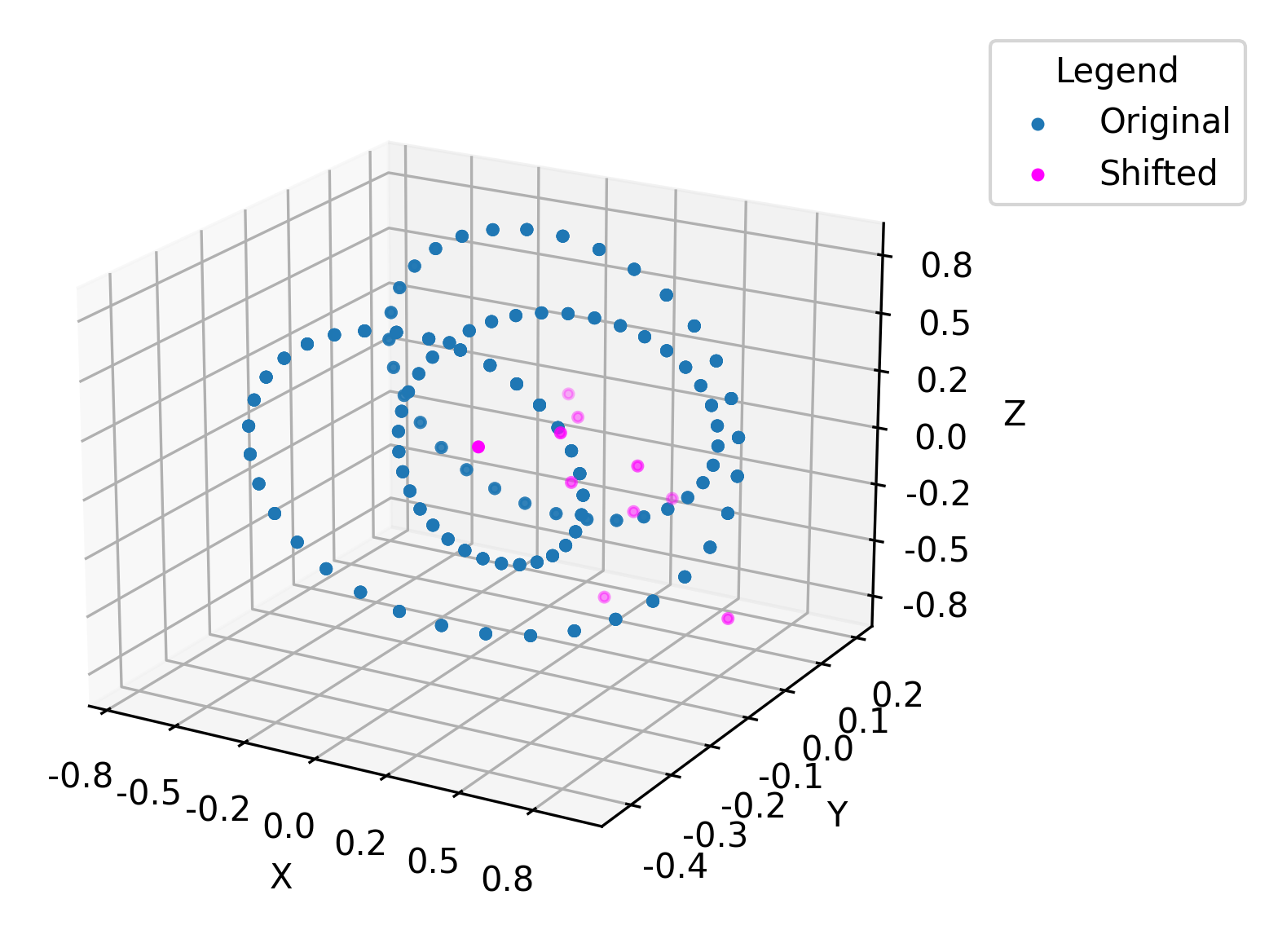}
	\caption*{(a) Triple loop point cloud with shifted points marked in pink.}
	\vspace{1em}
	
	\includegraphics[width=0.5\linewidth]{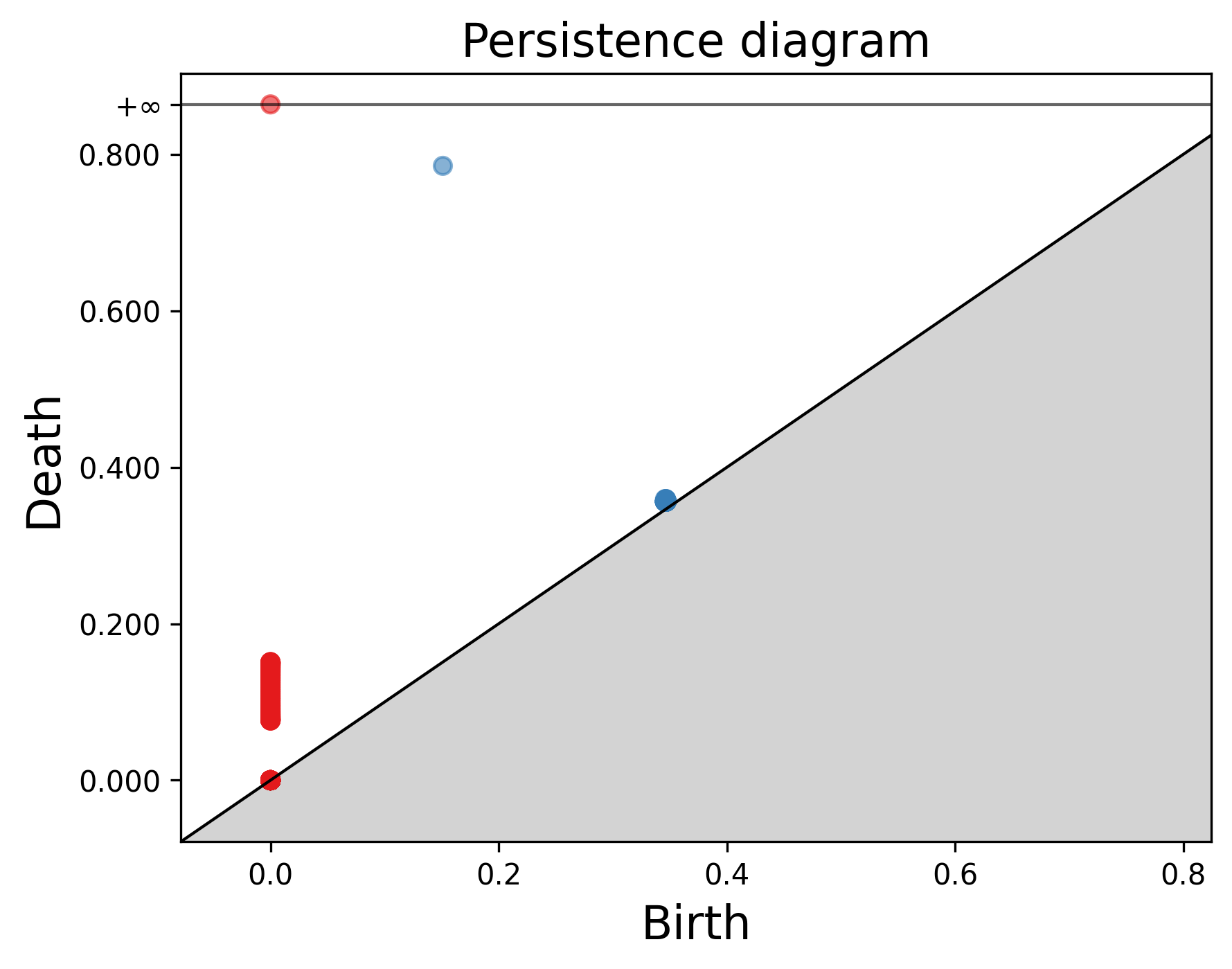}
	\caption*{(b) Persistence diagram of the triple loop point cloud (over $\mathbb{Z}/2\mathbb{Z}$).}
	\vspace{1em}
	
	\includegraphics[width=0.5\linewidth]{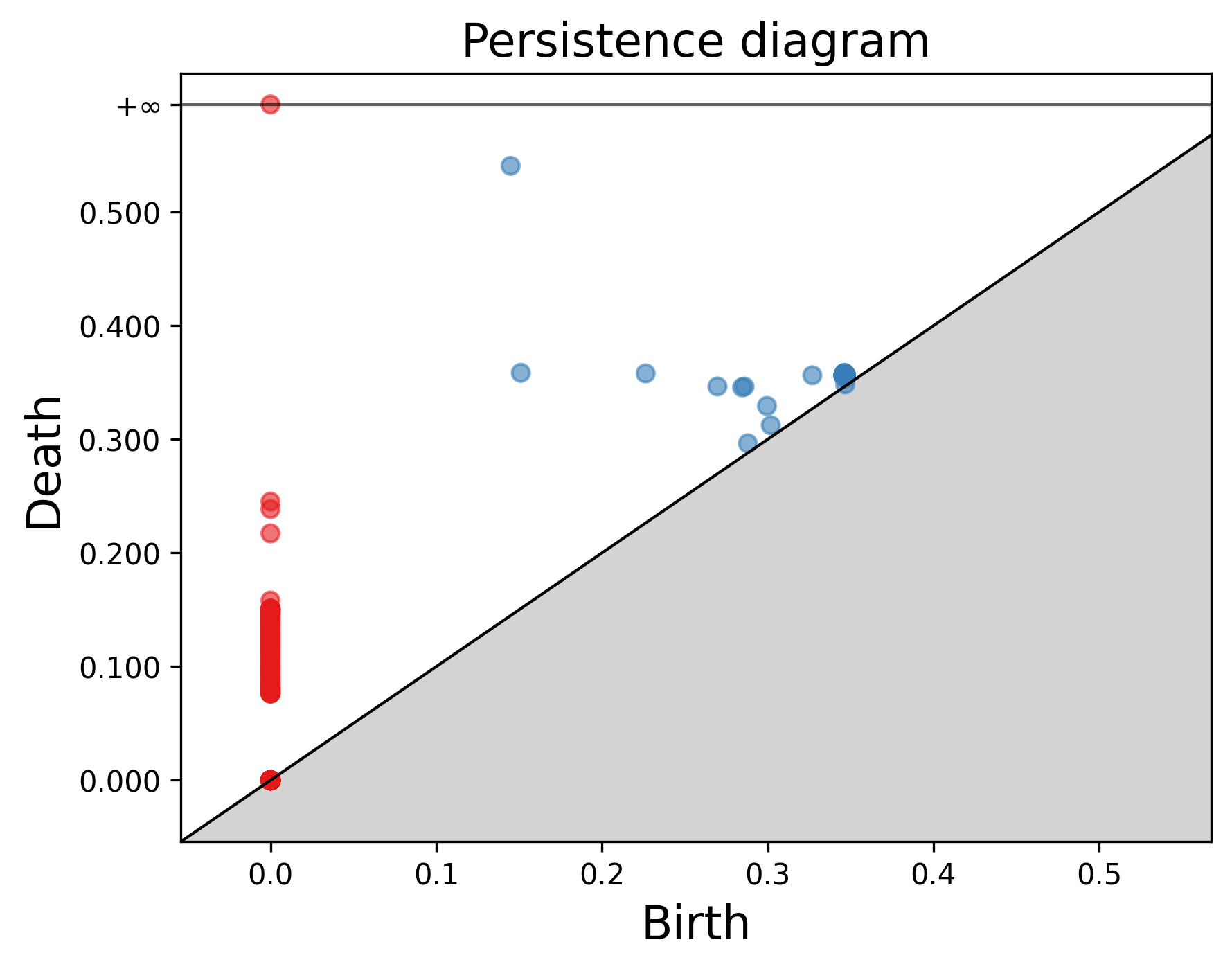}
	\caption*{(c) Persistence diagram of the shifted point cloud (over $\mathbb{Z}/2\mathbb{Z}$).}
	
	\caption{Triple loop point cloud with shifted points and corresponding persistence diagrams computed with GUDHI.}
	\label{fig:tors3loop}
\end{figure}

\begin{table}[ht]
	\centering
	\resizebox{\columnwidth}{!}{%
		\begin{tabular}{
				>{\centering\arraybackslash}p{4cm} 
				>{\centering\arraybackslash}p{2.5cm} 
				>{\centering\arraybackslash}p{3cm}
			}
			\toprule
			\textbf{Card. of shifted data points} & \textbf{MSE} & \textbf{Wasserstein dist.} \\
			\midrule
			10 & 0.0005 & 0.36 \\
			\midrule
			\multicolumn{3}{c}{\textbf{Bottleneck distances}} \\
			\midrule
			\textbf{$H_0$} & 0.09 & \\
			\textbf{$H_1$} & 0.24 & \\
			\bottomrule
		\end{tabular}
	}
	\caption{Comparison of triple loop point cloud to shifted point cloud.}
	\label{tab:comparison}
\end{table}

\noindent From Table~\ref{tab:comparison} can be deduced that first of all, the MSE is not meaningful with regard to the reconstructibility of torsional structures. Furthermore, a comparison of the two persistence diagrams shown in Figure~\ref{fig:tors3loop} unveils the aforementioned instability of (single-parameter) persistent homology with respect to \enquote{outliers} in the data.\newline

\noindent Even without the transformations leading to individual \enquote{outliers}, we demonstrate that already the multiplication with the weight matrix can induce or remove torsion from the data. \newline

\begin{example}[Weight Matrix and Torsion]
	In Figure~\ref{fig:doublepc} on the left a double loop point cloud $\in \mathbb{R}^3$ (taken from \cite{obayashi2023field}, Fig. 4 (a)) is shown. The following transformation matrix $\mathbf{T}: \mathbb{R}^3 \to \mathbb{R}^3$ is applied to this point cloud:
	\[
	\mathbf{T} = \begin{pmatrix}
		0.05 & 0 & 0 \\
		0 & 3 & 0 \\
		0 & 0 & 1
	\end{pmatrix}.
	\]
	This results in a point cloud without torsion as depicted in Figure~\ref{fig:4b}.
\end{example}

\begin{figure}[H]
	\centering
	\begin{subcaptionbox}{Original\label{fig:4a}}[0.48\columnwidth]
		{\includegraphics[width=\linewidth]{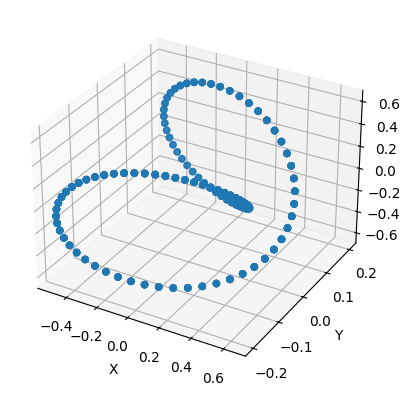}}
	\end{subcaptionbox}
	\hfill
	\begin{subcaptionbox}{Transformed\label{fig:4b}}[0.48\columnwidth]
		{\includegraphics[width=\linewidth]{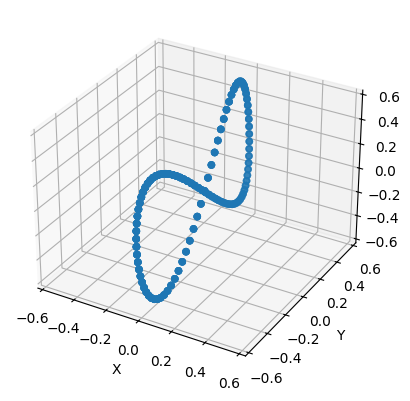}}
	\end{subcaptionbox}
	\caption{Illustration of torsion: double-loop point cloud before and after a linear transformation.}
	\label{fig:doublepc}
\end{figure}

\noindent With regard to the nonlinear activation function part, we applied six common activation functions to the double loop point cloud; the result is shown in Figure~\ref{fig:activfunc}.

\noindent From a mathematical perspective, the interaction between (network) induced point shifts and torsion is still entirely unexplored. Regarding the relationship between point shifts and torsional structures, \cite{dey2013edgecontractionssimplicialhomology} have established conditions under which an edge collapse (a very restricted class of point shifts \cite{deyedels}) \textit{does not} lead to the occurrence of relative torsion in the filtration. However, the converse problem — whether certain point shifts may induce torsion — remains entirely open.

\noindent The above two demonstrations indicate that the interaction of torsion in the input data and neural network operations is highly sensitive to the spatial distribution, scaling, and normalization of the data. 

\subsection{Dimensionality Reduction}
\label{sec:dimred}

The removal of torsional components in the latent space induced by the dimensional reduction through the encoder arises as a consequence of Alexander duality.
\begin{corollary}[\cite{hatcher}, Corollary 3.46]
	\label{cor:alexdual}
	Let \( X \subset \mathbb{R}^n \) be compact and locally contractible. Then the homology groups \( H_i(X; \mathbb{Z}) \) vanish for all \( i \geq n \), and are torsion-free for \( i = n-1 \) and \( i = n-2 \).
\end{corollary}

\begin{proposition}
	By transforming data from a $d$-dimensional input space $X \subset \mathbb{R}^d$, $d>3$ to a $\lambda$-dimensional latent space with $2<\lambda<d$ and under Assumption~\ref{ass:euclid}, all torsion effects in the absolute homology groups with homology dimension $\geq \lambda - 2$ are removed from the subspaces imposed by a Rips filtration on the data.	
\end{proposition}
\begin{proof}
	As finite CW complexes, the subspaces of the geometric realization of a Rips filtration on a finite metric space are compact and locally contractible \cite[p.~197, 520]{hatcher}.
	By Corollary~\ref{cor:alexdual}, the integral simplicial homology groups of these subspaces may exhibit a torsion part for homology dimension $1 \leq p \leq d-3$.
	We consider a Rips filtration also on the latent space $L$, whose geometric realization is by Assumption~\ref{ass:euclid} a subspace of a Euclidean space $L \subset \mathbb{R}^{\lambda}$. Thus, by Corollary~\ref{cor:alexdual}, the integral homology groups may exhibit torsion up to homology dimension $\lambda-3$.	
\end{proof}
\noindent It has been proven in \cite{obayashi2023field}, Theorem 1.6 and Corollary 1.8, that for a filtered space embedded in $\mathbb{R}^d$, the $d-1$-dimensional \textit{persistent} homology is torsionfree (which by Remark~\ref{rem:torsfiltsimpcomp} signifies that the absolute as well as the relative homology groups of the subspaces in the filtration are torsionfree).
\begin{corollary}
	By transforming data $X \subset \mathbb{R}^d$, $d>3$ from a $d$-dimensional input space to a $\lambda$-dimensional latent space $\lambda <d$ with $\lambda>2$ and under Assumption~\ref{ass:euclid}, all torsion effects in the persistent homology of the input space for homology dimension $\lambda -1 \geq 1$ are removed. 	
\end{corollary}

\section{Representing Torsion in Latent Space}
\label{sec:torsreplatsp}
The aim of training an autoencoder is to ensure that the latent space provides a low-dimensional yet highly representative encoding of the data. As we have seen in the examples taken from dynamical system theory above, torsion can serve as a discriminative feature within the data. Thus it is desirable for data classified as torsional in the input space to also be classified as torsional in the latent space. If torsion appears in the high-dimensional homology groups of the input data and these high-dimensional homology groups are no longer present in the latent data due to dimensionality reduction as demonstrated above, the representability of torsional structures in the latent space implies that these torsional structures must be reflected in lower-dimensional integral homology groups of the latent space.

\noindent This observation naturally leads to the following question:
\begin{enumerate}[label=\Roman*.]
	\setcounter{enumi}{2}  % Starts counting from III
	\item Under what conditions do torsional structures appear in the latent space when the input data exhibits such structures?	
\end{enumerate}

\noindent The representability of torsional structures in the latent space has implications for the properties that should be required of the trained encoder.

\begin{proposition}
	\label{prop:enc1}
	Let \( X \subset \mathbb{R}^d \) be a \( d \)-dimensional input space, and let 
	\( f_{\mathrm{enc}}: X \to L \) be an encoder into a 
	\( \lambda \)-dimensional latent space \( L \subset \mathbb{R}^{\lambda} \) 
	with \( \lambda \ll d \).  
	\noindent If \( X \) has \( q \)-torsion in its \( p \)-dimensional persistent homology for some \( p \geq \lambda - 1 \), and if \( f_{\mathrm{enc}} \) preserves torsional structures, in the sense that there exist a prime \( \varsigma  \) (possibly \( \varsigma = q \)) and a dimension \( p' \neq p \) (\( p' \leq \lambda-2 \)) such that \( f_{\mathrm{enc}}(X) \) has \( \varsigma \)-torsion in its \( p' \)-dimensional persistent homology, then there exists a prime \( q' \) such that \( f_{\mathrm{enc}} \) fails to induce an isomorphism on (relative) homology with coefficients in \( \mathbb{Z} / q'\mathbb{Z} \).
\end{proposition}

\begin{proof}
	The appearance of torsional structures in lower-dimensional homology implies that lower-dimensional (relative) homology groups (\( p < \lambda - 1 \)) over \( \mathbb{Z} \), which lacked torsion in the input data, now exhibit torsion in the latent space. Consequently, the encoder cannot induce an isomorphism on the homology groups over \( \mathbb{Z} \) for all dimensions \( p \). 
	\noindent By (\cite{hatcher}, Corollary 3A.6. (a)), the Betti numbers of the homology groups over \( \mathbb{Q} \) correspond to the free Betti numbers over \( \mathbb{Z} \), which are unaffected by torsion. According to (\cite{hatcher}, Corollary 3A.7. (b)), a map \( f: X \to Y \) induces an isomorphism on homology groups with \( \mathbb{Z} \) coefficients if and only if it induces isomorphisms over \( \mathbb{Q} \) and \( \mathbb{Z}/p\mathbb{Z} \) coefficients for all prime numbers \( p \). 
	\noindent Since torsion appears in persistent homology of the latent space for homology dimension $p< \lambda -1$, there must exist some prime \( q' \) such that the induced map on (relative) homology groups over \( \mathbb{Z}/q'\mathbb{Z} \) is not an isomorphism. The result follows.
\end{proof}

\noindent The representability of torsion in the latent space also leads to torsion-induced changes in the persistence diagrams, respectively barcode diagrams.

\begin{proposition}
	\label{prop:propo1}
	We assume that an input space \( X \subset \mathbb{R}^d \) is transformed into a latent space \( L \subset \mathbb{R}^{\lambda} \) with \( \lambda \ll d \) via an encoder \( f_{\mathrm{enc}} \). Assume that the Rips--filtration on the input space \( X \subset \mathbb{R}^d \) exhibits \( q \)-torsion in its \( p \)-dimensional persistent homology, for \( p \geq \lambda - 1 \ll d \). Then the representability of torsion in the latent space implies that there exists at least one prime \( q' \) such that the barcode diagram of the latent space over \( \mathbb{Z}/q'\mathbb{Z} \) exhibits at least one additional feature in comparison to the input space for at least two homology dimensions \( p < \lambda - 1 \).
\end{proposition}
\begin{proof}
The existence of \( q \)-torsion for homology dimension \( p \) implies that there is at least one torsional cycle forming the basis of the torsion subgroup of the (relative) homology group in \( p \). All elements of the torsion subgroup lie in the kernel of the boundary operator when computed over the finite field \( \mathbb{Z}/q\mathbb{Z} \). Consequently, cycles whose torsion coefficient times their representative loop bounds a simplex in the integral chain complex are incorrectly identified as nontrivial homology classes over \( \mathbb{Z}/q\mathbb{Z} \). Therefore, there exists at least one additional nontrivial cycle over \( \mathbb{Z}/q\mathbb{Z} \) (torsional feature) compared to other finite fields \( \mathbb{Z}/r\mathbb{Z} \) for \( r \) prime. Its lifetime corresponds to a point in the \( p \)-dimensional persistence diagram and a bar in the \( p \)-dimensional barcode diagram. Now, as the Euler characteristic (\cite{hatcher}, Theorem 2.44)
\[
\chi(X) = \sum_{p} (-1)^p \operatorname{rank} H_p(X)
\]
is determined only by the homotopy type of a finite CW complex \( X \) and thus independent of the choice of coefficient ring, any change in the count of torsional features in one dimension must be compensated by a corresponding change in another dimension to maintain the Euler characteristic. This compensation manifests as the appearance of at least one additional feature in a different homology dimension \( p' \neq p \) in the persistent homology of the latent space. Specifically, we conclude that there must exist another homology dimension \( p' < \lambda - 1 \) where an additional feature appears in the barcode diagram over \( \mathbb{Z}/q'\mathbb{Z} \). 
\noindent This establishes the claim.
\end{proof}
\noindent We use the notion of persistence entropy to quantify the change in the persistence diagram or barcode diagram of the latent space described in Proposition~\ref{prop:propo1}, which leads to a difference between the persistence diagrams of the input and latent space for homology dimension $p < \lambda - 1$ over a finite field $\mathbb{Z}/q\mathbb{Z}$ solely induced by torsion.
\begin{definition}[Persistence Entropy, \cite{rucco:hal-01260143, persbar}]
Given a filtration $F$ of a topological space and the corresponding persistence diagram \( \operatorname{dgm}(F) = \{a_i = (b_i, d_i) \mid 1 \leq i \leq n\} \) (where \( b_i < d_i \) for all \( i \)), we define the set of barcode lengths \( L = \{\ell_i = d_i - b_i \mid 1 \leq i \leq n\} \). The persistent entropy \( E(F) \) is calculated as the Shannon entropy of the normalized persistence bar lengths of a given filtration:
\begin{align}
	E(F) &= - \sum_{i=1}^n p_i \log(p_i), \quad \text{where} \label{eq:entropy}\\
	p_i &= \frac{\ell_i}{S_L}, \quad \ell_i = d_i - b_i, \quad
	S_L = \sum_{i=1}^n \ell_i \notag
\end{align}

\end{definition}
\noindent The persistent entropy reaches its maximum $\mathrm{log}(n)$ when all $n$ bars in the barcode are of equal length. We consider a Rips filtration on a finite metric space \( X \subset \mathbb{R}^d \), where \( r \) represents the minimum Euclidean distance between two points in \( X \) and \( 2T \) denotes the radius of \( X \).
 
\noindent In \cite{atienza2019persistent}, minimum and maximum values for the persistence entropy of a barcode diagram based on this Rips filtration were established, depending on the ratio \( \alpha:= r/T \). For our subsequent discussion, we require the following result, which arises from the minimization of the concave persistence entropy.
\begin{lemma}[Maximum Number of Features; \cite{atienza2019persistent}, Theorem 13]
	\label{lemm:8}
	Let $X \subset \mathbb{R}^d$ be a finite metric space with minimum Euclidean distance $r$ between two points and radius $2T$. Then the maximum number of bars for a barcode computed via the Rips filtration on $X$ with $n$ bars is given by:
	\begin{equation}
		\label{eq:Q}
		Q=n\frac{\alpha \mathrm{log}(\frac{1}{\alpha})-\alpha (1-\alpha)}{(1-\alpha)^{2}}.
	\end{equation}
\end{lemma}
\noindent Furthermore, we require the following result from \cite{atienza2019persistent}, which states that the persistence entropy for a given Rips filtration decreases when one or more of the bars are replaced by a bar with the length specified in the theorem, which depends on the persistence entropy of the remaining bars.
\begin{theorem}[\cite{atienza2019persistent}, Theorem 8]
	\label{th:8}
	Let \( F \) be a filtration with the persistence diagram \( \operatorname{dgm}(F) = \{ (x_i, y_i) \mid 1 \leq i \leq n \} \) and bar lengths \( L = \{ \ell_i = y_i - x_i \} \). For any \( i \in \{1, \ldots, n\} \), define 
	
	\[
	L' = \{ \ell_1', \ldots, \ell_i', \ell_{i+1}, \ldots, \ell_n \} 
	\]
	\noindent with \( \ell_j' = \frac{P_i}{e^{E(R_i)}} \) for \( 1 \leq j \leq i \), where \( R_i = \{ \ell_{i+1}, \ldots, \ell_n \} \) and \( P_i = \sum_{j = i+1}^{n} \ell_j \). Then the persistence entropy satisfies 
	
	\[
	E(L) \leq E(L').
	\]
\end{theorem}
\noindent Based on Theorem~\ref{th:8} and Eq.~\eqref{eq:Q} in Lemma~\ref{lemm:8}, \cite{atienza2019persistent} derive an algorithm that extracts the \( Q \) topological features from a barcode using the following quotient:
\begin{equation}
C = \frac{S_{L_{i-1}'} }{S_{L_i'}} = 
\frac{
	P_{i-1}' + (i - 1) \frac{P_{i-1}'}{e^{E(R_{i-1}')}} 
}{
	P_i' + i \frac{P_i'}{e^{E(R_i')}} .
}
\end{equation}
\noindent This quotient successively modifies the lengths of the bars based on the persistence entropy of the remaining bars. Bars \( [b_i, d_i) \) are classified as noise as long as \( C \geq 1 \) and \( Q < i \). For our further discussion, we derive the following corollary:
\begin{corollary}
	\label{cor:bla}
If an additional bar is added to a barcode with \( n \) bars, a total length \( P \) of the bars, and a persistence entropy \( e^R \), and it is known that this bar represents a topological feature, then its length must satisfy \( \ell > \frac{P}{e^R} \).
\end{corollary} 
\noindent From the above results, we can now estimate the bottleneck distance that arises solely from the representability of torsion in the latent space between the barcodes of the input and latent space.

\begin{proposition}
	Let \( X \subset \mathbb{R}^d \) be the input space for an autoencoder with a latent space \( L \subset \mathbb{R}^{\lambda} \) with \( \lambda \ll d \). If the persistent homology of \( X \) exhibits \( q \)-torsion in homology dimension \( p \geq \lambda - 1 \) and this torsion is represented in a lower dimensional homology group $p<\lambda-1$ in the latent space, this leads to the presence of an additional bar in the $p$-dimensional barcode of the latent space over a finite field \( \mathbb{Z}/q'\mathbb{Z} \) with a minimum length of 
	
	\[
	\frac{P}{e^R},
	\]
	\noindent where \( P \) is the total length of the bars in the barcode of the input space and \( e^R \) is the persistence entropy of the input space. The resulting minimal bottleneck distance between the input space and the latent space for at least one homology dimension is 
	
	\[
	\min \frac{d_i - b_i}{2},
	\]
	\noindent where \( [b_i, d_i) \) are the points in the persistence diagram of the input space.
\end{proposition}
\begin{proof}
A \( q \)-torsion cycle that is identified as a cycle over \( \mathbb{Z} / q \mathbb{Z} \) constitutes a topological feature over this coefficient field. Corollary~\ref{cor:bla} implies that any bar representing this torsion feature must have a length \( > \frac{P}{e^R} \), with the input space being the original space. Adding an additional bar to the barcode of the latent space results in at least one existing bar becoming unmatched and therefore being matched to the diagonal. If the newly added bar is closer to an already existing bar, it will be matched to this bar rather than the diagonal. Else, the resulting bottleneck distance will be $\geq P/2e^R$. Consequently, the shortest possible unmatched bar in the latent space is the shortest bar in the input space, leading to the result.
\end{proof}
\noindent From the above discussion, it follows that the representability of high-dimensional torsional structures as torsional structures in lower-dimensional homology groups in the latent space leads to a change in persistence entropy during the encoding process.

\section{Torsion Reconstruction in the Decoder}
If torsional structures in the input data can be transformed into the latent space via the encoder, such that they manifest in the homology groups of a filtration of the latent space, important follow-up questions arise. 
\begin{enumerate}[label=\Roman*.]
	\setcounter{enumi}{3}  % Starts counting from III
	\item Given that the autoencoder is trained using a loss function based on reconstruction in the output space, can the decoder accurately reconstruct torsion?
	\item If so, under which mathematical conditions?
\end{enumerate}

\subsection{Linear Decoders}
\label{sec:lindec}
For decoders with \textit{linear activation functions}, which are indeed used in research, especially for low-dimensional input data, statements about the reconstructability of torsion can be derived from dimension-based arguments.\\

\noindent Already in dimension 3, one can construct intuitive and instructive torsional examples, such as the triple and double loop point clouds depicted in Figure~\ref{fig:tors3loop} and Figure~\ref{fig:4a}. Yet, torsion seems to be rare in $\mathbb{R}^3$, at least in randomly generated datasets \cite{obayashi2023field}. 
%\noindent We consider an autoencoder $\mathcal{A}_{3D}$ for 3D input data with the following cominatorial architecture: $3 \to 2 \to 3 $. 
\begin{proposition}
Under Assumption~\ref{ass:euclid}, a decoder with a linear activation function cannot map a two-dimensional latent space \( L \) to three-dimensional input data in a way that preserves torsional structures in the reconstructed data.
\end{proposition}

\begin{proof}
By Corollary 1.7 to Theorem~\ref{th:obay} in \cite{obayashi2023field}, persistence diagrams $D_p(K; k)$ for a filtered complex $K$ with coefficients in a field $k$ are independent of the choice of the field $k$ for homology dimension $p=0$. By the subsequent Corollary 1.8 in \cite{obayashi2023field}, for a filtered simplicial complex $K$ embedded in an Euclidean space $\mathbb{R}^M$, the persistent diagrams of homology dimension $p=M-1$ are independent of the choice of coefficient field $k$. Thus, for the filtered simplicial complex $K_L \subset \mathbb{R}^2$ constructed on the elements of the latent space $L$, all possible persistence diagrams $D_0(K_L; k )$ and $D_1(K_L; k)$ are independent of the choice of the coefficient field $k$ and the respective integral homology groups are torsionfree.
By the rank-nullity theorem, a linear decoder maps this two-dimensional latent space to a two-dimensional output space, which by Assumption~\ref{ass:euclid} is $X' \subset \mathbb{R}^2$. Thus, an autoencoder with a two-dimensional latent space and a linear decoder cannot reconstruct torsion from three dimensional input data. 	
\end{proof}

\begin{corollary}
	Let $X \subset \mathbb{R}^d$ be the input data to an autoencoder, exhibiting a nontrivial torsion subgroup in its $p$-dimensional homology group with $p>\lambda-1$. Let the encoder transform $X$ to a $\lambda<<d$ dimensional latent space $L$, which is a subset of a Euclidean space $\mathbb{R}^{\lambda}$ by Assumption~\ref{ass:euclid}. A decoder with a linear activation architecture cannot fully reconstruct the torsion in $X$. 
\end{corollary}
\begin{proof}
	Torsion subgroups of the integral homology groups can potentially appear up to homology dimension $p=d-1$ by definition. By the rank-nullity theorem, the output $X'$ reconstructed from the latent space $L$ with a linear decoder has dimension $\leq \lambda$ and is embedded in a Euclidean space $\mathbb{R}^{\lambda}$ by Assumption~\ref{ass:euclid}. 

\end{proof}

\begin{corollary}
	\label{cor:1}
	Let	$f_{\mathrm{dec}}:= f_\lambda \circ f_{\lambda+(d-n_{l_1})}\circ ... \circ f_{\lambda+(d-n_{l_N})}$ denote the decoder of an autoencoder with nonlinear activation architecture and a $\lambda$-dimensional latent space, $\lambda < d$, consisting of $N$ layers with natural numbers $\lambda \leq n_{l_i} \leq d-1 $, $0 \leq i \leq N$. Let the subscript denote the dimension of the output space of the respective layer under consideration. Then torsion within homology groups of homology dimension $p < d-n_{l_{i+1}} -1 $ is reconstructible in one layer transition from $f_{\lambda+(d-n_{l_i})}$ to $f_{\lambda+(d-n_{l_{i+1}})}$ under the condition that the decoder architecture is able to reconstruct torsion.
\end{corollary}
\noindent Corollary~\ref{cor:1} immediately raises the following questions:
\begin{enumerate}[label=\Roman*.]
	\setcounter{enumi}{5}  % Starts counting from III
	\item Is torsion reconstructible with a nonlinear decoder?	
\end{enumerate}

\subsection{Nonlinear Decoders}
We investigate Question (VI.) experimentally with three different autoencoder architectures: A vanilla architecture, a topological autoencoder and an RTD autoencoder. The latter are described in detail in Section~\ref{app:AE} of the appendix. We used an NVIDIA GeForce RTX 4080 GPU for the experiments in Section~\ref{sec:doubtriploop}. For the experiments in Section~\ref{sec:highdpc}, we relied on a high-performance node of the IAV Kubeflow cluster equipped with dual AMD EPYC 7742 CPUs (128 cores, 256 threads) and 1~TiB of RAM. To investigate whether torsion was present in the data—and, if so, for which prime divisor—we employed the algorithm introduced in \cite{obayashi2023field}, which outputs a tuple of the form $\mathbf{(q^k, \textit{index of the first simplex causing torsion})}$.
 
\subsubsection{Prime Divisors of Torsion Coefficients}
As a first small experiment to gain an intuition for the sensitivity of torsion phenomena to slight perturbations of the data, we added Gaussian noise to the points of the triple loop and subsequently investigated whether and for which prime divisor $q$-torsion was present in the data. An excerpt of the results are shown in Figure~\ref{fig:comp}.
\begin{figure}[htbp]
	\centering
	\begin{subcaptionbox}{$2$-torsion\label{subfiga}}[0.48\columnwidth]
		{\includegraphics[width=\linewidth]{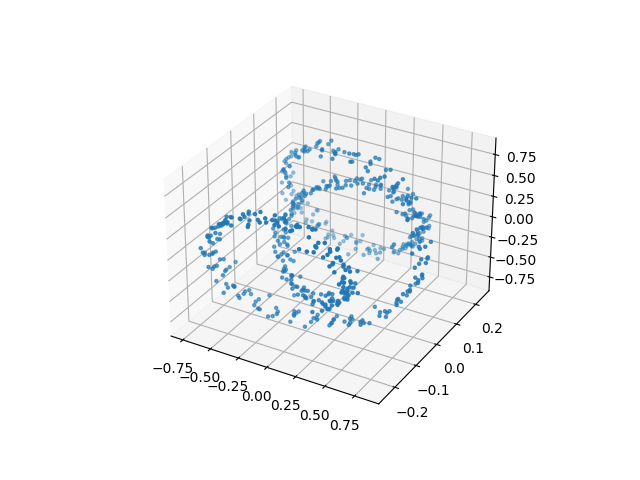}}
	\end{subcaptionbox}
	\hfill
	\begin{subcaptionbox}{$3$-torsion\label{subfigb}}[0.48\columnwidth]
		{\includegraphics[width=\linewidth]{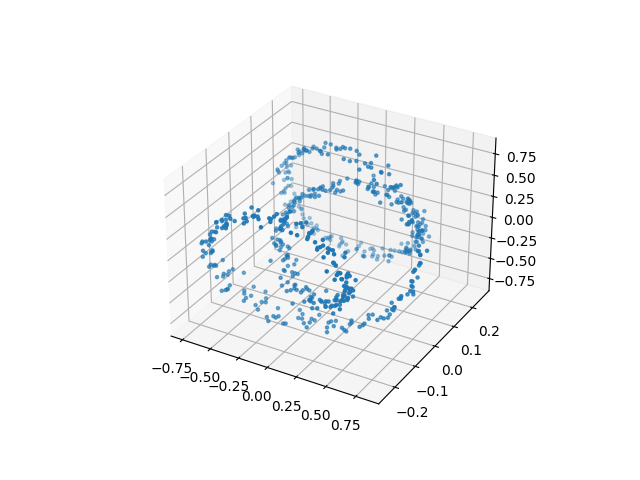}}
	\end{subcaptionbox}
	\caption{Triple-loop with Gaussian noise and standard deviation 0.02.}
	\label{fig:comp}
\end{figure}

\noindent The two point clouds exhibit a MSE of 0.000816 and albeit they look almost identical, in Figure~\ref{subfiga} the point cloud has $2$-torsion and in Figure~\ref{subfigb} $3$-torsion was identified by the algorithm in \cite{obayashi2023field}. Hence, in this case, a shift of the points in the noisy triple loop that is negligible in terms of MSE leads to a misinterpretation of the structure as a noisy double loop, as indicated by its prime divisor.

\subsubsection{Reconstruction of Double and Triple Loops}
\label{sec:doubtriploop}
We first investigate the reconstructability of torsion in autoencoders using the double and triple-loop point clouds as examples. We consistently use the following layer architecture:

\[
3 \rightarrow 32 \rightarrow 32 \rightarrow 2 \rightarrow 32 \rightarrow 32 \rightarrow 3
\]

\noindent i.e., we employ a two-dimensional latent space. We train using batches of the input data and evaluate using the entire data set. ReLU nonlinearities were used. We employed batch normalization for better training stability. 
\\
\noindent As we use only one data set for training and testing respectively, we are clearly in the domain of overfitting. Thus we fixed for all three autoencoder models (namely the vanilla, the TopoAE and the RTD autoencoder) the learning rate to $10^{-3}$ and the batch size to $32$. For the TopoAE and the RTD autoencoder, we employed hyperparameter optimization for the parameter(s) regularizing the \enquote{TDA} loss term (namely $\eta$ in Eq.~\eqref{eq:topoae} and $\chi$ in Eq.~\eqref{eq:rtdloss}).\\
For hyperparameter tuning we applied Bayesian optimization using the scikit-optimize library \cite{scikit-optimize}. The best value for the regularization parameter~$\eta$ was selected based on the total loss (reconstruction loss plus regularization term) computed on a single training batch. We chose this single-batch approach as the values of $\eta$ across the individual batches are close to each other, and since we only have one dataset with few points, we had to strike a balance between hyperparameter optimization and overfitting.\\
\noindent The loss was minimized by fitting a Gaussian process to the observed values and iteratively selecting promising~$\eta$ and $\chi$ values using an acquisition function. This was done for 20 calls per method. The best-performing~$\eta$ respectively $\chi$ was then used to train the final model on the full training set.\\
\noindent Each model was trained for $100$ epochs. We repeated this training cycle for 40 times. All neural network architectures were fit using Adam. 

\paragraph{Vanilla Autoencoder:}

In Figure~\ref{fig:vanillaAE}, for the double loop, the MSE loss versus the 40 runs is displayed and the torsional runs in the output are marked. In Figure~\ref{fig:vanillaloop} the three outputs with the smallest MSE losses and the three torsional outputs with the smallest MSE losses are displayed. The results are summarized in Table~\ref{tab:vanilla}. The comparison shows that the reconstructability of torsion is not coupled to minima in the MSE loss. As a consequence, the MSE loss is not informative with regard to the reconstructability of torsion. 

\begin{figure}[ht]
	\centering
	\includegraphics[width=\columnwidth]{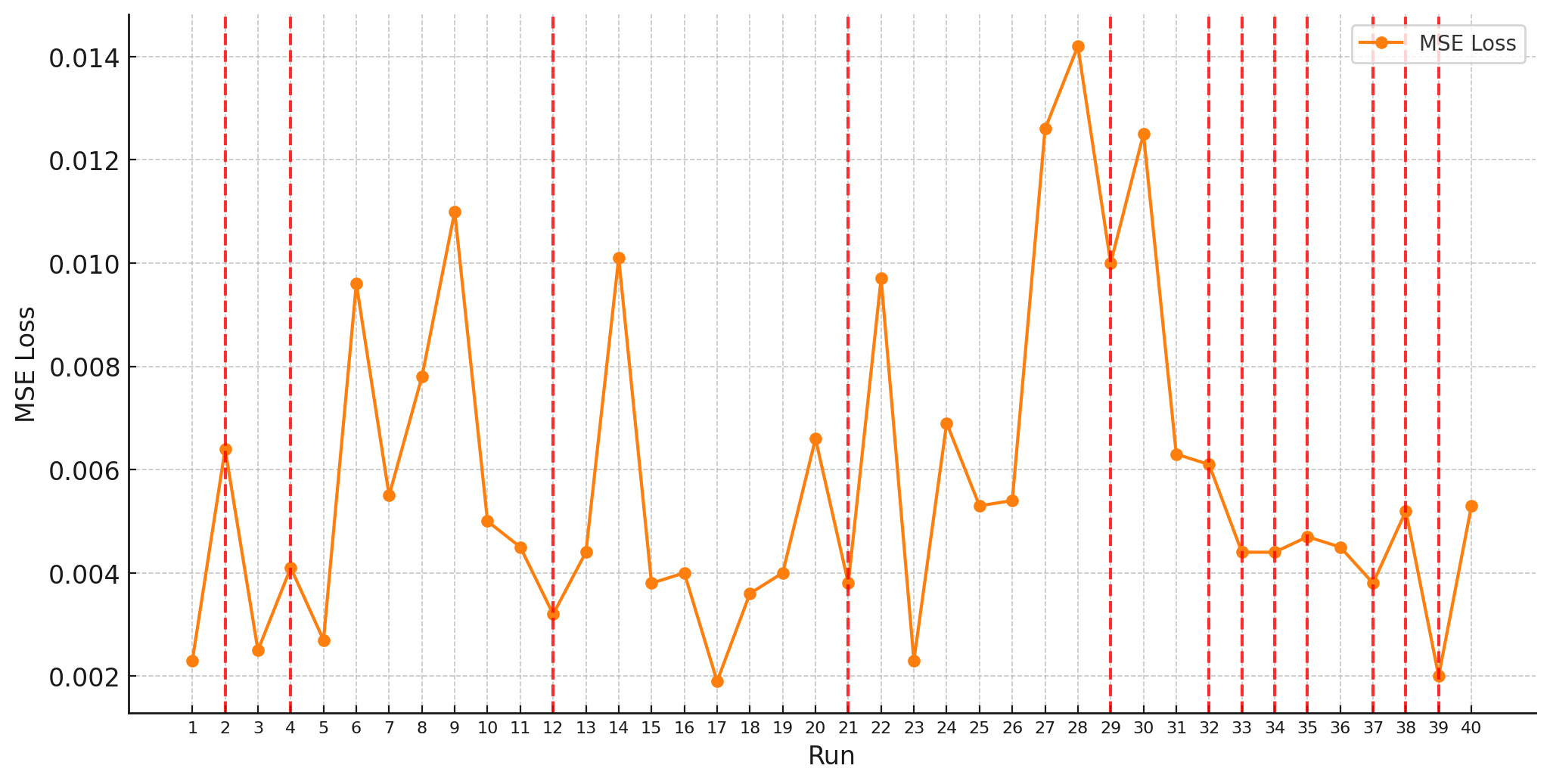}
	\caption{MSE loss over 40 training cycles using the vanilla autoencoder on the double loop input. Training cycles that produced torsional outputs are indicated by red vertical lines.}
	\label{fig:vanillaAE}
\end{figure}

\begin{table}[ht]
	\centering
	\begin{tabular}{|c|c|c|}
		\hline
		\textbf{Run} & \cellcolor{orange!40}\textbf{MSE Loss} & \textbf{Torsion} \\
		\hline
		\textbf{17} & 0.0019 & No Torsion \\
		\textbf{39} & 0.0020 & \textbf{Torsion: (2,8701)} \\
		\textbf{1} & 0.0023 & No Torsion \\
		\hline
	\end{tabular}
	\caption{The three runs with the lowest MSE loss for the vanilla autoencoder trained on the double loop input.}
	\label{tab:vanilla}
\end{table}

\noindent For the triple loop, torsion could not be reconstructed. The three outputs with the lowest MSE loss are displayed in Figure~\ref{fig:mod3fig}. 

\paragraph{Topological Autoencoder:}
For the training of the topological autoencoder, the regularization parameter $\eta$ was selected via hyperparameter optimization from the log-uniform range $[0.1, 3]$.\\
\noindent Figure~\ref{fig:torsiontopoAE} displays the MSE, the topological (\enquote{Topo}) and total loss for 40 training cycles for the double loop. Table~\ref{tab:topo} summarizes the results for the three outputs with the least MSE, Topo and total loss respectively. Visualizations of the outputs can be found in Figure~\ref{fig:topoae_visuals}. Again, the outputs with the least losses did not coincide with the torsional ones. 

\begin{figure}[ht]
	\centering
	\includegraphics[width=\columnwidth]{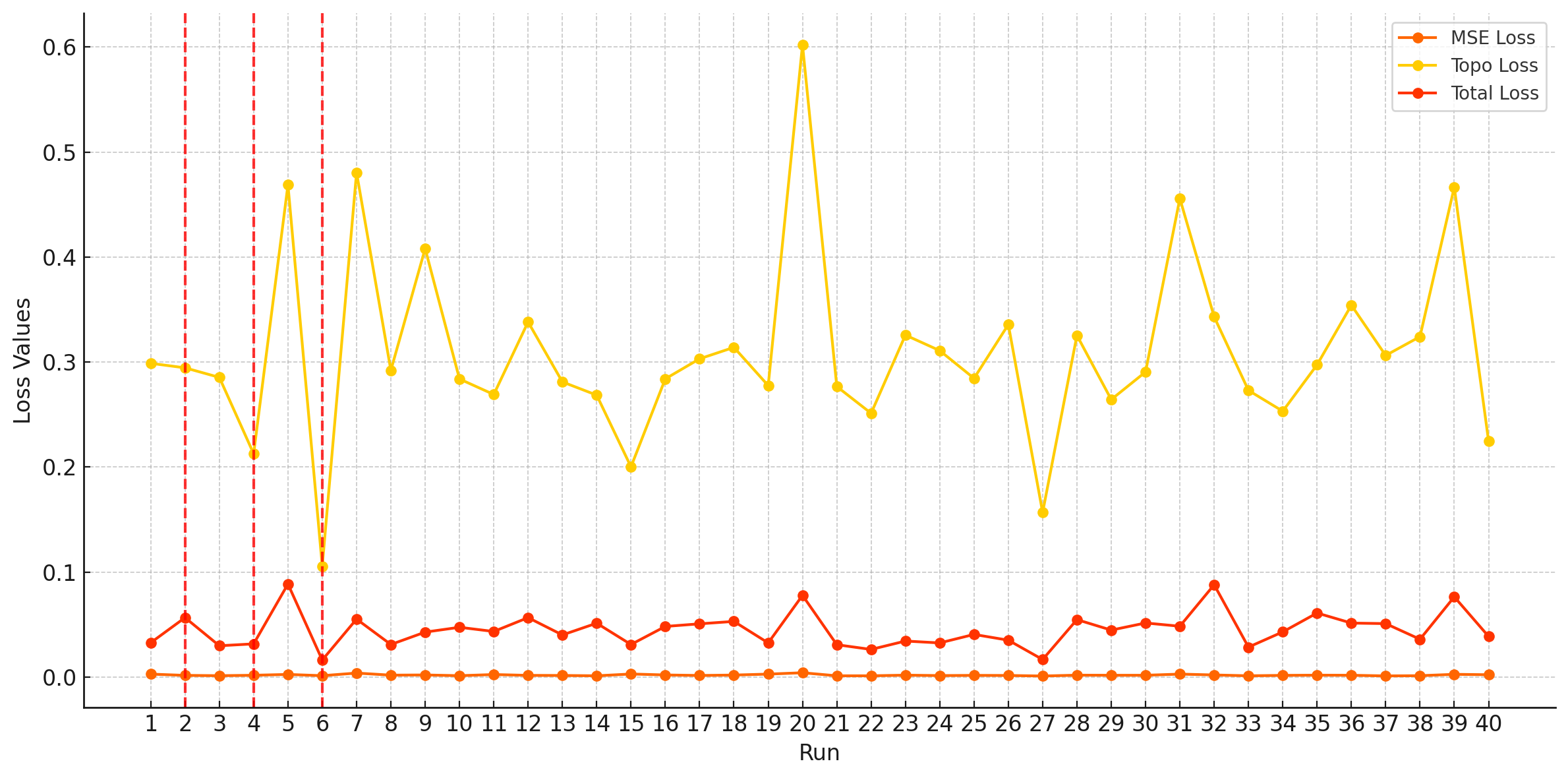}
	\caption{MSE loss, topological loss, and total loss over 40 training cycles of the topological autoencoder trained on the double loop input. Training cycles that resulted in torsion are marked by red vertical lines.}
	\label{fig:torsiontopoAE}
\end{figure}

\begin{table}[ht]
	\centering
	\resizebox{\columnwidth}{!}{%
		\begin{tabular}{|c|c|c|c|c|}
			\hline
			\textbf{Run} & \cellcolor{orange!40}\textbf{MSE Loss} & \cellcolor{yellow!40}\textbf{Topo Loss} & \cellcolor{red!40}\textbf{Total Loss} & \textbf{Torsion} \\
			\hline
			\textbf{27} & 0.0010 & - & - & No Torsion \\
			\textbf{37} & 0.0011 & - & - & No Torsion \\
			\textbf{14} & 0.0012 & - & - & No Torsion \\
			\hline
			\hline
			\textbf{6}  & - & 0.1057 & - & \textbf{Torsion: (2, 8833)} \\
			\textbf{27} & - & 0.1568 & - & No Torsion \\
			\textbf{15} & - & 0.2002 & - & No Torsion \\
			\hline
			\hline
			\textbf{6}  & - & - & 0.0162 & \textbf{Torsion: (2, 8833)} \\
			\textbf{27} & - & - & 0.0167 & No Torsion \\
			\textbf{22} & - & - & 0.0263 & No Torsion \\
			\hline
		\end{tabular}
	}
	\caption{Three runs with the lowest MSE, Topo, and total loss values for the double loop as input. Header colors match the respective loss colors in Figure~\ref{fig:torsiontopoAE}. Torsional runs are highlighted in bold.}
	\label{tab:topo}
\end{table}

\noindent For the triple loop, torsion was not reconstructable. The results are summarized in Table~\ref{tab:topomod3}. The outputs with the lowest respective losses are visualized in Figure~\ref{fig:topomod3_visuals}.

\begin{table}[ht]
	\centering
	\resizebox{\columnwidth}{!}{%
		\begin{tabular}{|c|c|c|c|c|}
			\hline
			\textbf{Run} & \cellcolor{orange!40}\textbf{MSE Loss} & \cellcolor{yellow!40}\textbf{Topo Loss} & \cellcolor{red!40}\textbf{Total Loss} & \textbf{Torsion} \\
			\hline
			\textbf{13} & 0.0021 & - & - & No Torsion \\
			\textbf{33} & 0.0025 & - & - & No Torsion \\
			\textbf{15} & 0.0026 & - & - & No Torsion \\
			\hline
			\hline
			\textbf{13} & - & 0.2695 & - & No Torsion \\
			\textbf{33} & - & 0.2698 & - & No Torsion \\
			\textbf{21} & - & 0.2840 & - & No Torsion \\
			\hline
			\hline
			\textbf{13} & - & - & 0.0290 & No Torsion \\
			\textbf{33} & - & - & 0.0295 & No Torsion \\
			\textbf{28} & - & - & 0.0318 & No Torsion \\
			\hline
		\end{tabular}
	}
	\caption{Three runs with the lowest MSE, Topo, and total loss values for the triple loop as input. All runs are non-torsional.}
	\label{tab:topomod3}
\end{table}

\paragraph{RTD Autoencoder:}

For the RTD autoencoder, the 100 training epochs where split in three phases as described in \cite{trofimov2023learningtopologypreservingdatarepresentations}. Initially, the autoencoder was trained for 10 epochs using only the reconstruction loss with a learning rate of $10^{-4}$, and then the training was continued with the RTD loss. Epochs 11--30 used a learning rate of $10^{-2}$, epochs 31--50 used $10^{-3}$, and all subsequent epochs used $10^{-4}$. The regularization parameter~$\chi$ was selected via hyperparameter optimization from the log-uniform range $[10^{-6}, 10^{3}]$.\\
The results for the double loop point cloud are displayed in Table~\ref{tab:rtdresults} and Figure~\ref{fig:torsionrtd}. Interestingly, the three runs with the lowest MSE loss now exhibited torsion while this was not the case for the respective three runs with the lowest RTD loss and total loss. Apparent is also that, compared to the results displayed for the vanilla autoencoder in Figure~\ref{fig:vanillaAE} and the topological autoencoder in Figure~\ref{fig:torsiontopoAE}, with the RTD autoencoder 25 of 40 runs exhibited 2-torsional outputs. However, again, a low RTD loss is not informative with regard to the reconstructability of torsional structures, as becomes apparent in Table~\ref{tab:rtdresults}: Only one out of the three runs with the lowest RTD loss exhibit torsion. The $3$D plots corresponding to Table~\ref{tab:rtdresults} are depicted in Figure~\ref{fig:rtd_table_visuals}.

\begin{figure}[H]
	\centering
	\includegraphics[width=\columnwidth]{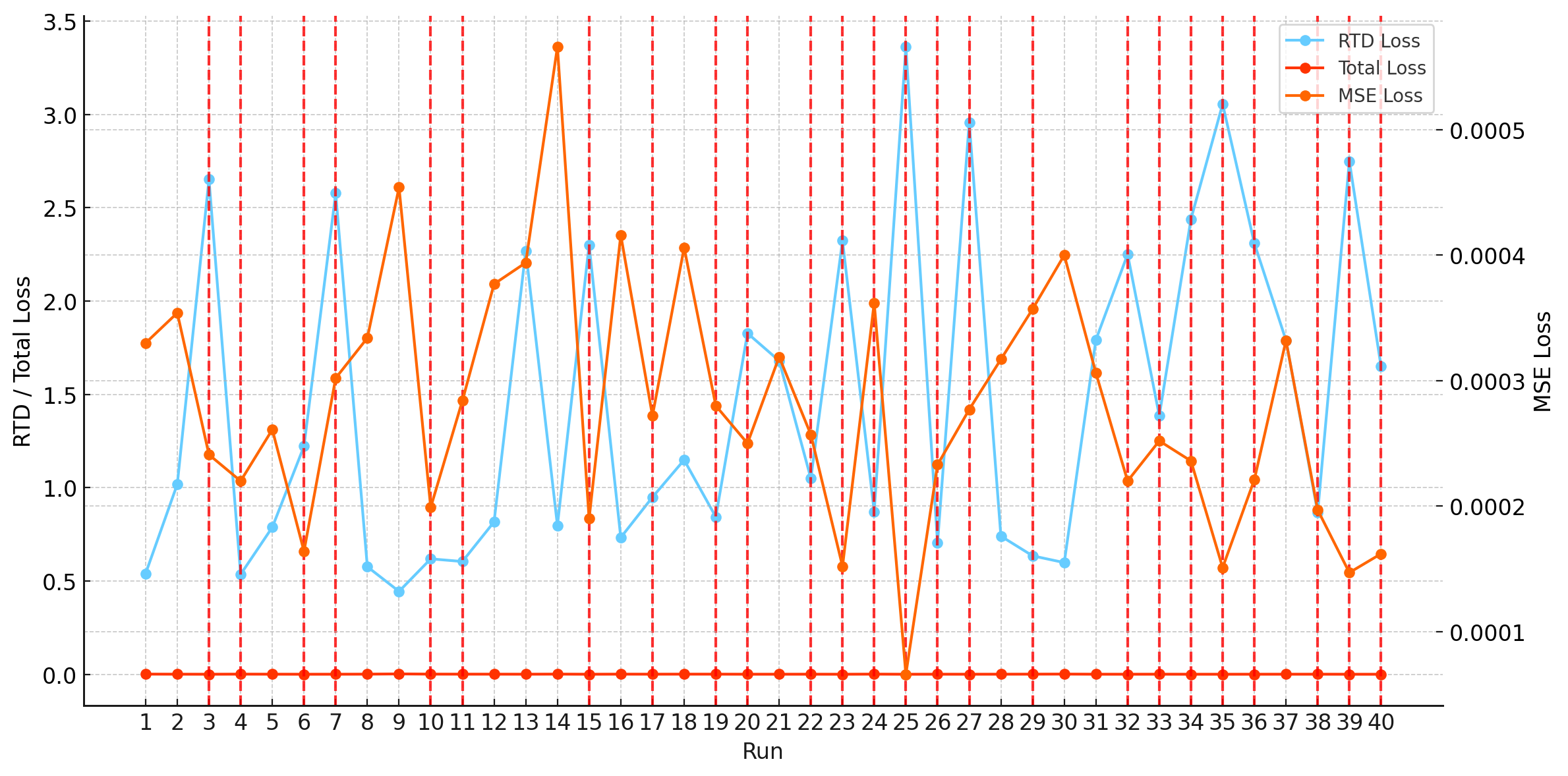}
	\caption{MSE Loss, RTD loss and total loss for 40 training cycles with the RTD autoencoder. Torsional training cycles are marked with red vertical lines. To be able to discern the total loss, which lies quite close to the MSE loss, a double labeled $y$-axis was used. }
	\label{fig:torsionrtd}
\end{figure}

\begin{table}[H]
	\centering
	\resizebox{\columnwidth}{!}{%
		\begin{tabular}{|c|c|c|c|c|}
			\hline
			\textbf{Run} & \cellcolor{orange!40}\textbf{MSE Loss} & \cellcolor{cyan!30}\textbf{RTD Loss} & \cellcolor{red!40}\textbf{Total Loss} & \textbf{Torsion} \\
			\hline
			\textbf{25} & 0.000066 & -        & -         & \textbf{Torsion: (2, 6775)} \\
			\textbf{39} & 0.000147 & -        & -         & \textbf{Torsion: (2, 6787)} \\
			\textbf{35} & 0.000151 & -        & -         & \textbf{Torsion: (2, 6783)} \\
			\hline
			\hline
			\textbf{9}  & -        & 0.444102 & -         & No Torsion \\
			\textbf{4}  & -        & 0.535345 & -         & \textbf{Torsion: (2, 6619)} \\
			\textbf{1}  & -        & 0.539877 & -         & No Torsion \\
			\hline
			\hline
			\textbf{4}  & 0.000220 & 0.535345 & 0.000760  & \textbf{Torsion: (2, 6619)} \\
			\textbf{10} & 0.000199 & 0.618562 & 0.000823  & \textbf{Torsion: (2, 6683)} \\
			\textbf{9}  & 0.000454 & 0.444102 & 0.001923  & No Torsion \\
			\hline
		\end{tabular}
	}
	\caption{Three runs with lowest MSE, RTD and total loss, with header colors matching the respective loss colors in Figure~\ref{fig:torsionrtd}. Torsional runs are highlighted in bold.}
	\label{tab:rtdresults}
\end{table}

\noindent For the triple loop point cloud, also the RTD autoencoder could not reconstruct torsion. The results are displayed in Table~\ref{tab:mod33loop} and the corresponding reconstructed point clouds are shown in Figure~\ref{fig:bestrtdmod3}.

\begin{table}[H]
	\centering
	\resizebox{\columnwidth}{!}{%
		\begin{tabular}{|c|c|c|c|c|}
			\hline
			\textbf{Run} & \cellcolor{orange!40}\textbf{MSE Loss} & \cellcolor{cyan!30}\textbf{RTD Loss} & \cellcolor{red!40}\textbf{Total Loss} & \textbf{Torsion} \\
			\hline
			\textbf{28} & 0.002108 & -        & -         & No Torsion \\
			\textbf{11} & 0.002630 & -        & -         & No Torsion \\
			\textbf{12} & 0.002725 & -        & -         & No Torsion \\
			\hline
			\hline
			\textbf{10} & -        & 1.127720 & -         & No Torsion \\
			\textbf{7}  & -        & 1.202779 & -         & No Torsion \\
			\textbf{30} & -        & 1.212620 & -         & No Torsion \\
			\hline
			\hline
			\textbf{8}  & 0.003103 & 1.443048 & 0.004113  & No Torsion \\
			\textbf{7}  & 0.003749 & 1.202779 & 0.006140  & No Torsion \\
			\textbf{10} & 0.005802 & 1.127720 & 0.007577  & No Torsion \\
			\hline
		\end{tabular}
	}
	\caption{Three runs with the lowest MSE, RTD, and total loss values. All runs are non-torsional.}
	\label{tab:mod33loop}
\end{table}

\subsubsection{Reconstruction of High-Dimensional Point Clouds}
\label{sec:highdpc}

To investigate the reconstructability of torsion in higher-dimensional point clouds using autoencoders, we generate high-dimensional random point clouds and analyze them for torsion using the algorithm developed by \cite{obayashi2023field}. The point clouds identified as torsion-prone are then used as an input to train autoencoders, and we examine both the latent and output spaces for the presence of torsion. We employ hyperparameter optimization using \textit{Optuna} \cite{optuna}. The hyperparameters are deployed in Table~\ref{tab:combined-hyperparams-10d} for the $10$-dimensional input data and in Table~\ref{tab:combined-hyperparams-13d} for the $13$-dimensional input data.\\

\noindent In addition, for both the RTD and topological autoencoder, the hyperparameter ($\eta$ in Eq.~\eqref{eq:topoae} and $\chi$ in Eq.~\eqref{eq:rtdloss}) controlling the trade-off between reconstruction and geometric or topological loss was selected via random search on five subsamples, with candidate values drawn from a logarithmic scale $[10^{-6}, 10^3]$ for RTD and $[10^{-6}, 10^1]$ for the topological variant; the final hyperparameter was set as the average of the best-performing values across subsamples.
\newline Across all experiments, the latent space consistently showed no signs of torsion, indicating that torsion can reappear in the output even when it is absent from the latent space.

\paragraph{Vanilla Autoencoder}

\paragraph{$d=10$:}
Out of 1000 randomly generated point clouds in $\mathbb{R}^{10}$, 18 were identified as torsional. These were used as an input to a vanilla autoenocder with the hyperparameters deployed in Table~\ref{tab:combined-hyperparams-10d}.  The torsional outputs are shown in Figure~\ref{fig:highd2} and the resutls with regard to the loss are summarized in Table~\ref{tab:tors10d2}. Only one output was identified as torsional and this did not coincide with a minimum MSE loss. 

\begin{figure}[ht]
	\centering
	\includegraphics[width=\columnwidth]{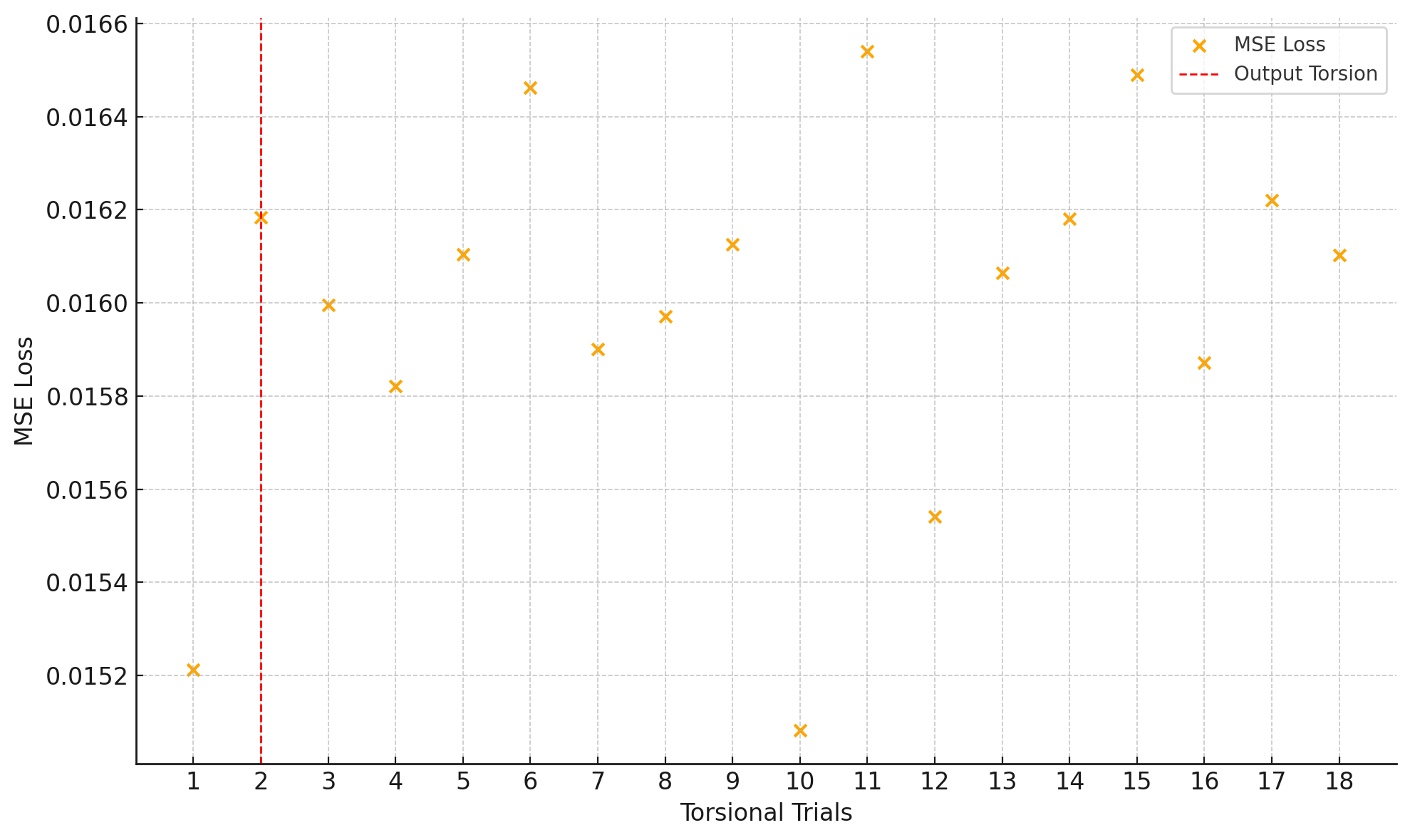}
	\caption{MSE loss for torsional point clouds, processed by a vanilla autoencoder. Torsional outputs are marked with vertical red lines.  }
	\label{fig:highd2}
\end{figure}

\begin{table}[ht]
	\centering
	\begin{tabular}{|c|c|c|}
		\hline
		\textbf{Trial} & \cellcolor{orange!40}\textbf{MSE Loss} & \textbf{Torsion} \\
		\hline
		\textbf{10} & 0.015083 & No Torsion \\
		\textbf{1} & 0.015213 & No Torsion \\
		\textbf{12} & 0.015541 & No Torsion \\
		\hline
	\end{tabular}
	\caption{Three trials with the lowest MSE loss and their torsional status for $10d$ torsional input data.}
	\label{tab:tors10d2}
\end{table}

\paragraph{$d=13$:}
Out of 1000 randomly generated point clouds in $\mathbb{R}^{13}$, 23 were identified as torsional. These were used as an input to a vanilla autoencoder with the hyperparameters deployed in Table~\ref{tab:combined-hyperparams-13d}, with latent dimension $8$. The torsional outputs are shown in Figure~\ref{fig:highd}. The three runs with the lowest MSE are depicted in Table~\ref{tab:tors13d}. Consistently with the 3D experiments, not all of them are torsional.
\begin{figure}[H]
	\centering
	\includegraphics[width=\columnwidth]{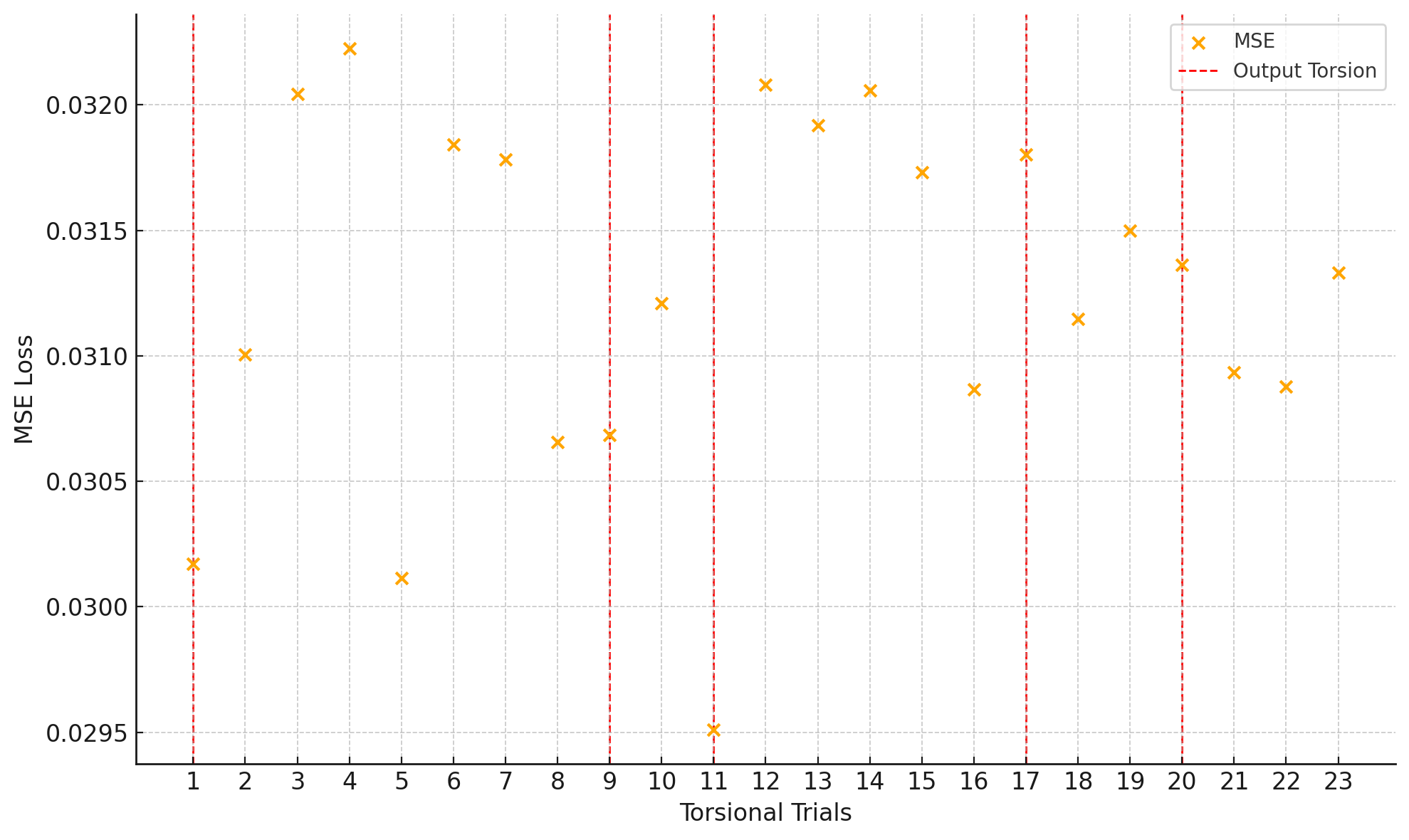}
	\caption{MSE loss for torsional point clouds, processed by a vanilla autoencoder. Torsional outputs are marked with vertical red lines.  }
	\label{fig:highd}
\end{figure}

\begin{table}[H]
	\centering
	\begin{tabular}{|c|c|c|}
		\hline
		\textbf{Trial} & \cellcolor{orange!40}\textbf{MSE Loss} & \textbf{Torsion} \\
		\hline
		\textbf{11} & 0.029509 & \textbf{Torsion: (2, 23146)} \\
		\textbf{5}  & 0.030115 & No Torsion \\
		\textbf{1}  & 0.030172 & \textbf{Torsion: (2, 18861) } \\
		\hline
	\end{tabular}
	\caption{Three trials with the lowest MSE loss and their torsional status for $13d$ torsional input data.}
	\label{tab:tors13d}
\end{table}

\paragraph{Topological Autoencoder}

For the $10$-dimensional torsional inputs, the topological autoencoder reconstructed torsion in 2 of the 18 samples. The results are summarized in Table~\ref{tab:topo10d}.
\noindent For the $13$-dimensional input data, $4$ out of 23 trials were identified as torsional. The results are summarized in Table~\ref{tab:topo13d}. 
For both high-dimensional point clouds, only one torsional output matched a minimum in the topological or total loss term, thus showing consistence with the low-dimensional data investigated in the previous section.

\paragraph{RTD Autoencoder}

For the $10$-dimensional input data, 2 out of 18 trials resulted in torsional outputs. The results are summarized in Table~\ref{tab:rtdloss10d}.
\noindent For the $13$-dimensional inputs, 3 out of 23 outputs exhibited torsion. The results are summarized in Table~\ref{tab:rtdloss13d}. Also for the RTD autoencoder over $13d$ data, the torsional runs did not coincide with those of minimum RTD loss.

\subsection{Summary and Discussion}
All our experiments have consistently shown across various dimensions that neither a low MSE loss, nor the topological loss, nor the RTD loss serves as an indicator of the autoencoder’s ability to reconstruct torsional structures. Except for the double loop example using the RTD loss term, the TDA loss terms did not lead to a significantly higher reconstruction of torsional outputs compared to training based on the MSE reconstruction error. Torsion occurs in homology dimensions $\geq 0$, and while the RTD loss is trained to preserve the persistent homology in homology dimension 1, the topological autoencoder is limited in its application to homology dimension 0.\\
Although both autoencoder architectures come with theoretical guarantees (\cite{clémot2025topologicalautoencodersfastaccurate}, \cite{barannikov2022representationtopologydivergencemethod}) for the case when their loss is zero, a generalizable trained autoencoder does have neither a MSE loss nor a topological loss equal to zero.\\
\noindent Our experiments thus provide an intuition that torsion is a topological feature for which not only the magnitude of the error matters, but also the regions of the data where the error arises. This helps explain why, in some cases, torsionally reconstructed outputs coincided with low losses, while in others they did not. \\
\noindent In this respect, torsion highlights that loss terms based on global persistent homology have their limitations in scenarios where it is crucial to reconstruct local regions of the data accurately, while other regions may exhibit higher loss without significantly affecting the reconstruction of the overall topological structure.\\
\noindent This is not possible with the architecture chosen by both the topological autoencoder and the RTD autoencoder in Eq.\eqref{eq:topoae} and Eq.\eqref{eq:rtdloss}, where a TDA component is added via a trainable hyperparameter, but is always evaluated globally on the data so as to minimize the overall loss. As our examples have shown, this often leads to the loss of torsional structures in particular.

\section{Conclusion}
The effect of transformations in deep neural networks on the topological structures of the input space is still little understood. Using simple mathematical examples and computational experiments, we have demonstrated that torsional structures, in particular, are difficult to reconstruct. Since in an autoencoder the latent space serves as a lower-dimensional representation of high-dimensional input data and is relevant for the classification of such data, we have mathematically investigated the consequences of representing torsion in the latent space. Computationally, we were able to show that torsion can be reconstructed even if it is lost in the lower-dimensional space, but this is not enforced by the MSE reconstruction loss or TDA loss terms.

\paragraph{Future Work}
From a mathematical perspective, there is still great potential to further investigate the questions raised in this work. In particular, it would be important to identify nonlinear autoencoder architectures that are unable to reconstruct torsion. From a computational perspective, a loss term that ensures torsion is preserved in the latent representation—provided it is present in the input data—would be relevant for data-driven classification of dynamical systems, especially given that computational limitations often prevent high-dimensional data from being tested for torsion. 

\section*{Acknowledgments}
We would like to thank Prof. Dr. Ippei Obayashi for providing the double and triple loop point cloud data and helpful suggestions for using the torsion checker algorithm in homcloud. Also we would like to thank Prof. Dr. Ulrich Thiel for some helpful comments for a newer version of this paper.

\printbibliography

%\bibliographystyle{unsrtnat}  % oder abbrvnat, plainnat, unsrt, ...
%\bibliography{references}     % nutzt deine references.bib

\clearpage
\onecolumn
\appendix
\section{Activation Functions and Torsion}
Figure~\ref{fig:activfunc} displays common activation functions as presented in Table~\ref{tab:actvfunc} and their effect on the  torsional double loop structure. Note that these graphs are intended to provide only a first impression; in deep neural networks, many activation functions are composed sequentially with linear weight matrices.

\begin{table}[h!]
	\centering
	\caption{Common activation functions and their mathematical definitions with illustrative sketches.}
	\renewcommand{\arraystretch}{2}
	\begin{tabular}{@{}lll@{}}
		\toprule
		\textbf{Activation Function} & \textbf{Definition} & \textbf{Sketch} \\
		\midrule
		
		ReLU &
		\( f(x) = \max(0, x) \) &
		\begin{adjustbox}{valign=m}
			\begin{tikzpicture}[scale=0.5]
				\begin{axis}[
					axis lines=middle,
					xmin=-2, xmax=2,
					ymin=-1, ymax=2,
					xtick=\empty, ytick=\empty,
					width=3.5cm, height=3.5cm,
					domain=-2:2, samples=100
					]
					\addplot[blue, thick] {max(0,x)};
				\end{axis}
			\end{tikzpicture}
		\end{adjustbox} \\
		
		Leaky ReLU &
		\( f(x) = \begin{cases} x & \text{if } x \geq 0 \\ \alpha x & \text{else} \end{cases},\ \alpha = 0.01 \) &
		\begin{adjustbox}{valign=m}
			\begin{tikzpicture}[scale=0.5]
				\begin{axis}[
					axis lines=middle,
					xmin=-2, xmax=2,
					ymin=-1, ymax=2,
					xtick=\empty, ytick=\empty,
					width=3.5cm, height=3.5cm,
					domain=-2:2, samples=100
					]
					\addplot[blue, thick, domain=-2:0] {0.01*x};
					\addplot[blue, thick, domain=0:2] {x};
				\end{axis}
			\end{tikzpicture}
		\end{adjustbox} \\
		
		Sigmoid &
		\( f(x) = \frac{1}{1 + e^{-x}} \) &
		\begin{adjustbox}{valign=m}
			\begin{tikzpicture}[scale=0.5]
				\begin{axis}[
					axis lines=middle,
					xmin=-6, xmax=6,
					ymin=-0.1, ymax=1.1,
					xtick=\empty, ytick=\empty,
					width=3.5cm, height=3.5cm,
					domain=-6:6, samples=100
					]
					\addplot[blue, thick] {1/(1 + exp(-x))};
				\end{axis}
			\end{tikzpicture}
		\end{adjustbox} \\
		
		Tanh &
		\( f(x) = \tanh(x) = \frac{e^x - e^{-x}}{e^x + e^{-x}} \) &
		\begin{adjustbox}{valign=m}
			\begin{tikzpicture}[scale=0.5]
				\begin{axis}[
					axis lines=middle,
					xmin=-3, xmax=3,
					ymin=-1.1, ymax=1.1,
					xtick=\empty, ytick=\empty,
					width=3.5cm, height=3.5cm,
					domain=-3:3, samples=100
					]
					\addplot[blue, thick] {tanh(x)};
				\end{axis}
			\end{tikzpicture}
		\end{adjustbox} \\
		
		ELU &
		\( f(x) = \begin{cases} x & \text{if } x \geq 0 \\ \alpha (e^x - 1) & \text{else} \end{cases},\ \alpha = 1 \) &
		\begin{adjustbox}{valign=m}
			\begin{tikzpicture}[scale=0.5]
				\begin{axis}[
					axis lines=middle,
					xmin=-2, xmax=2,
					ymin=-1.1, ymax=2,
					xtick=\empty, ytick=\empty,
					width=3.5cm, height=3.5cm,
					domain=-2:2, samples=100
					]
					\addplot[blue, thick, domain=-2:0] {exp(x)-1};
					\addplot[blue, thick, domain=0:2] {x};
				\end{axis}
			\end{tikzpicture}
		\end{adjustbox} \\
		
		Softplus &
		\( f(x) = \ln(1 + e^x) \) &
		\begin{adjustbox}{valign=m}
			\begin{tikzpicture}[scale=0.5]
				\begin{axis}[
					axis lines=middle,
					xmin=-4, xmax=4,
					ymin=-0.5, ymax=5,
					xtick=\empty, ytick=\empty,
					width=3.5cm, height=3.5cm,
					domain=-4:4, samples=100
					]
					\addplot[blue, thick] {ln(1 + exp(x))};
				\end{axis}
			\end{tikzpicture}
		\end{adjustbox} \\
		
		\bottomrule
	\end{tabular}
	\label{tab:actvfunc}
\end{table}

\begin{figure}[h!]
	\centering
	% Erste Zeile
	\begin{subfigure}[b]{0.3\textwidth}
		\centering
		\includegraphics[width=\textwidth]{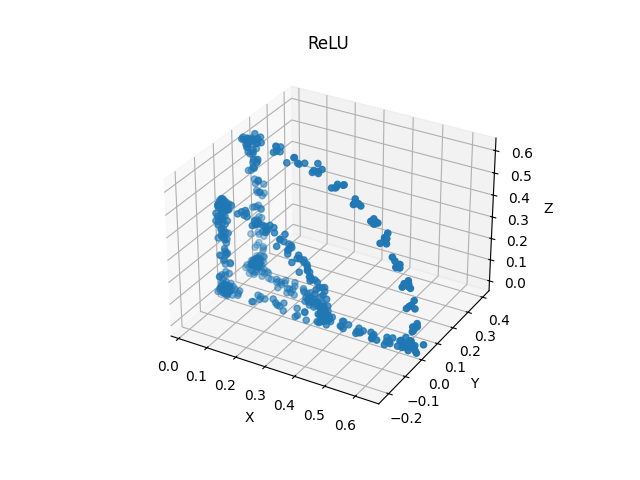}
		\caption{ReLU: No torsion.}
	\end{subfigure}
	\begin{subfigure}[b]{0.3\textwidth}
		\centering
		\includegraphics[width=\textwidth]{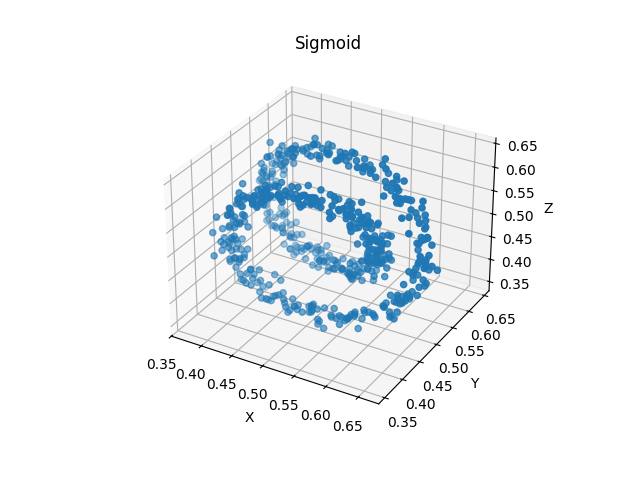}
		\caption{Sigmoid: Torsion.}
	\end{subfigure}
	\begin{subfigure}[b]{0.3\textwidth}
		\centering
		\includegraphics[width=\textwidth]{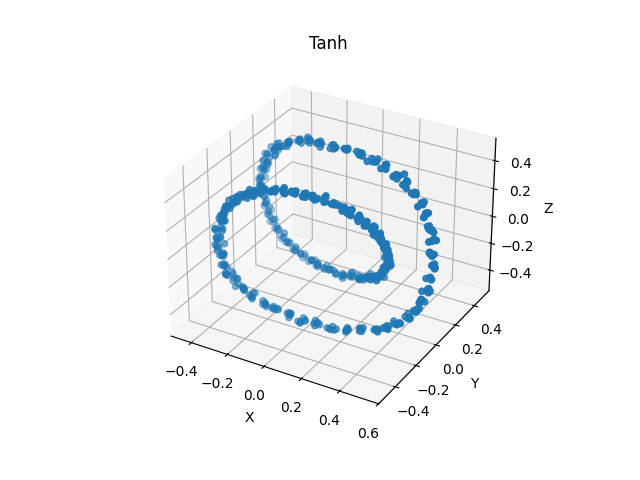}
		\caption{Tanh: Torsion.}
	\end{subfigure}
	
	% Zweite Zeile
	\begin{subfigure}[b]{0.3\textwidth}
		\centering
		\includegraphics[width=\textwidth]{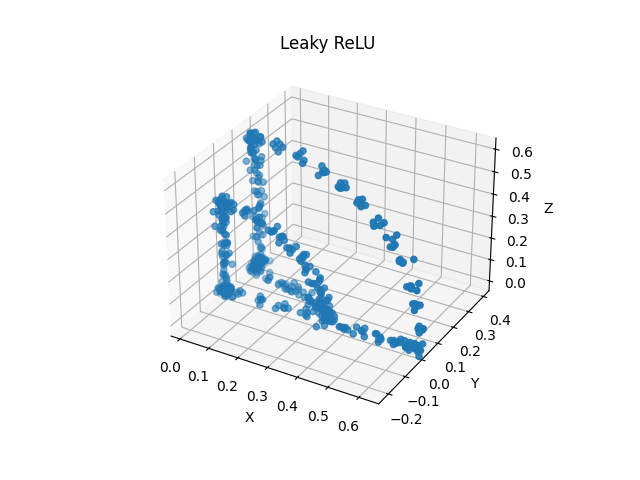}
		\caption{Leaky ReLU: No torsion.}
	\end{subfigure}
	\begin{subfigure}[b]{0.3\textwidth}
		\centering
		\includegraphics[width=\textwidth]{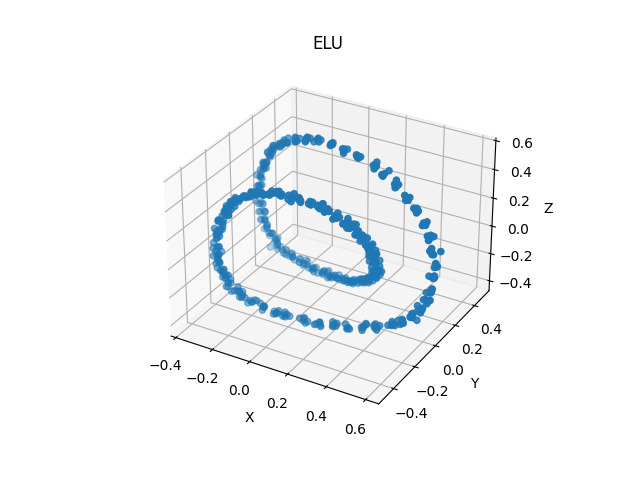}
		\caption{ELU: Torsion.}
	\end{subfigure}
	\begin{subfigure}[b]{0.3\textwidth}
		\centering
		\includegraphics[width=\textwidth]{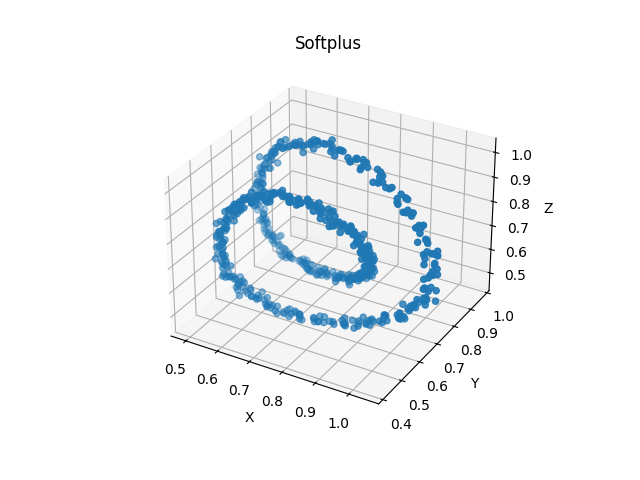}
		\caption{Softplus: Torsion}
	\end{subfigure}
	
	\caption{Double loop point cloud after application of the denoted activation functions. Small noise was added to avoid violation of the general position.}
	\label{fig:activfunc}
\end{figure}

\section{Autoencoders with TDA Losses}
\label{app:AE}
Loss terms are designed such that a (local) minimum loss ideally guarantees that certain characteristica of the data are learned. It appears desirable for the loss term of an autoencoder to be designed to maximize the similarity between the corresponding persistence diagrams of its layer states. In the following, we present two loss terms inspired by topological data analysis, which were designed with this purpose. The code of both algorithms is publicely available.

\subsection{Topological Autoencoder}
In a filtered simplicial complex the intervals in the barcode are each associated with a specific \enquote{creator simplex} and a \enquote{destroyer simplex} in the filtration. If the filtration is based on the metric distance between the data set elements, there is accordingly a \enquote{creator edge} and a \enquote{destroyer edge} corresponding to each interval, where for simplices with a dimension \( >1 \), this edge is selected from the simplex according to a weighting scheme.\newline
\noindent In \cite{moor2021topologicalautoencoders}, a loss term (in the following: \enquote{Topo loss}) is introduced with the aim to align these creator and destroyer edges between the input and the latent space. Imposing Assumption~\ref{ass:euclid}, whereby a simplicial complex is constructed on both the input data and the latent space according to the respective inherited Euclidean metric of each embedding space, the Topo loss term aligns the corresponding creator and destroyer simplices bidirectionally. Namely, 
\[
\mathcal{L}_{\mathrm{Topo}} := \mathcal{L}_{\mathcal{X} \to \mathcal{Z}} + \mathcal{L}_{\mathcal{Z} \to \mathcal{X}},
\]
where $\mathcal{X}$ denotes the input and $\mathcal{Z}$ denotes the latent space, with:
\begin{align}
	\mathcal{L}_{\mathcal{X} \to \mathcal{Z}} & := \frac{1}{2} \left\| A^X[\pi^X] - A^Z[\pi^X] \right\|^2, \label{eq:L_X_to_Z}\\
	\mathcal{L}_{\mathcal{Z} \to \mathcal{X}} & := \frac{1}{2} \left\| A^Z[\pi^Z] - A^X[\pi^Z] \right\|^2. \label{eq:L_Z_to_X}
\end{align}
In Eq.~\eqref{eq:L_X_to_Z}, $A^X[\pi^X]$ is a vector containing the creator and destroyer edges of the filtration on the input space and $A^Z[\pi^X]$ is their image under the encoder. Accordingly, in Eq.~\eqref{eq:L_Z_to_X}, $A^Z[\pi^Z]$ denotes the vector of creator and destroyer edges of the filtration constructed on the latent data and $A^X[\pi^Z]$ denotes their preimages under the encoder. The autoencoder is trained using a combination of the MSE loss and the Topo loss as:
\begin{equation}
	\label{eq:topoae}
\mathcal{L}= \mathcal{L}_{\mathrm{MSE}}+\eta \mathcal{L}_{\mathrm{Topo}}.	
\end{equation}

\subsection{RTD Autoencoder}
Representation Topology Divergence (RTD), as defined in \cite{barannikov2022representationtopologydivergencemethod} and adapted in \cite{trofimov2023learningtopologypreservingdatarepresentations} to define a loss term, assumes that the autoencoder induces a bijective mapping. The RTD loss term is designed in order to measure how far the induced map on the (persistent) homology groups deviates from an isomorphism.\\
\noindent For each of the two point clouds \( \mathcal{P} \) and \( \tilde{\mathcal{P}} \) (considered here as the input and output of an autoencoder), we construct a Vietoris-Rips complex represented as a weighted graph, where the weights correspond to the Euclidean distances between points. Specifically, let \( G^{w \leq \alpha} \) and \( G^{\tilde{w} \leq \alpha} \) denote the weighted graphs for \( \mathcal{P} \) and \( \tilde{\mathcal{P}} \), respectively, and let \( R_{\alpha}(G^{w}) \) and \( R_{\alpha}(G^{\tilde{w}}) \) represent their associated Vietoris-Rips complexes. Let \( f \) be a bijective map between the points in \( \mathcal{P} \) and \( \tilde{\mathcal{P}} \).
\newline
\noindent Define \( G^{\min(w, \tilde{w}) \leq \alpha} \) as the weighted graph formed by the union of the edges in \( G^{w \leq \alpha} \) and \( G^{\tilde{w} \leq \alpha} \). Specifically, \( G^{\min(w, \tilde{w}) \leq \alpha} \) includes an edge between two points \( p, p' \in \mathcal{P} \) if either the distance between \( p \) and \( p' \) or the distance between their images under \( f \), $f(p)$ and $f(p')$, is \( \leq \alpha \).\newline
\noindent Let 
\begin{equation}
	i: R_{\alpha}(G^{w}) \hookrightarrow R_{\alpha}(G^{\min(w, \tilde{w})})
\end{equation}
denote the inclusion map from the Vietoris-Rips complex associated with \( G^{w} \) to the Vietoris-Rips complex associated with \( G^{\min(w, \tilde{w})} \). In \cite{barannikov2022representationtopologydivergencemethod}, an auxiliary complex \( R_{\alpha}(\tilde{G}^{(w, \tilde{w})}) \) is constructed, which is homotopy equivalent to the cone of \( i \). The analogous construction is carried out for the inclusion \( \tilde{i} : R_{\alpha}(G^{\tilde{w}}) \hookrightarrow R_{\alpha}(G^{\min(w, \tilde{w})}) \), yielding an auxiliary complex \( R_{\alpha}(\tilde{G}^{(\tilde{w}, w)}) \) homotopy equivalent to the cone of \( \tilde{i} \).\newline
\noindent The representation topology divergence $\mathrm{RTD}_p(P,\tilde{P})$ (RTD) in homology dimension ${p\geq 1}$ associated with $R_{\alpha}(\tilde{G}^{(w, \tilde{w})})$ is defined as the sum of the interval's length in the $p$-dimensional barcode of $R_{\alpha}(\tilde{G}^{(w, \tilde{w})})$ and likewise for $R_{\alpha}(\tilde{G}^{(\tilde{w}, w)})$. For the inclusions $i$ and $\tilde{i}$ to induce isomorphisms on the homology groups, the homology of $R_{\alpha}(\tilde{G}^{(w, \tilde{w})})$ and likewise for $R_{\alpha}(\tilde{G}^{(\tilde{w}, w)})$ should be trivial and thus the RTD should be 0.\newline
\noindent Based on this, \cite{trofimov2023learningtopologypreservingdatarepresentations} define a loss term (\enquote{RTD loss}) as follows:
\begin{equation}
	\mathcal{L}_{\mathrm{RTD}}:=\sum_{i \geq 1} \left( \mathrm{RTD}_p (P, \tilde{P}) + \mathrm{RTD}_p(\tilde{P}, P) \right),
\end{equation}
for a homology dimension $p \geq 1$. when this loss term is used for the optimisation of the training of an autoencoder architecture, the two point clouds under consideration $P$ and $\tilde{P}$ would be identified with the input and output data set respectively.\newline
\noindent In practice, \cite{trofimov2023learningtopologypreservingdatarepresentations} used 
\begin{equation}
	\mathrm{RTD}_1(P,\tilde{P}) + \mathrm{RTD}_1(\tilde{P},P).
\end{equation}
The autoencoder was trained for $E_1$ epochs using only the MSE loss and $E_2$ epochs using a combined loss given as:
\begin{equation}
	\label{eq:rtdloss}
	\mathcal{L}=\mathcal{L}_{\mathrm{MSE}}+\chi \mathcal{L}_{\mathrm{RTD}},
\end{equation}
where $\chi$ denotes a hyperparameter regularizing the contribution of RTD to the total loss.

\newpage
\section{Double and Triple Loop Reconstructions}
The double and triple loop point clouds displayed in Figure~\ref{fig:original-loops} below were used as a proof of concept, as they represent datasets with a well-analyzable torsion structure and, due to their low dimensionality, are also easy to visualize. The following figures illustrate that torsion may remain unreconstructable even when the network performs well on the global dataset. Since the torsional and non-torsional datasets also appear very similar visually, this highlights that torsion is a topological property that cannot be learned or captured using the given global training methods.
\begin{table}[h!]
	\centering
	\begin{tabular}{cc}
		\includegraphics[width=0.25\textwidth]{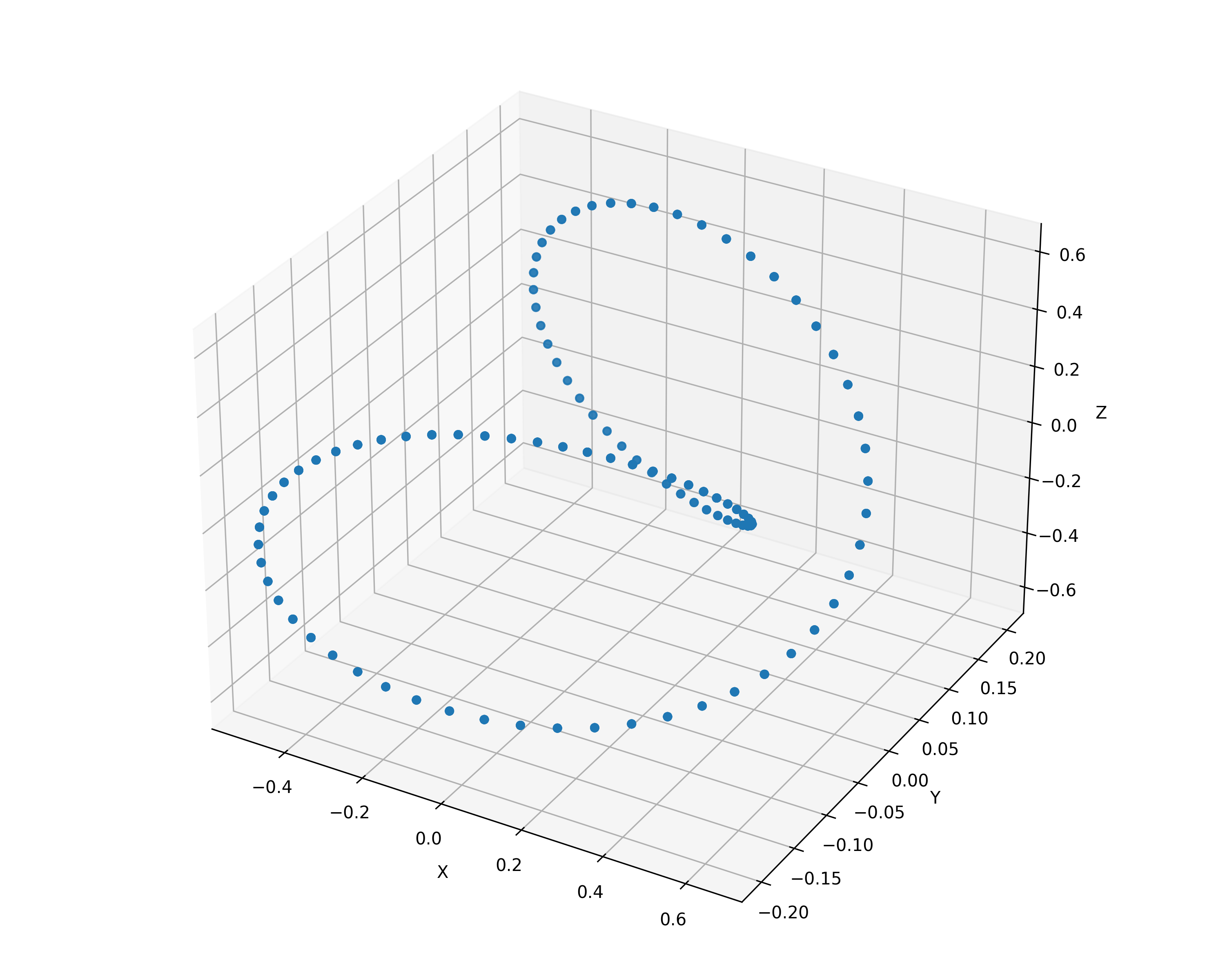} &
		\includegraphics[width=0.25\textwidth]{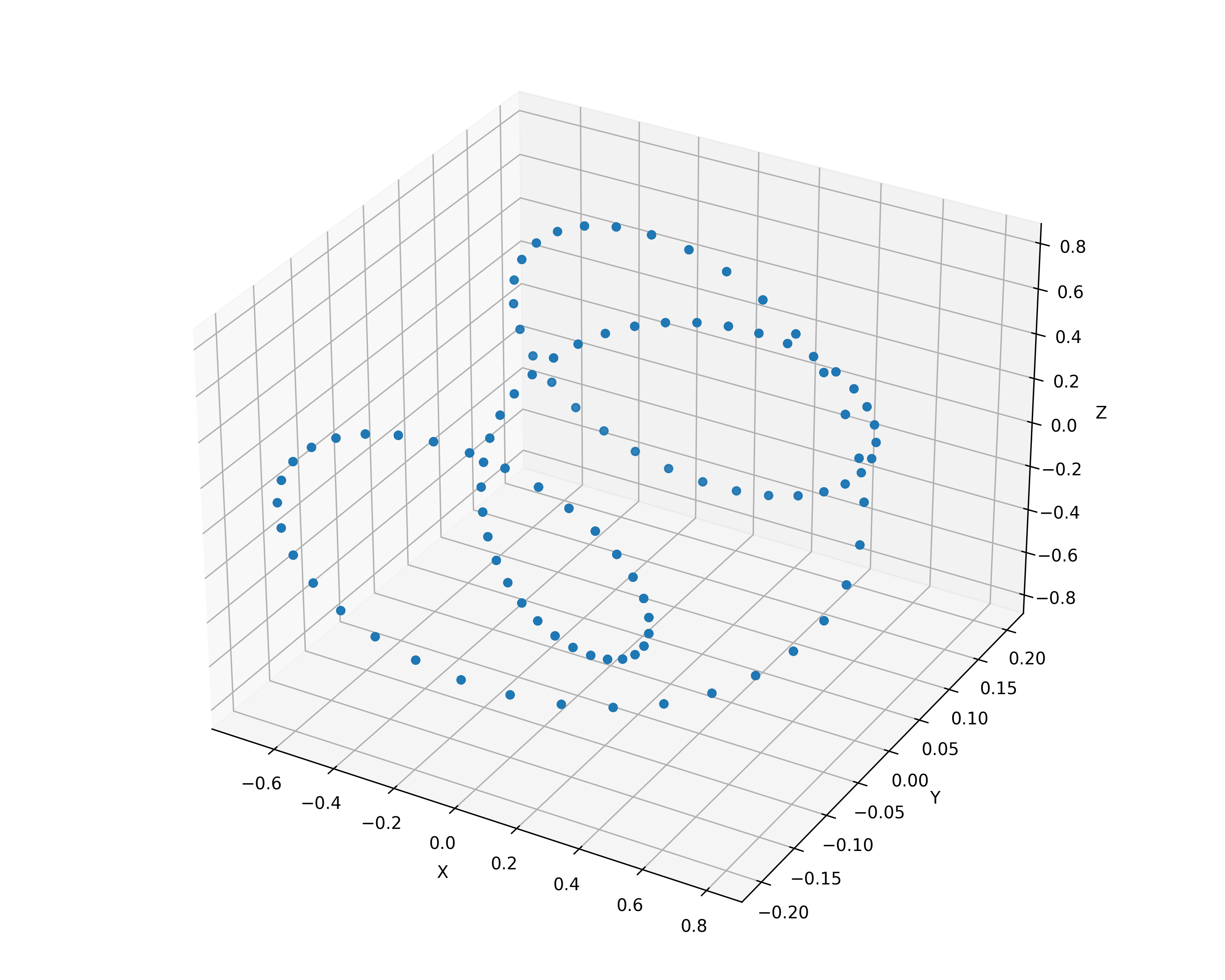} \\
		\toprule
		\textbf{Double Loop} & \textbf{Triple Loop} \\
	\end{tabular}
	\caption{Double and triple loop before transformation.}
	\label{fig:original-loops}
\end{table}

\subsection{Vanilla Autoencoder}

In Figure~\ref{fig:vanillaloop}, for the double loop, the three runs with the lowest MSE loss as well as the three torsional outputs with the lowest MSE loss are shown. These sets do not coincide.

\begin{figure}[H]
	\centering
	\renewcommand{\arraystretch}{1.2}
	\setlength{\tabcolsep}{10pt}
	\begin{tabular}{ccc}
		\includegraphics[width=0.25\textwidth]{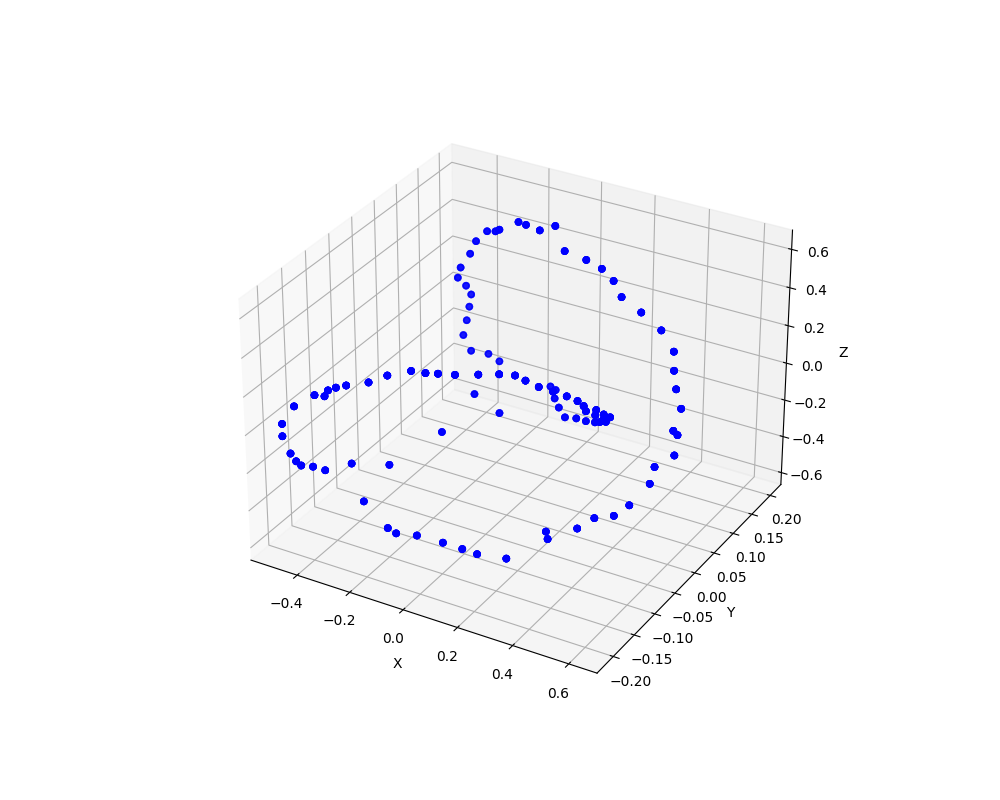} &
		\includegraphics[width=0.25\textwidth]{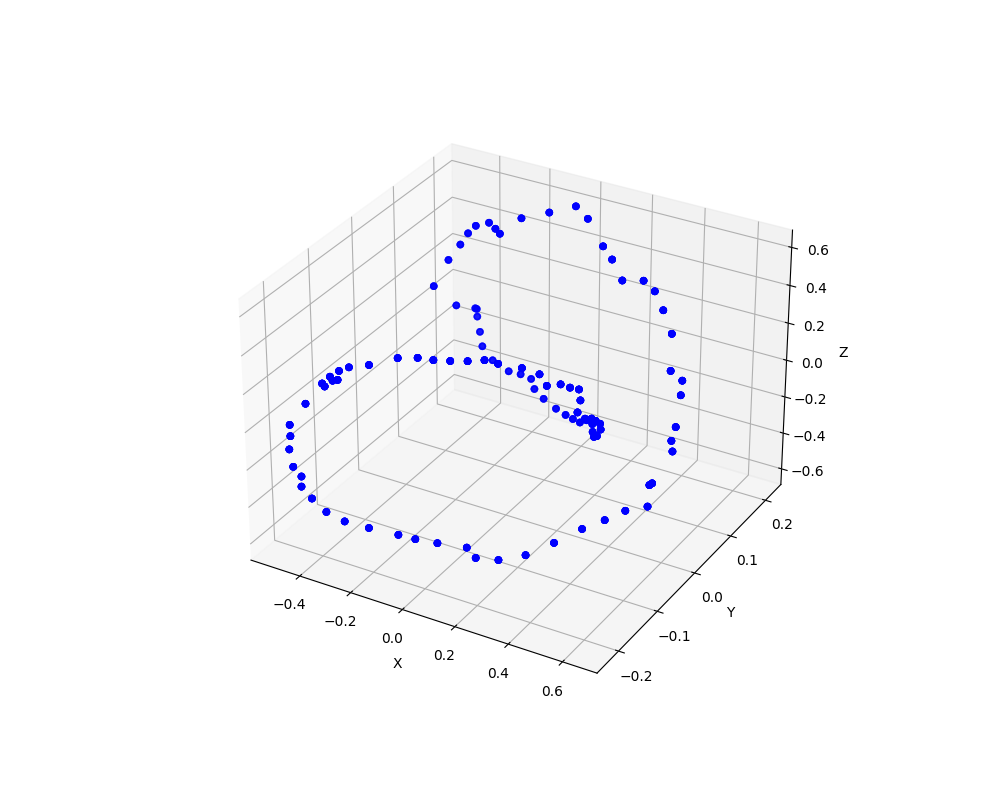} &
		\includegraphics[width=0.25\textwidth]{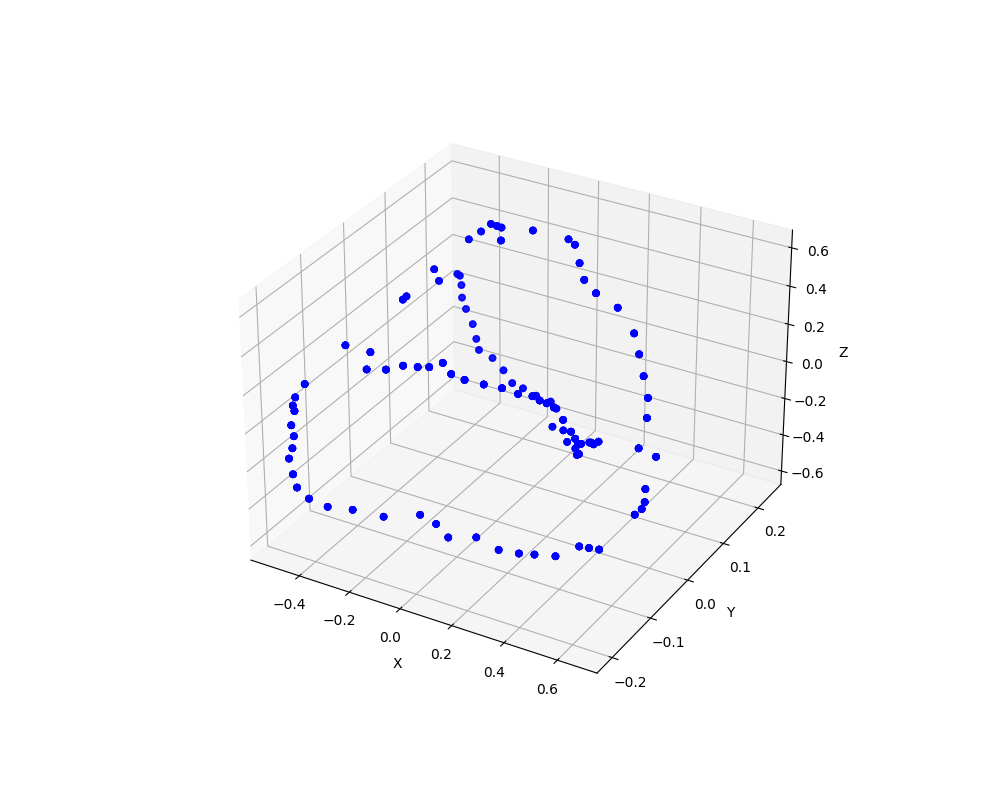} \\
		\midrule
		Run 17 & Run 39 & Run 1 \\
		\midrule
		MSE loss 0.0019 & MSE loss 0.0020 & MSE loss 0.0023 \\
		\addlinespace[1em]
		\includegraphics[width=0.25\textwidth]{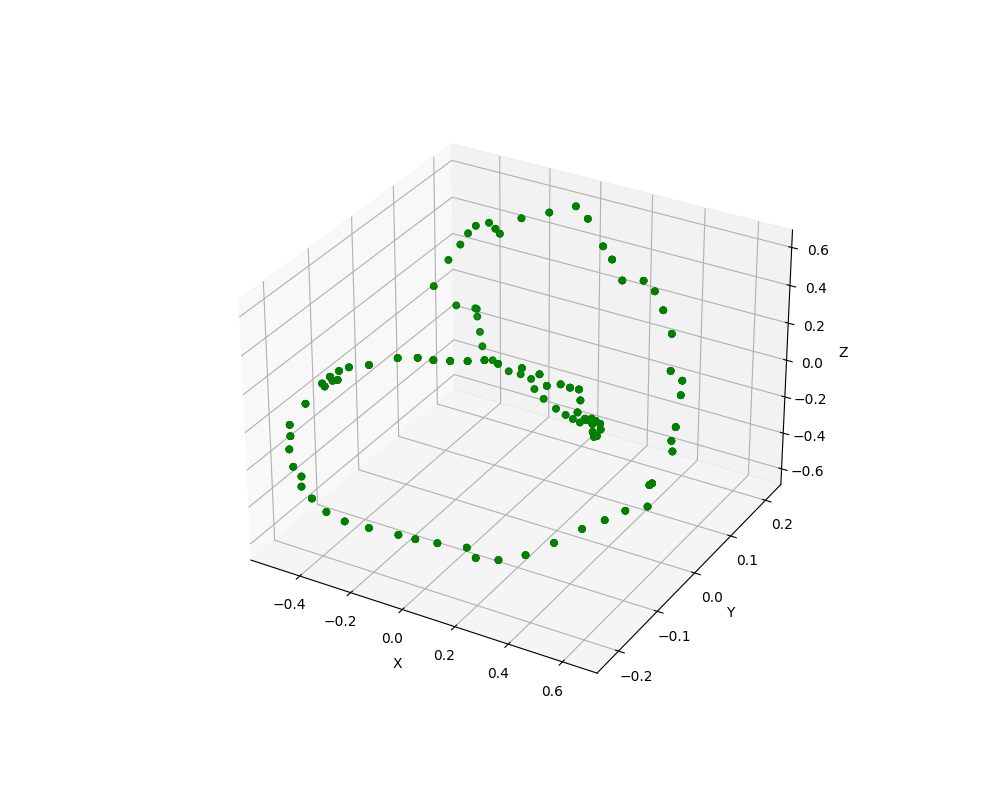} &
		\includegraphics[width=0.25\textwidth]{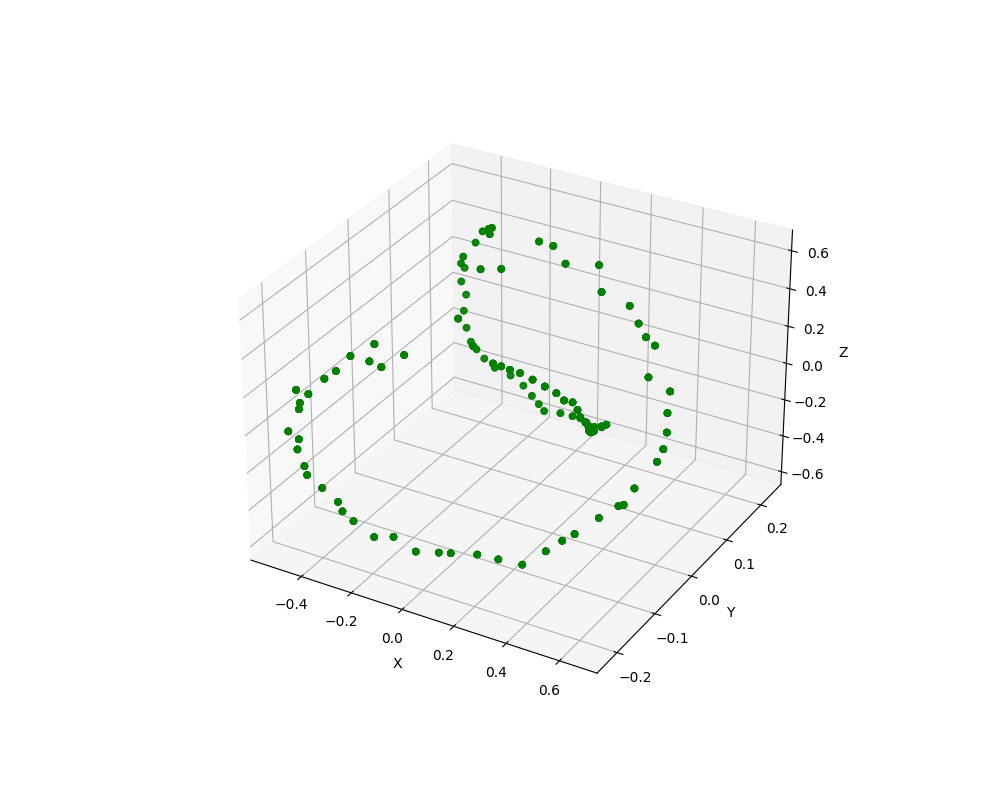} &
		\includegraphics[width=0.25\textwidth]{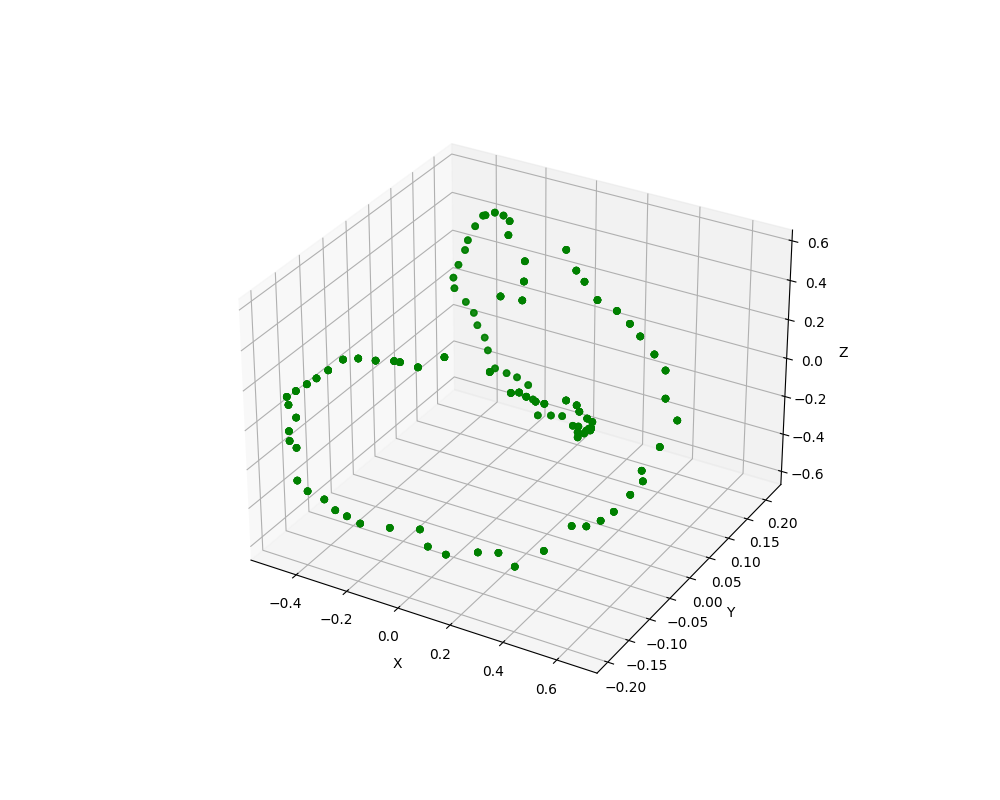} \\
		\midrule
		Run 39 & Run 12 & Run 21 \\
		\midrule
		MSE loss 0.0020 & MSE loss 0.0032 & MSE loss 0.0038 \\
	\end{tabular}
	\caption{Top row: Runs with lowest MSE loss for the vanilla autoencoder on the double loop. Bottom row: Torsional outputs with lowest MSE loss.}
	\label{fig:vanillaloop}
\end{figure}

\noindent In Figure~\ref{fig:mod3fig}, for the triple loop, the three runs with the lowest MSE loss are depicted. None of them exhibits torsion. The visual comparison with the original triple loop also shows that the overfitted vanilla autoencoder was unable to reconstruct the complex triple winding structure.

\begin{figure}[H]
	\centering
	\renewcommand{\arraystretch}{1.2}
	\setlength{\tabcolsep}{10pt}
	\begin{tabular}{ccc}
		\includegraphics[width=0.25\textwidth]{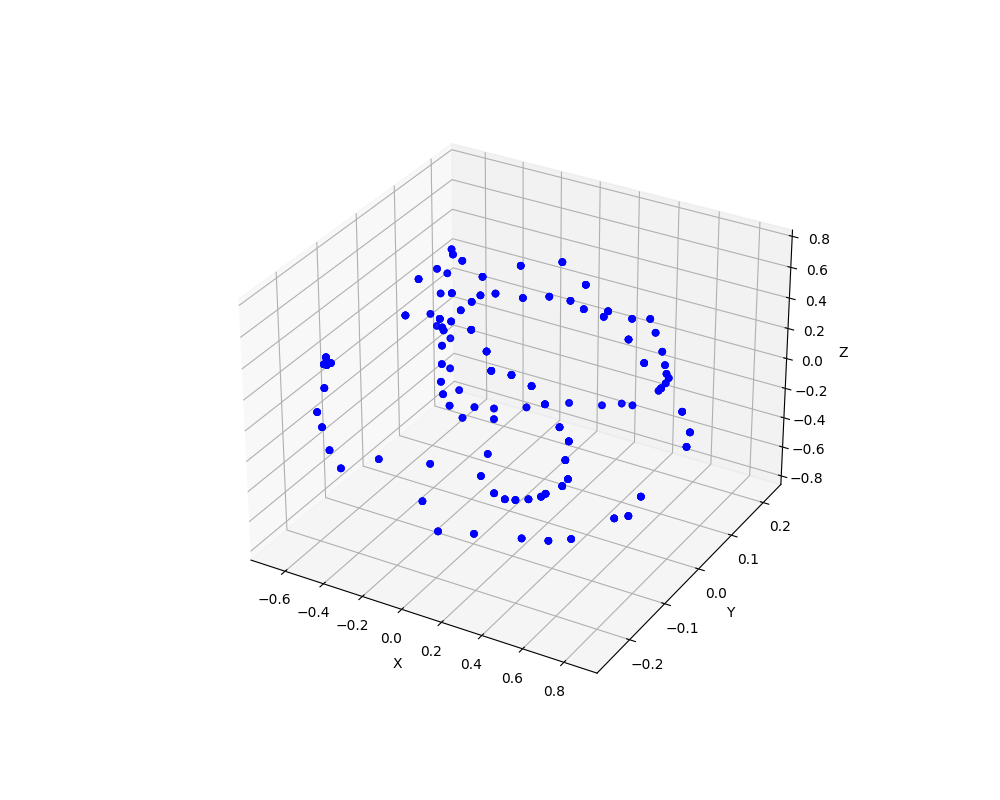} &
		\includegraphics[width=0.25\textwidth]{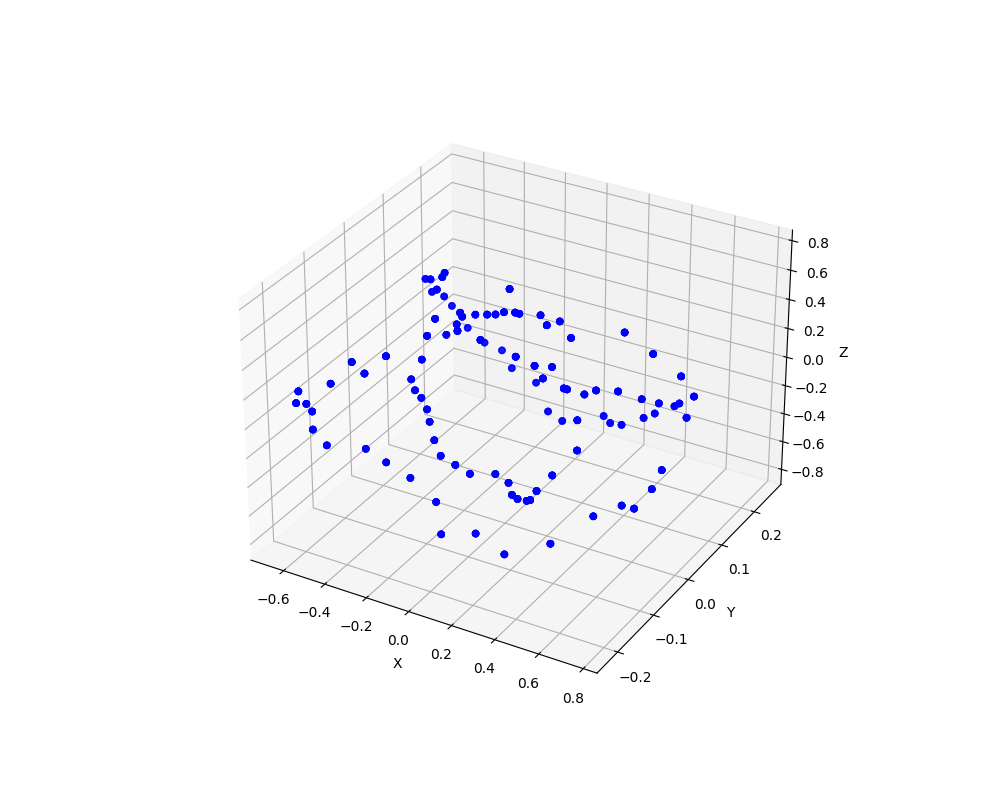} &
		\includegraphics[width=0.25\textwidth]{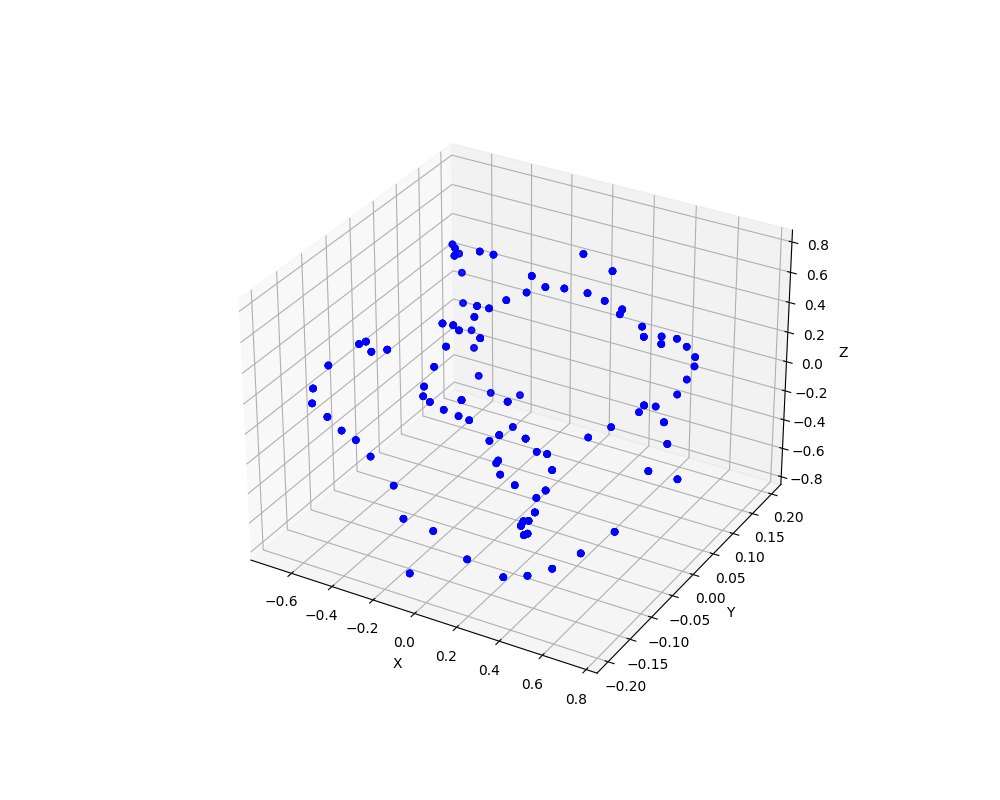} \\
		\midrule
		Run 2 & Run 19 & Run 36 \\
		\midrule
		MSE loss 0.0038 & MSE loss 0.0055 & MSE loss 0.0037 \\
	\end{tabular}
	\caption{Runs with the lowest MSE loss for the vanilla autoencoder for the triple loop.}
	\label{fig:mod3fig}
\end{figure}

\newpage

\subsection{Topological Autoencoder}

In Figure~\ref{fig:topoae_visuals}, for the double loop, the three runs with the lowest MSE loss as well as the three runs with the lowest topological and total loss are depicted. These do not coincide with the three torsional runs.

\begin{figure}[H]
	\centering
	\renewcommand{\arraystretch}{1.2}
	\setlength{\tabcolsep}{10pt}
	\begin{tabular}{ccc}
		\includegraphics[width=0.25\textwidth]{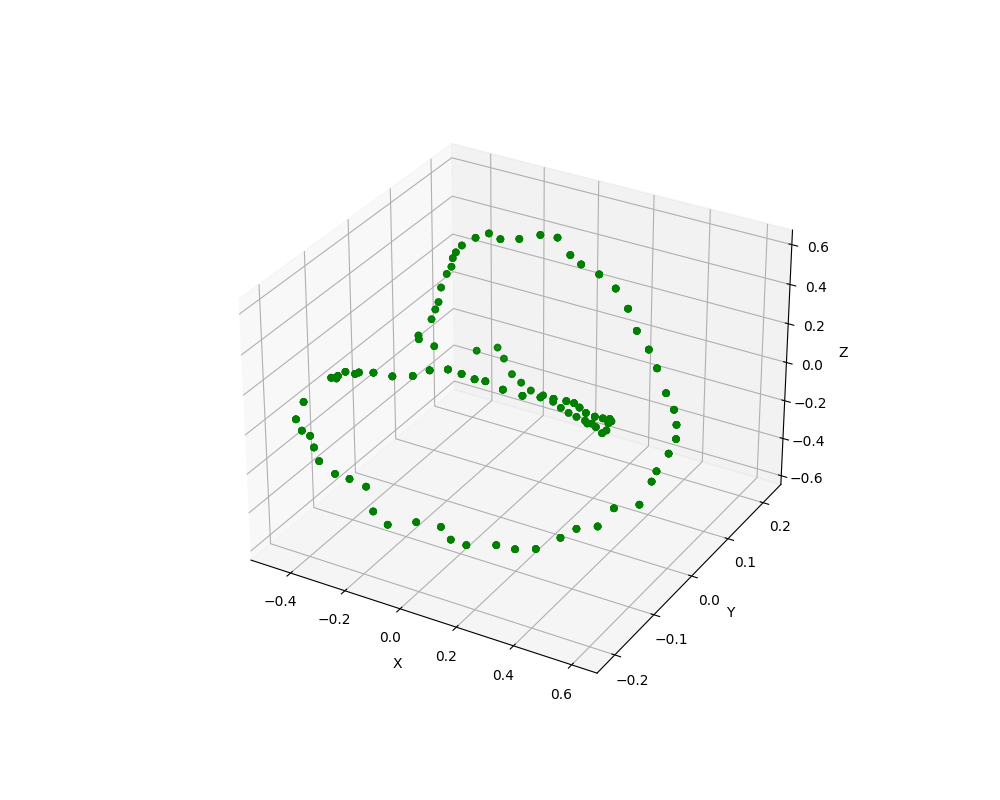} &
		\includegraphics[width=0.25\textwidth]{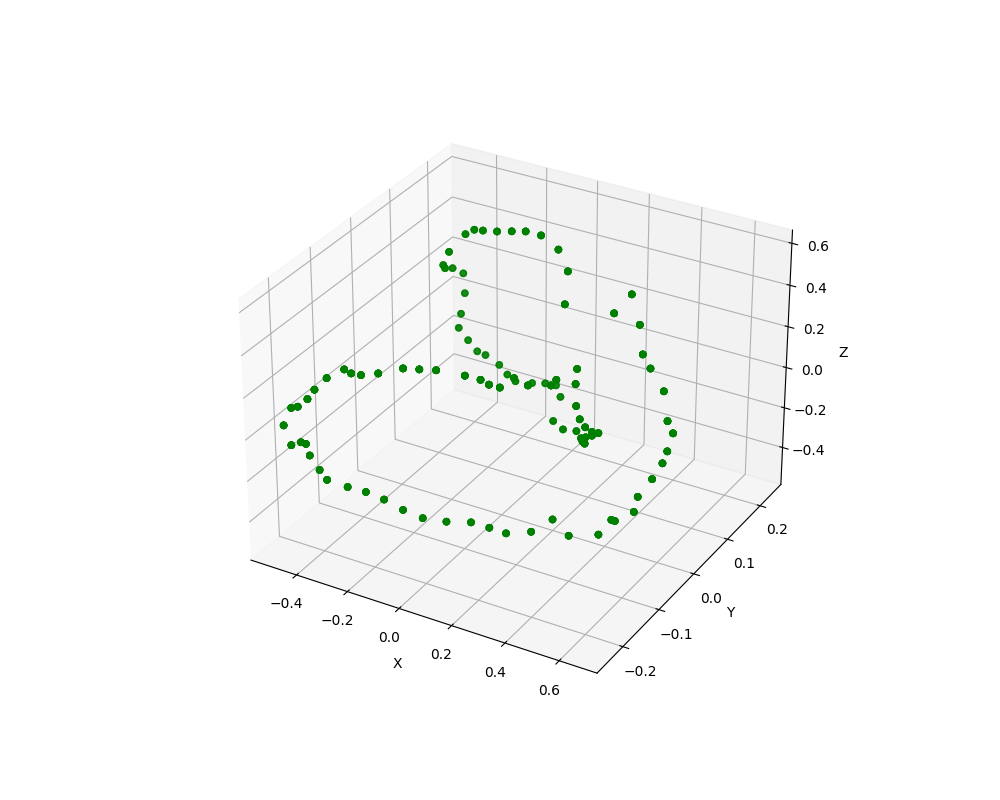} &
		\includegraphics[width=0.25\textwidth]{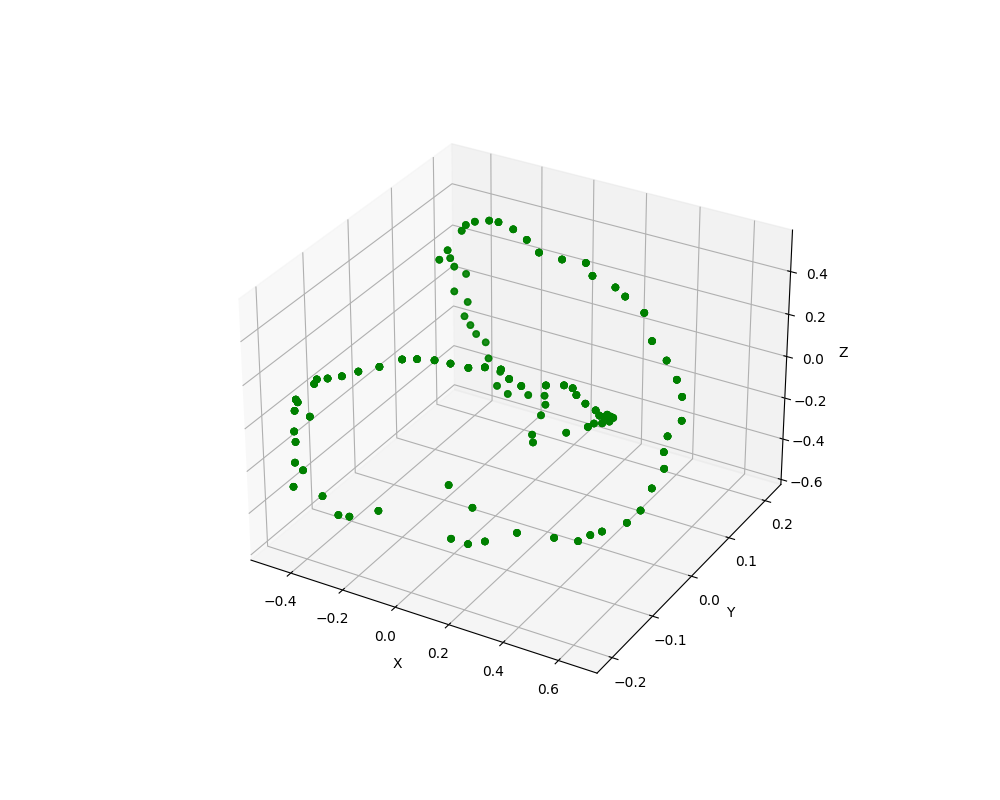} \\
		\midrule
		Run 2 & Run 4 & Run 6 \\
		\midrule
		MSE loss 0.0016 & MSE loss 0.0017 & MSE loss 0.0013 \\
		Topo loss 0.2945 & Topo loss 0.2127 & Topo loss 0.1057 \\
		Total loss 0.0566 & Total loss 0.0316 & Total loss 0.0162 \\
		\addlinespace[1em]
		\includegraphics[width=0.25\textwidth]{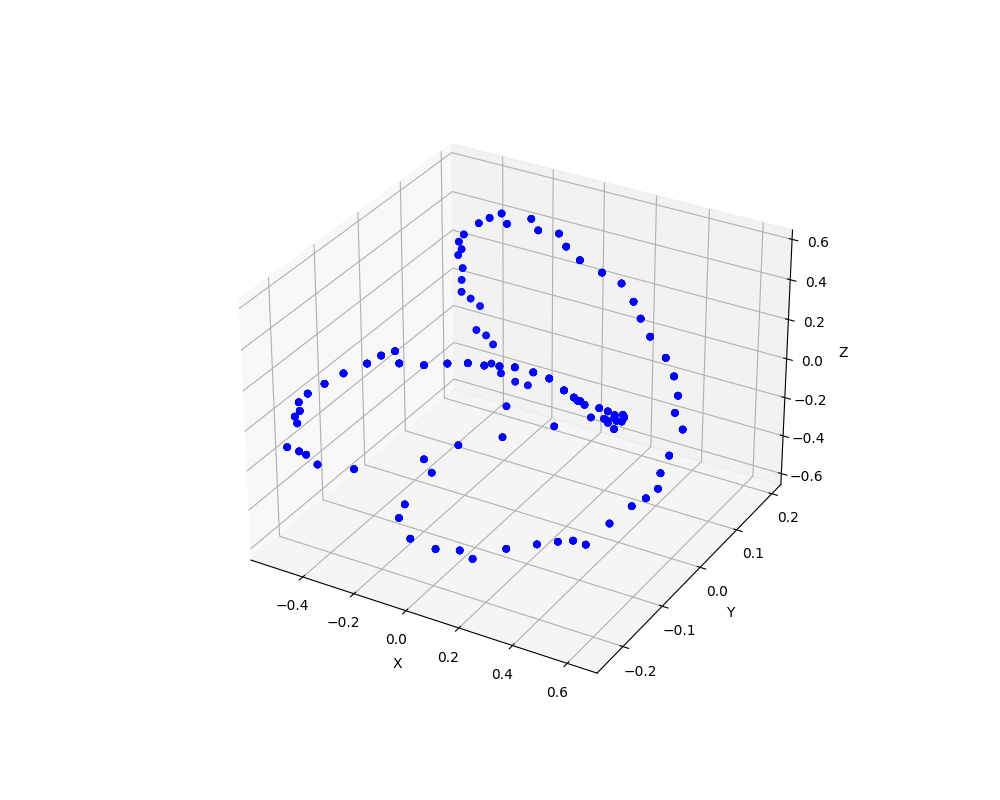} &
		\includegraphics[width=0.25\textwidth]{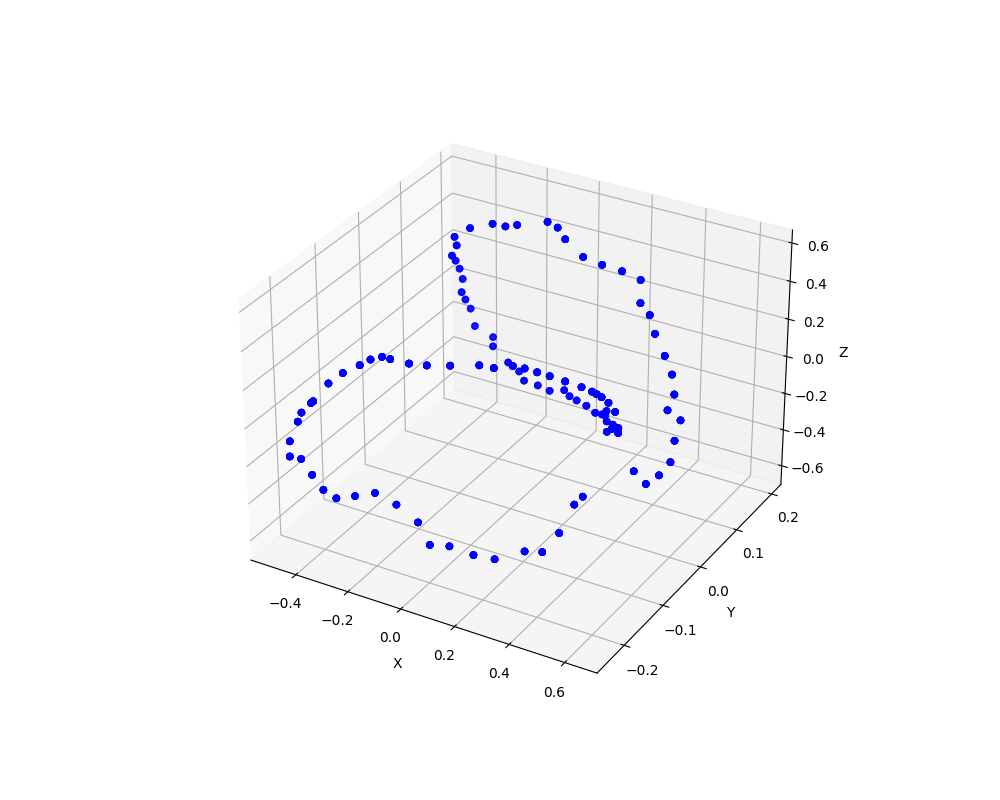} &
		\includegraphics[width=0.25\textwidth]{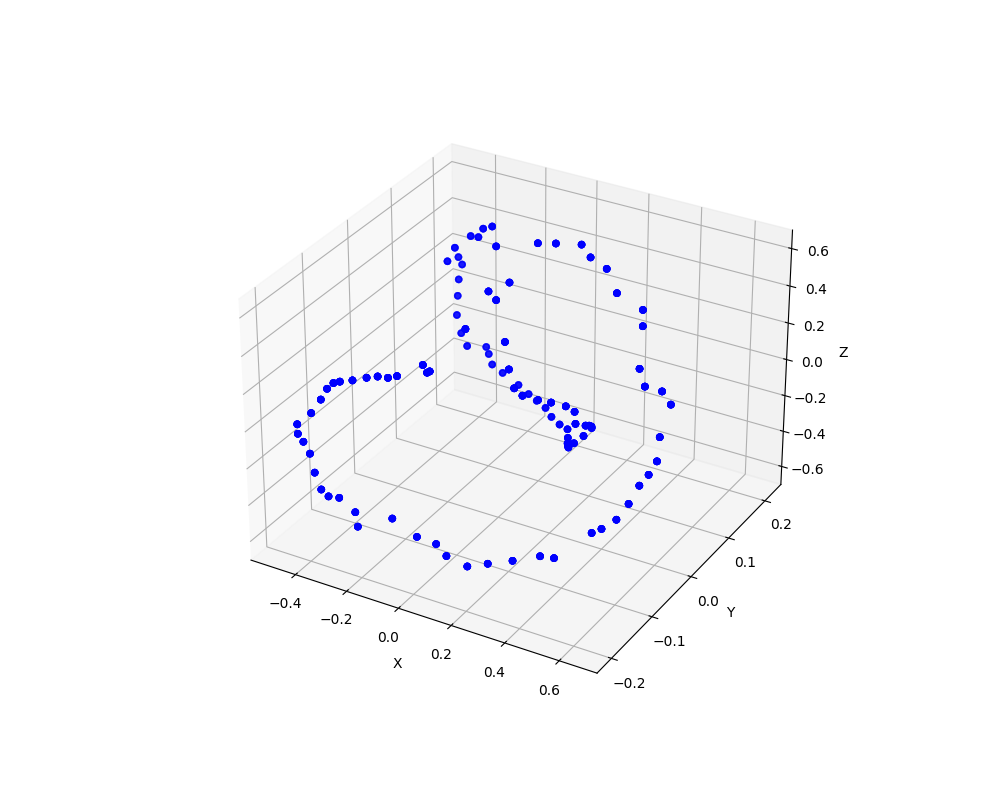} \\
		\midrule
		Run 14 & Run 27 & Run 37 \\
		\midrule
		MSE loss 0.0012 & MSE loss 0.0010 & MSE loss 0.0011 \\
		Topo loss 0.2685 & Topo loss 0.1568 & Topo loss 0.3063 \\
		Total loss 0.0513 & Total loss 0.0167 & Total loss 0.0509 \\
		\addlinespace[1em]
		\includegraphics[width=0.25\textwidth]{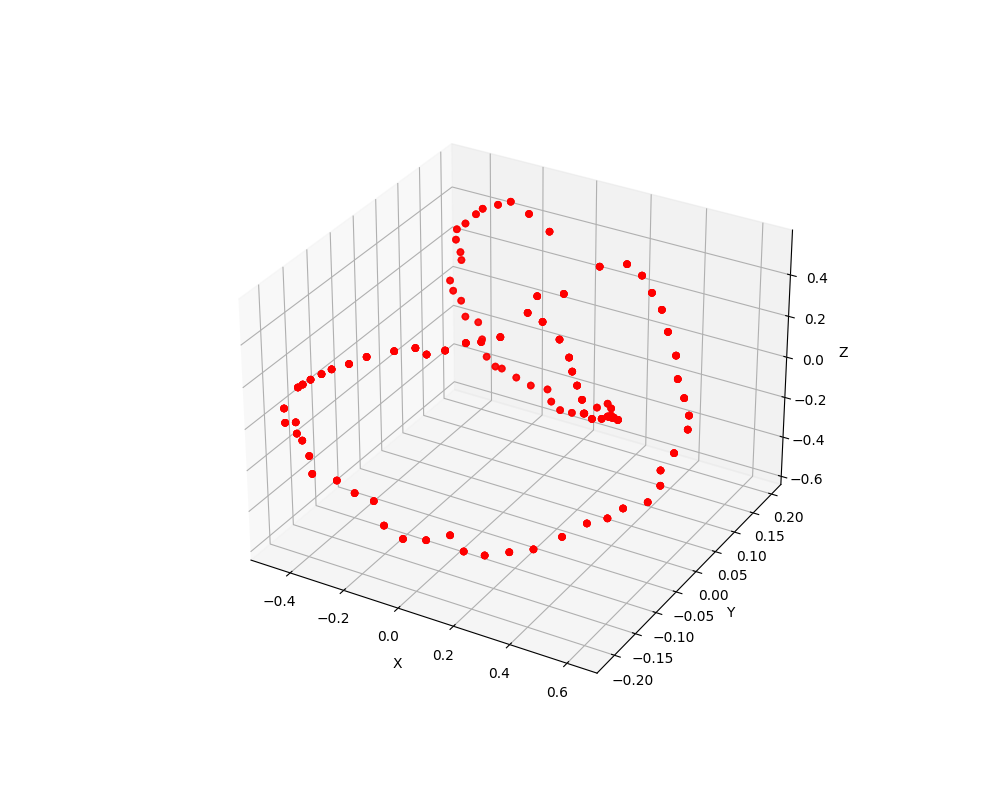} &
		\includegraphics[width=0.25\textwidth]{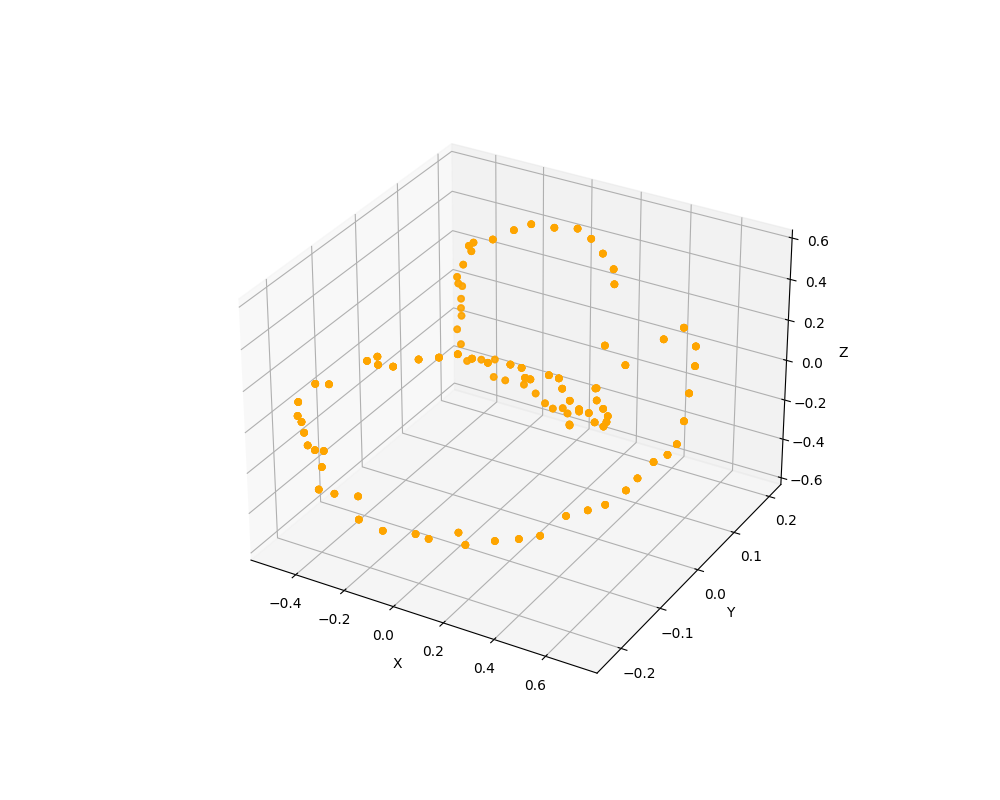} &
		\\[-1em]
		\midrule
		Run 22 & Run 15 & \\
		\midrule
		MSE loss 0.0012 & MSE loss 0.0029 & \\
		Topo loss 0.2512 & Topo loss 0.2002 & \\
		Total loss 0.0263 & Total loss 0.0310 & \\
	\end{tabular}
	\caption{Top row: Torsional outputs for the topological autoencoder. Middle row: Outputs with lowest MSE loss. Bottom row: Outputs with lowest total and topological loss.}
	\label{fig:topoae_visuals}
\end{figure}

\noindent Also the topological autoencoder could not reconstruct the triple loop structure, as shown in Figure~\ref{fig:topomod3_visuals}.

\begin{figure}[H]
	\centering
	\renewcommand{\arraystretch}{1.2}
	\setlength{\tabcolsep}{10pt}
	\begin{tabular}{ccc}
		\includegraphics[width=0.25\textwidth]{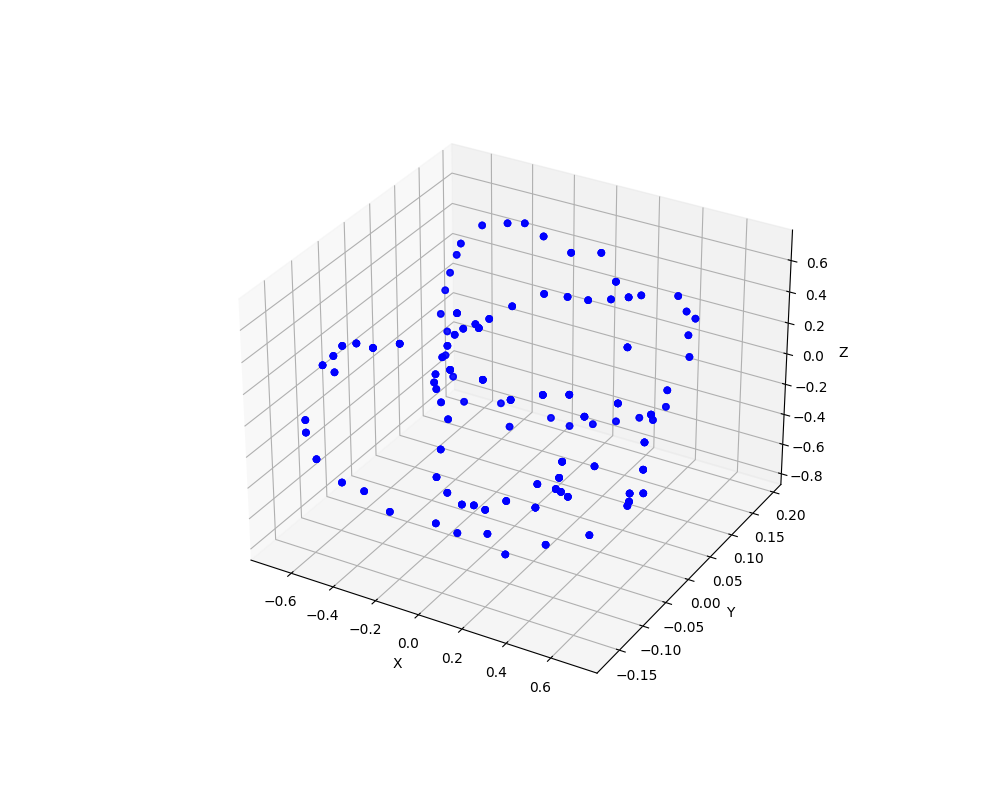} &
		\includegraphics[width=0.25\textwidth]{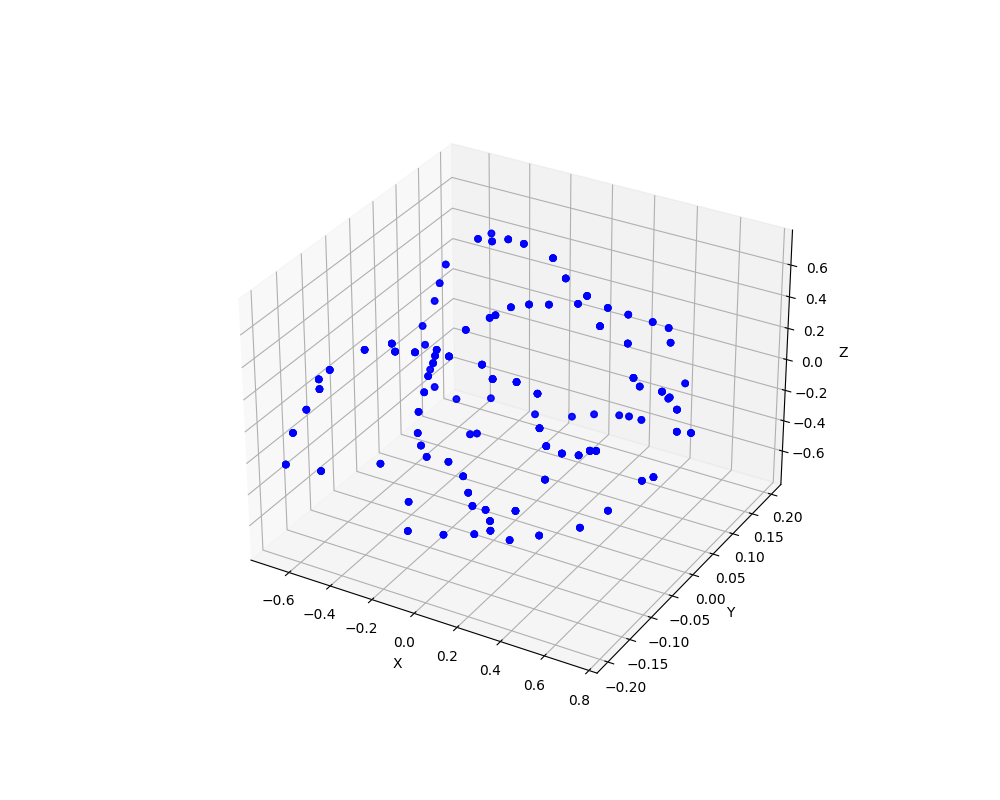} &
		\includegraphics[width=0.25\textwidth]{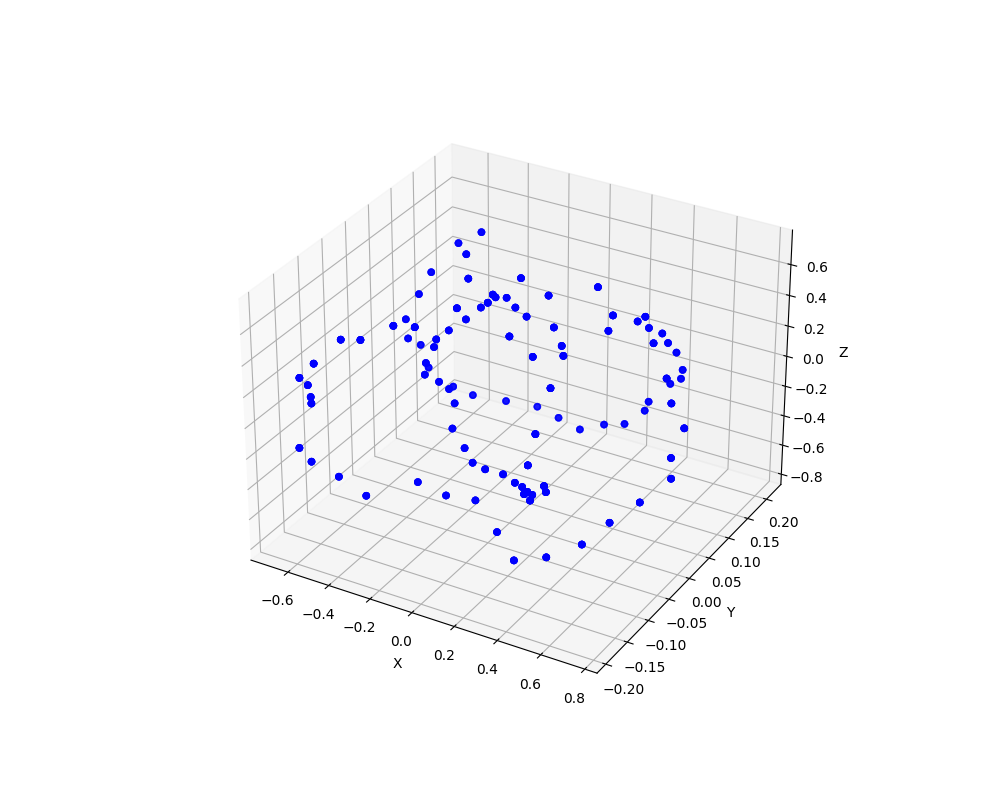} \\
		\midrule
		Run 13 & Run 33 & Run 15 \\
		\midrule
		MSE loss 0.0021 & MSE loss 0.0025 & MSE loss 0.0026 \\
		Topo loss 0.2695 & Topo loss 0.2698 & Topo loss 0.3142 \\
		Total loss 0.0290 & Total loss 0.0295 & Total loss 0.0362 \\
		\addlinespace[1em]
		\includegraphics[width=0.25\textwidth]{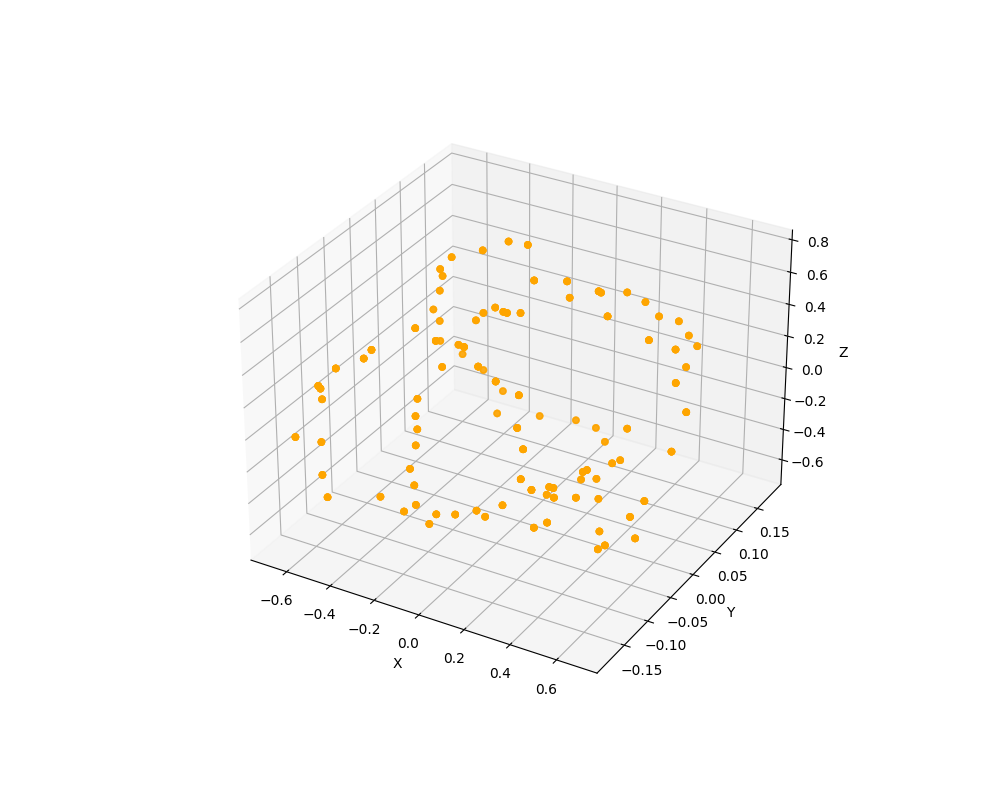} &
		\includegraphics[width=0.25\textwidth]{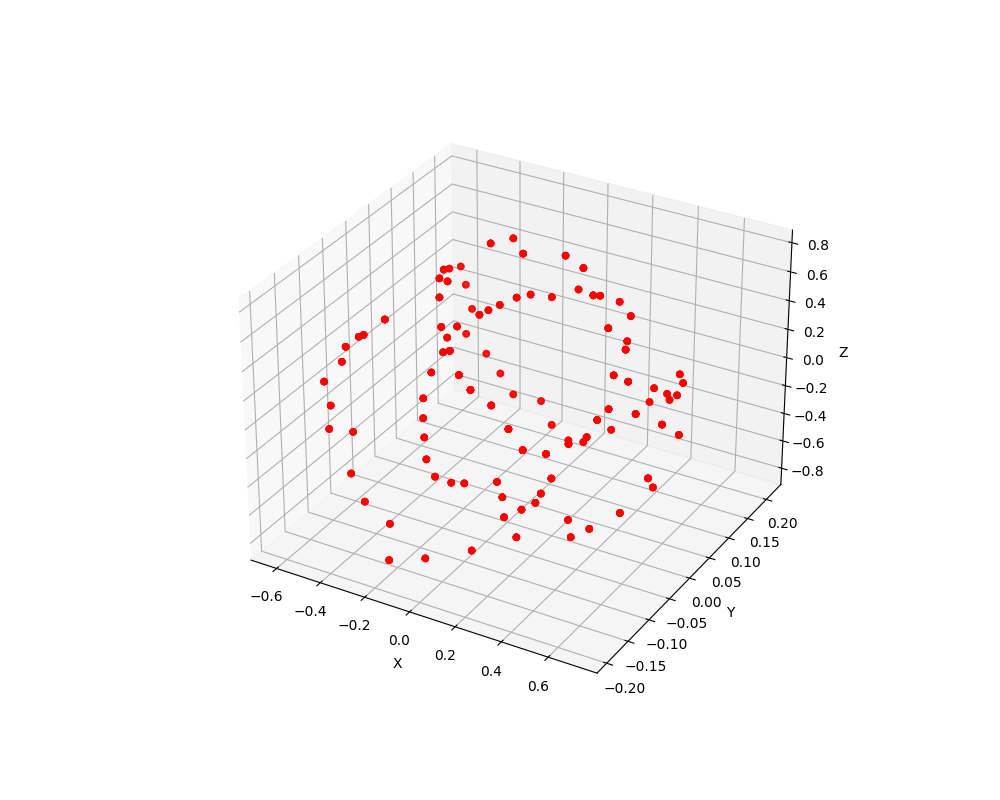} &
		\\
		\midrule
		Run 21 & Run 28 & \\
		\midrule
		MSE loss 0.0033 & MSE loss 0.0026 & \\
		Topo loss 0.2840 & Topo loss 0.2919 & \\
		Total loss 0.0450 & Total loss 0.0318 & \\
	\end{tabular}
	\caption{Top row: Runs with lowest MSE loss for the topological autoencoder. Bottom row: Outputs with best topological and total loss.}
	\label{fig:topomod3_visuals}
\end{figure}

\newpage

\subsection{RTD Autoencoder}
For the RTD autoencoder, the three runs with the lowest MSE loss depicted in Figure~\ref{fig:rtd_table_visuals} exhibit torsion. However, also a minimum RTD loss does not guarantee reconstructability of torsional structures.

\begin{figure}[H]
	\centering
	\renewcommand{\arraystretch}{1.2}
	\setlength{\tabcolsep}{10pt}
	\begin{tabular}{ccc}
		\includegraphics[width=0.25\textwidth]{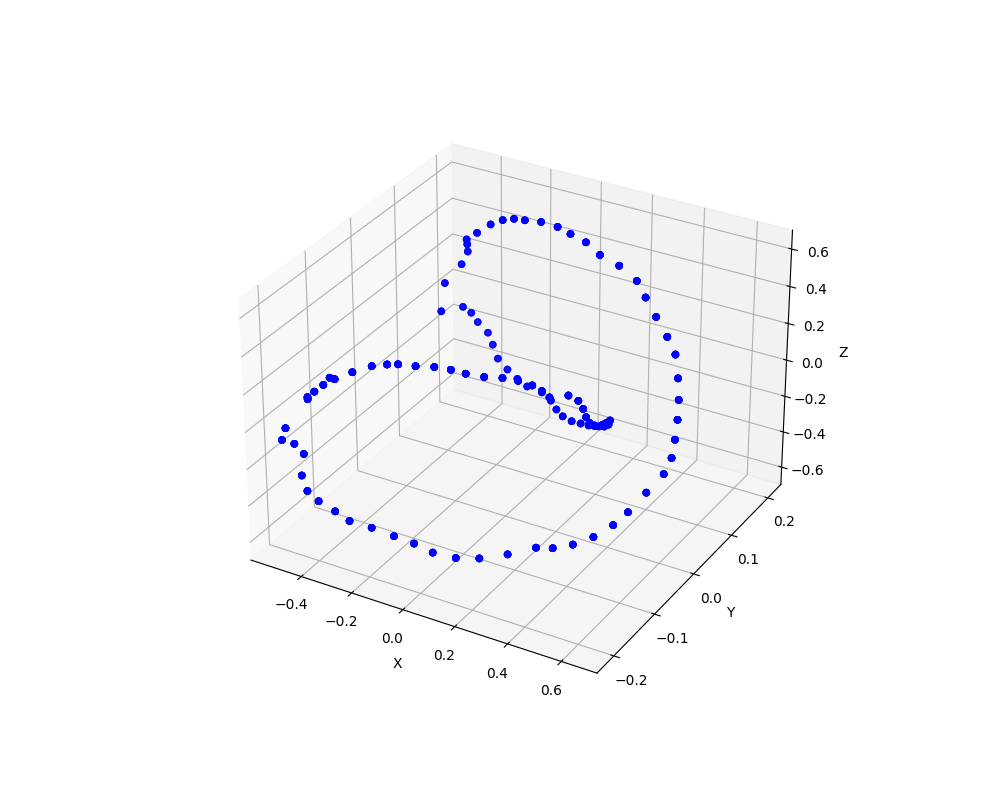} &
		\includegraphics[width=0.25\textwidth]{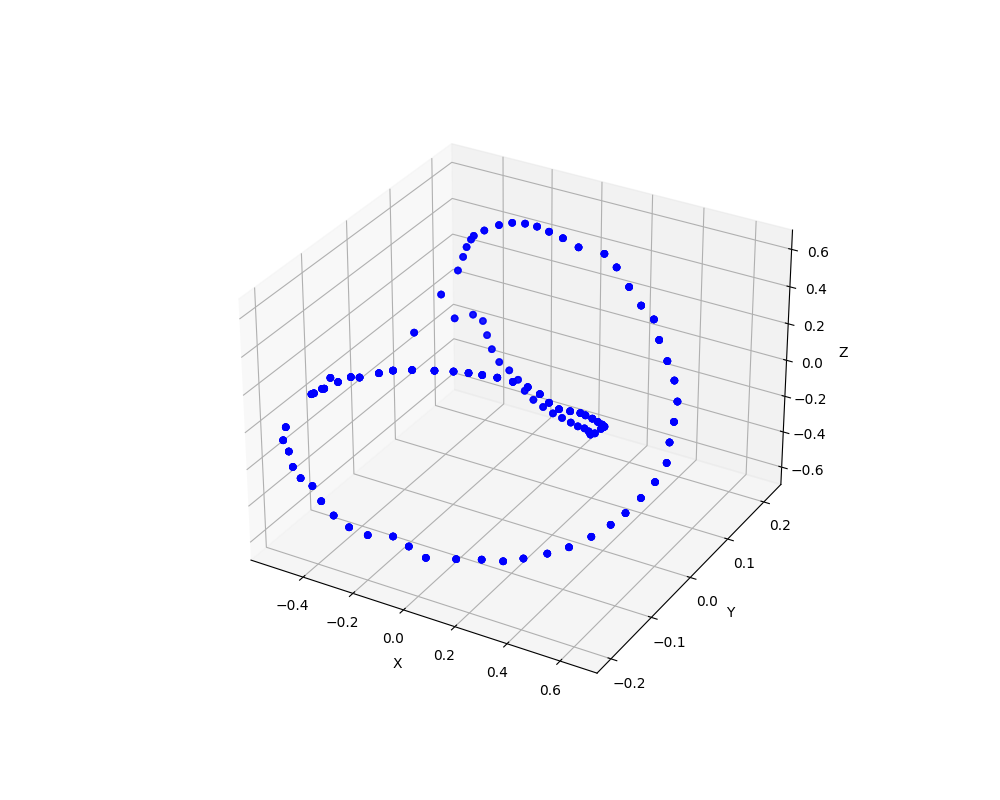} &
		\includegraphics[width=0.25\textwidth]{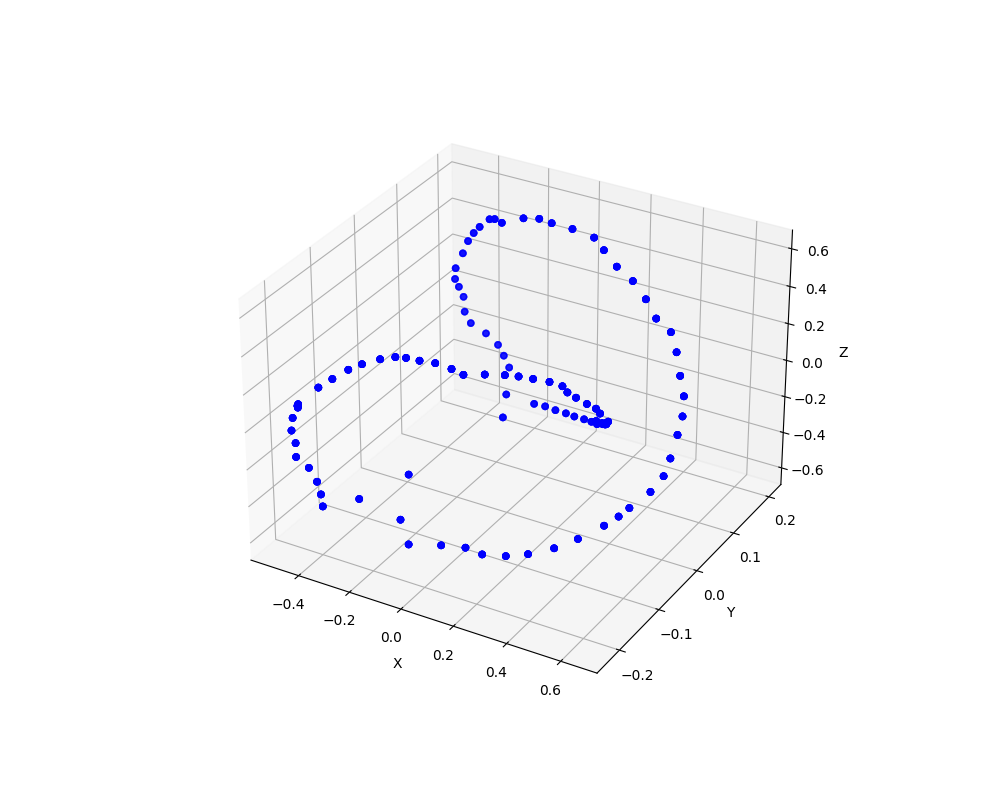} \\
		\midrule
		Run 25 & Run 39 & Run 35 \\
		\midrule
		MSE loss 0.000066 & MSE loss 0.000147 & MSE loss 0.000151 \\
		RTD loss 3.361837 & RTD loss 2.749048 & RTD loss 3.055326 \\
		Total loss 3.361903 & Total loss 2.749195 & Total loss 3.055477 \\
		\addlinespace[1em]
		\includegraphics[width=0.25\textwidth]{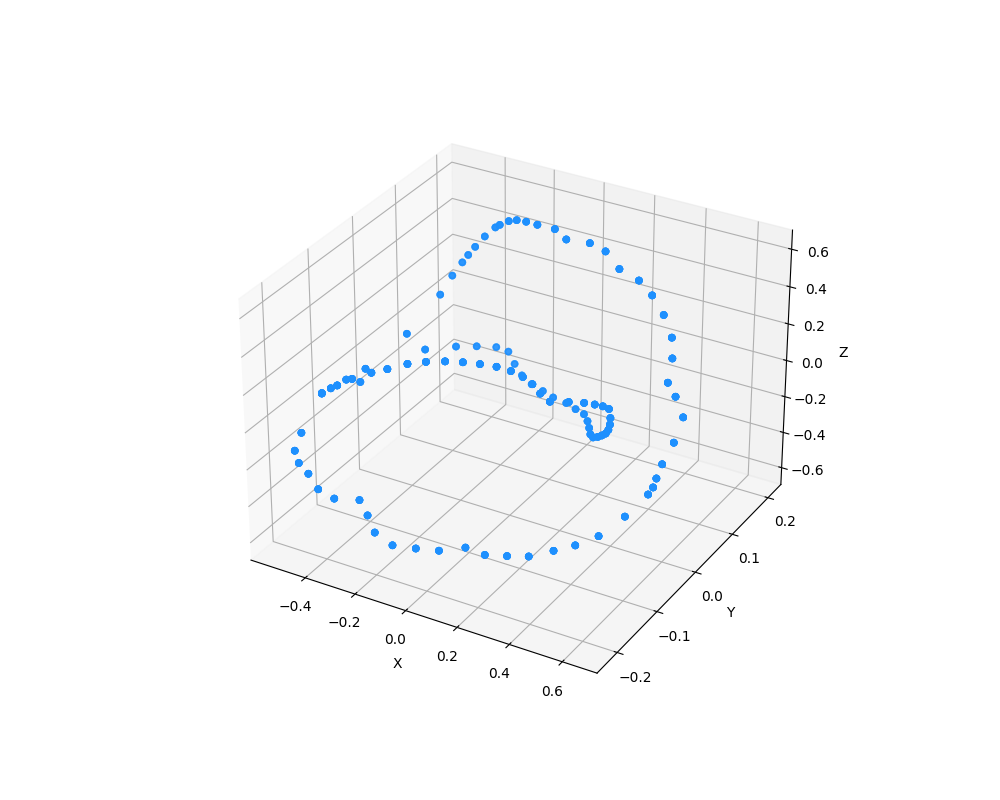} &
		\includegraphics[width=0.25\textwidth]{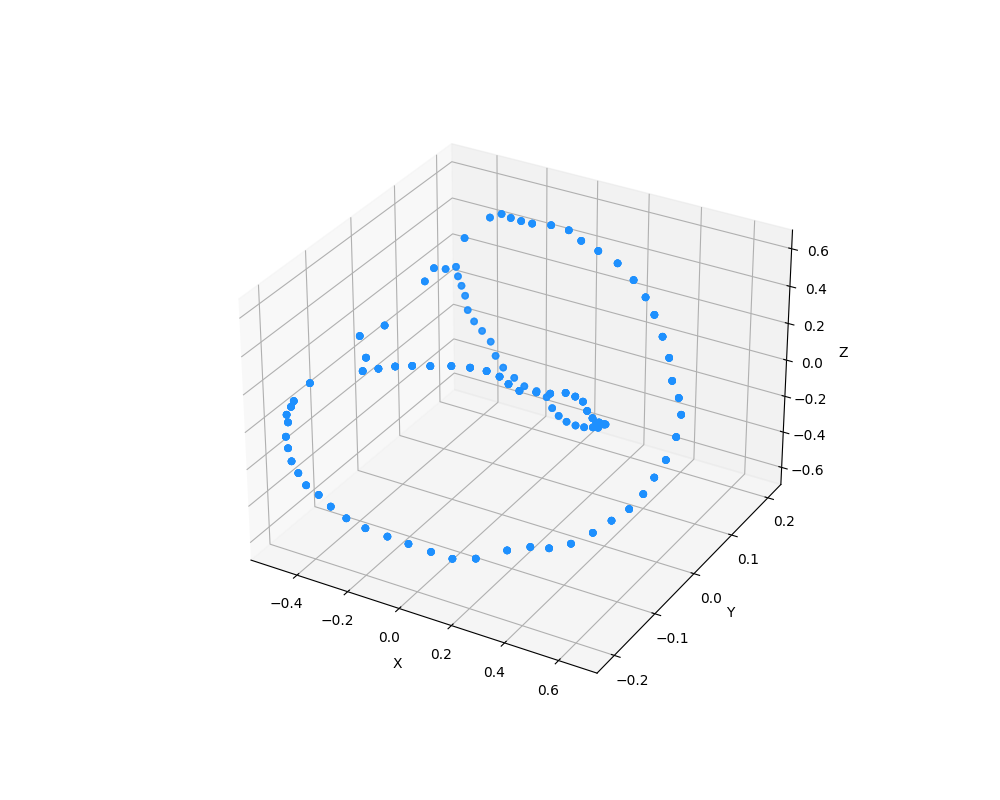} &
		\includegraphics[width=0.25\textwidth]{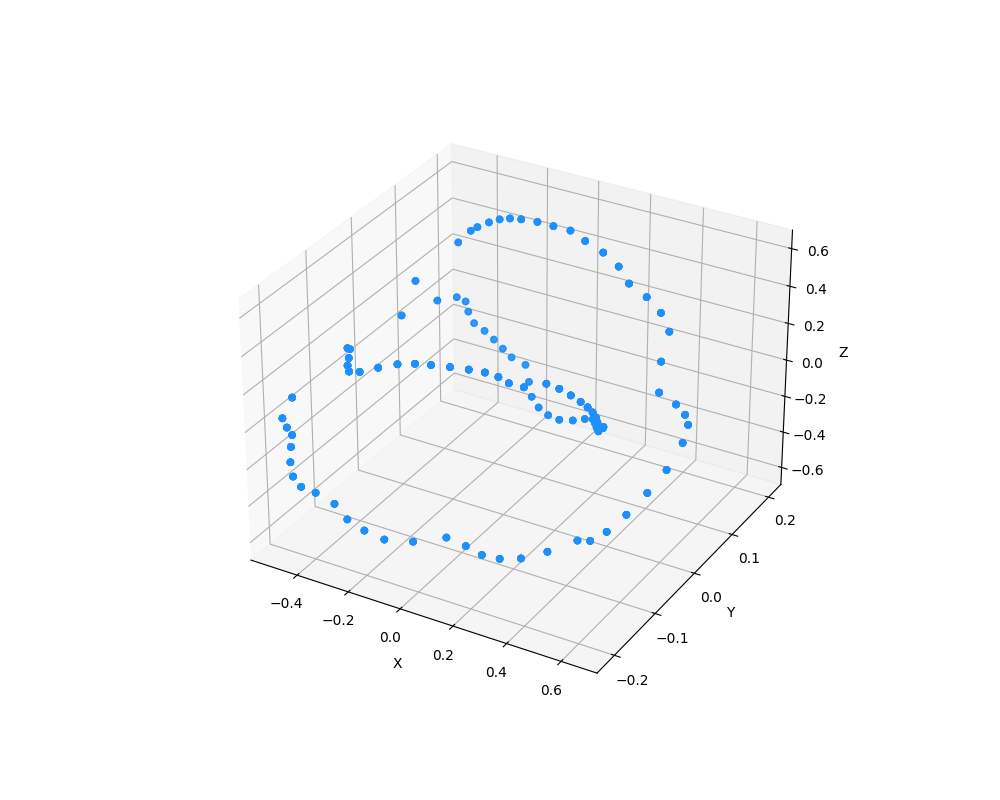} \\
		\midrule
		Run 9 & Run 4 & Run 1 \\
		\midrule
		MSE loss 0.000454 & MSE loss 0.000220 & MSE loss 0.000330 \\
		RTD loss 0.444102 & RTD loss 0.535345 & RTD loss 0.539877 \\
		Total loss 0.444556 & Total loss 0.535565 & Total loss 0.540207 \\
		\addlinespace[1em]
		\includegraphics[width=0.25\textwidth]{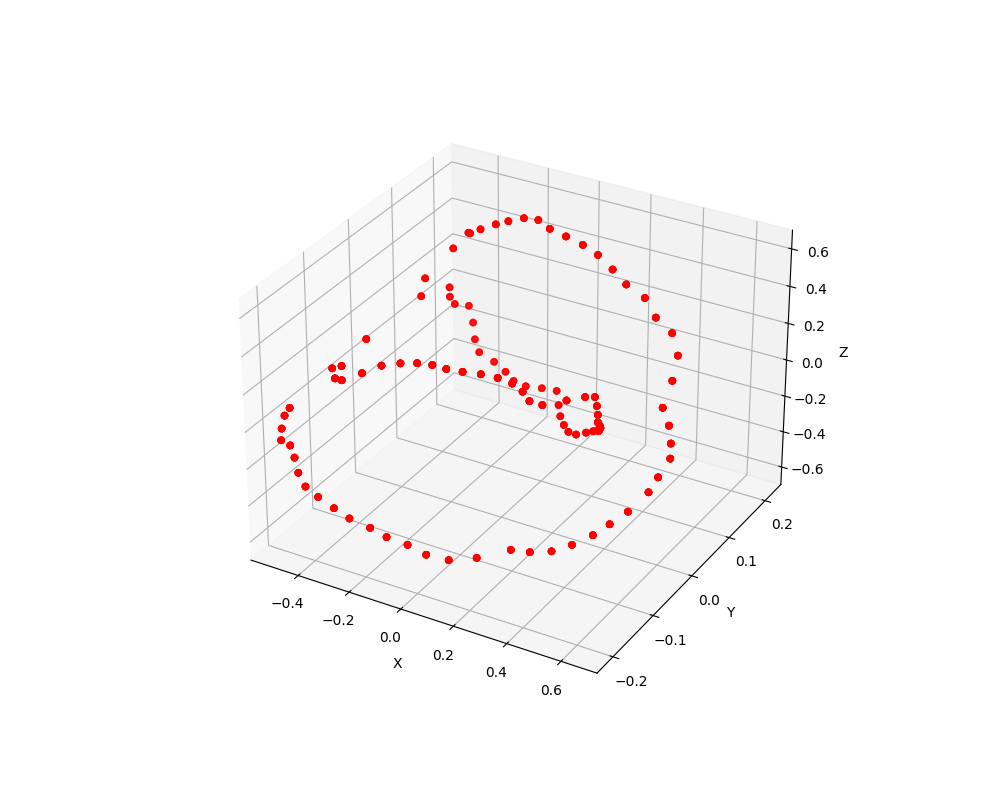} &
		& \\
		\midrule
		Run 10 & & \\
		\midrule
		MSE loss 0.000199 & & \\
		RTD loss 0.618562 & & \\
		Total loss 0.618761 & & \\
	\end{tabular}
	\caption{Top row: Runs with lowest MSE loss for the RTD autoencoder. Middle row: Runs with lowest RTD loss. Bottom row: Run with lowest total loss (MSE + RTD).}
	\label{fig:rtd_table_visuals}
\end{figure}

\noindent For the triple loop, also the RTD autoencoder could neither reconstruct the global structure nor torsion, as depicted in Figure~\ref{fig:bestrtdmod3}.

\begin{figure}[H]
	\centering
	\renewcommand{\arraystretch}{1.2}
	\setlength{\tabcolsep}{10pt}
	\begin{tabular}{ccc}
		\includegraphics[width=0.25\textwidth]{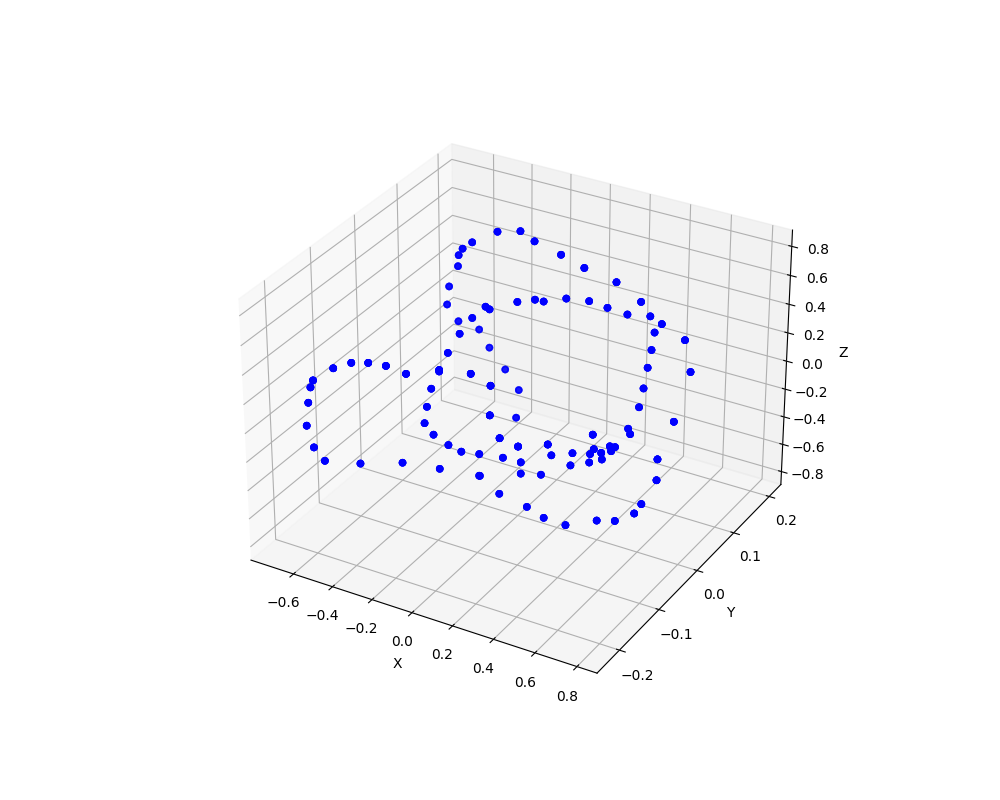} &
		\includegraphics[width=0.25\textwidth]{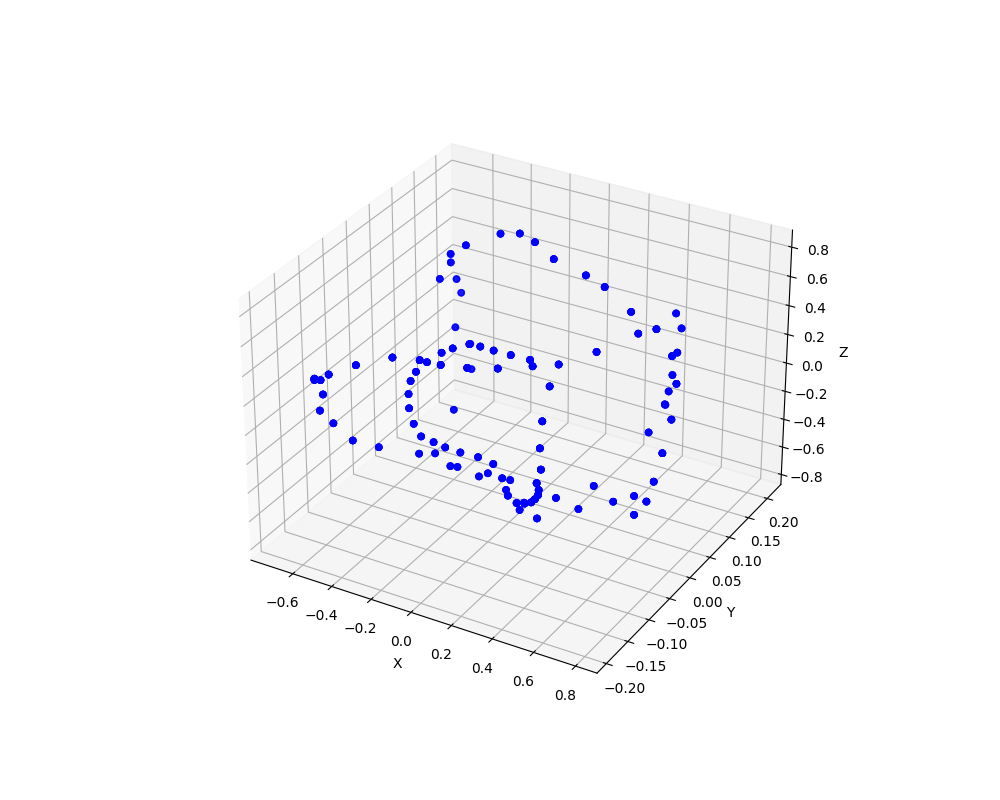} &
		\includegraphics[width=0.25\textwidth]{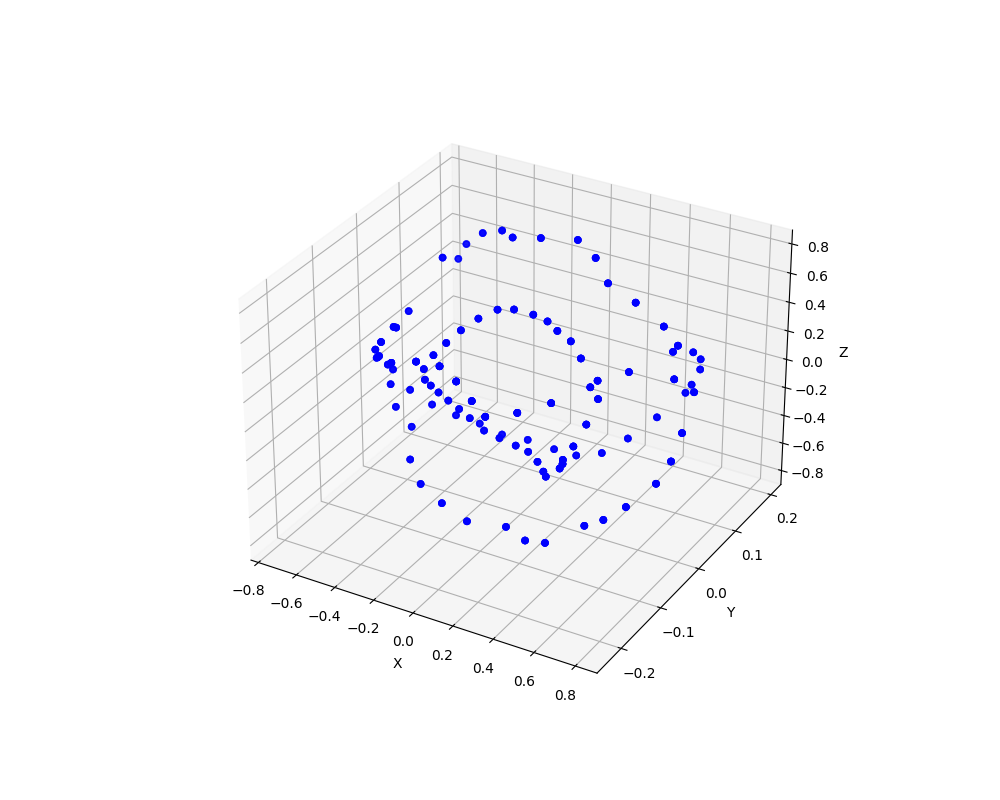} \\
		\midrule
		Run 28 & Run 11 & Run 12 \\
		\midrule
		MSE loss 0.002108 & MSE loss 0.002630 & MSE loss 0.002725 \\
		RTD loss 2.781865 & RTD loss 2.947375 & RTD loss 1.309629 \\
		Total loss 2.783973 & Total loss 2.950005 & Total loss 1.312354 \\
		\addlinespace[1em]
		\includegraphics[width=0.25\textwidth]{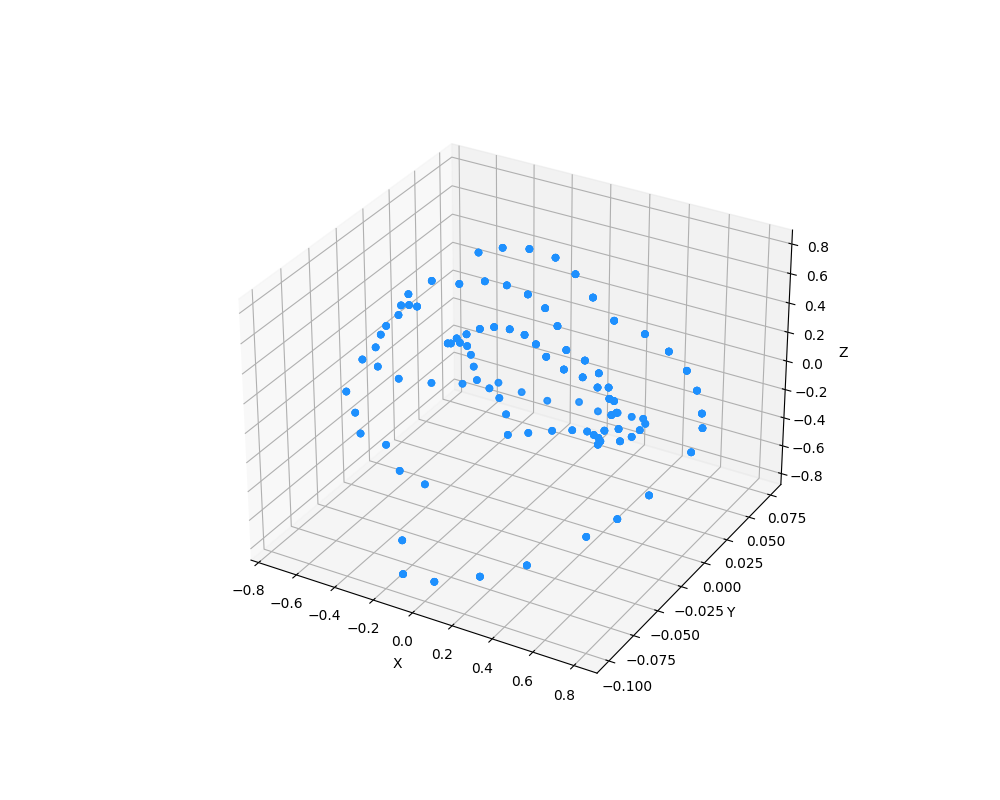} &
		\includegraphics[width=0.25\textwidth]{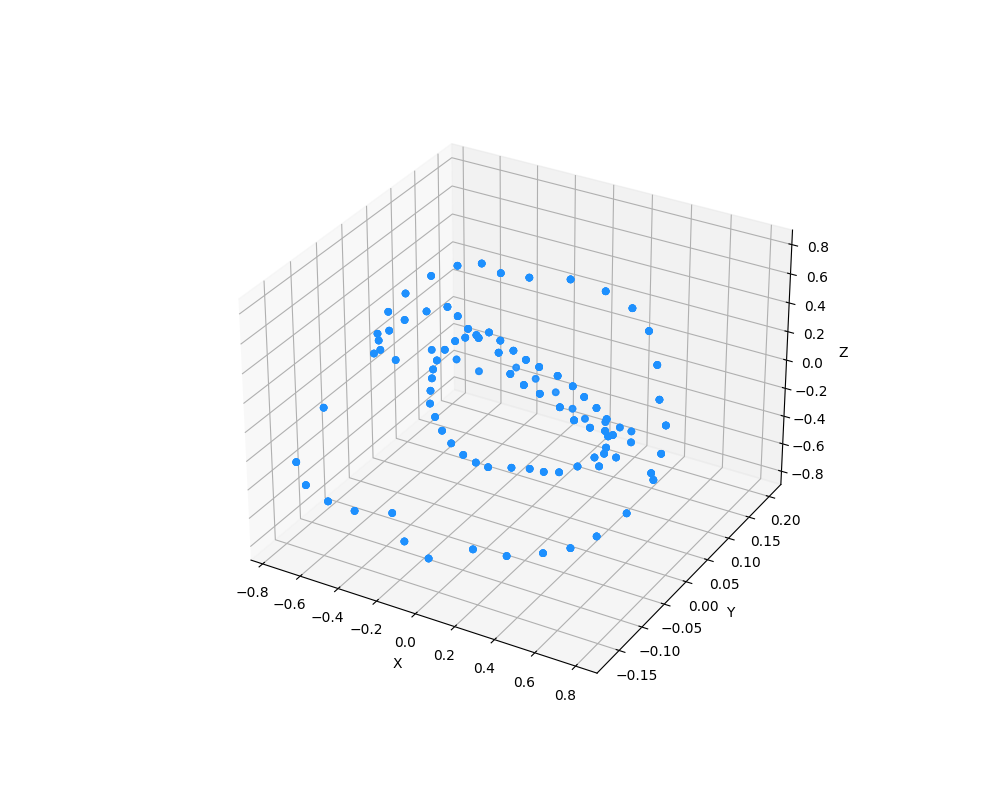} &
		\includegraphics[width=0.25\textwidth]{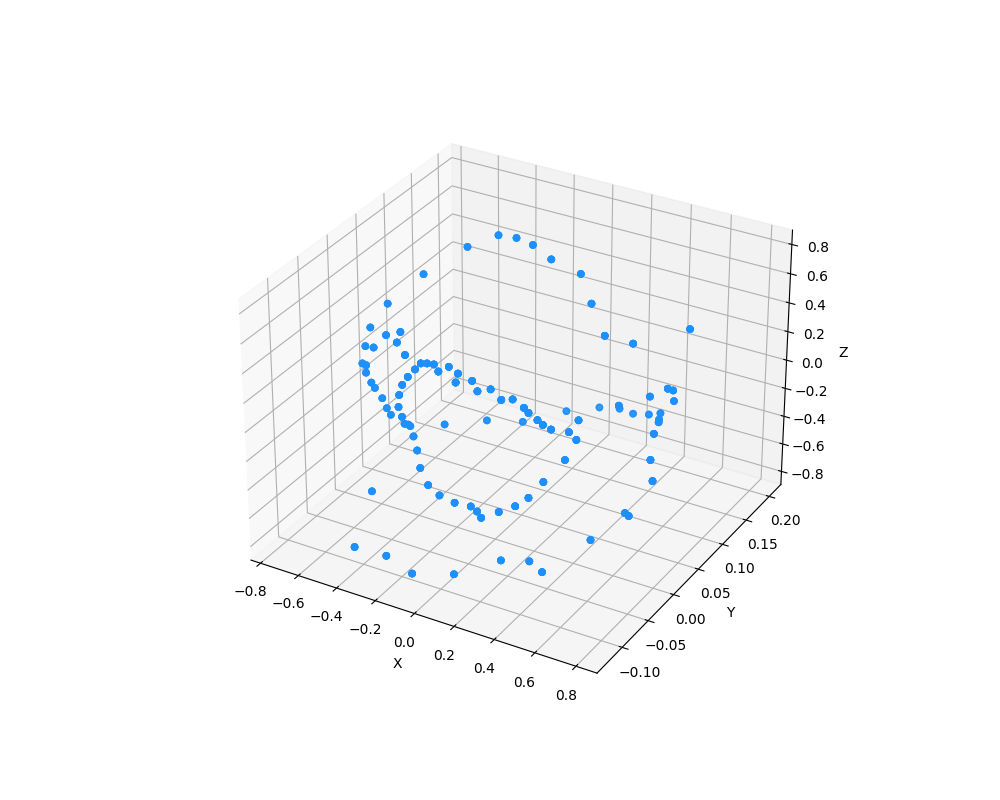} \\
		\midrule
		Run 10 & Run 7 & Run 30 \\
		\midrule
		MSE loss 0.005802 & MSE loss 0.003749 & MSE loss 0.003717 \\
		RTD loss 1.127720 & RTD loss 1.202779 & RTD loss 1.212620 \\
		Total loss 1.133522 & Total loss 1.206528 & Total loss 1.216337 \\
		\addlinespace[1em]
		\includegraphics[width=0.25\textwidth]{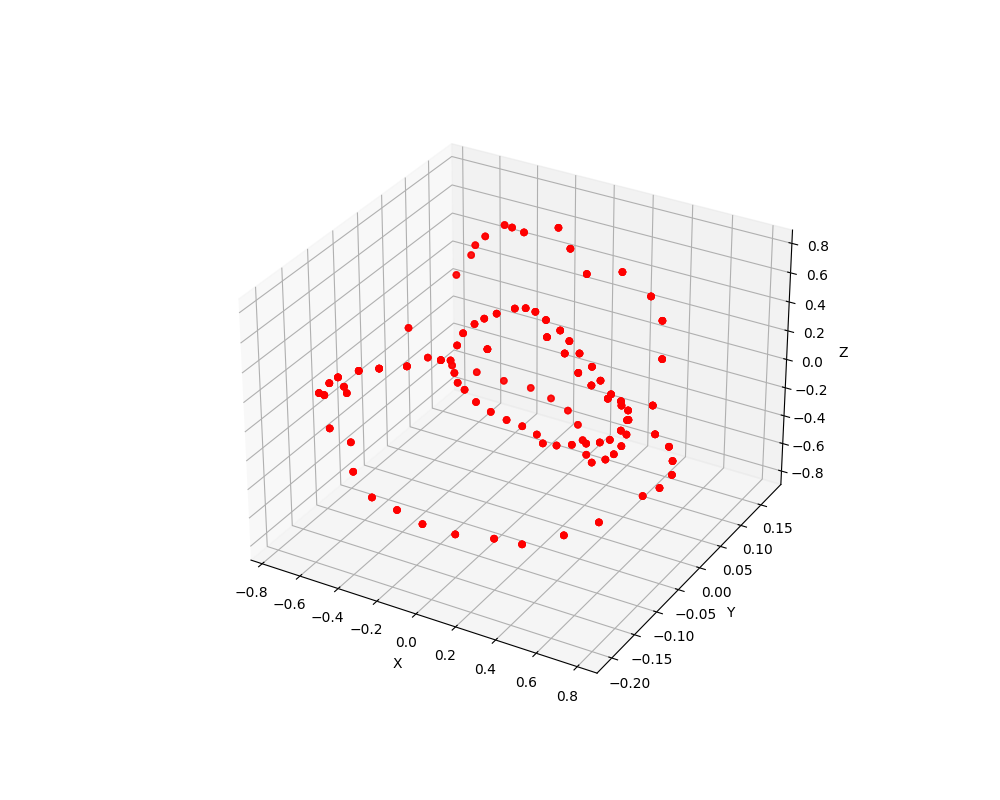} &
		& \\
		\midrule
		Run 8 & & \\
		\midrule
		MSE loss 0.003103 & & \\
		RTD loss 1.443048 & & \\
		Total loss 1.446151 & & \\
	\end{tabular}
	\caption{Top row: Runs with lowest MSE loss for the RTD-modified autoencoder. Middle row: Runs with lowest RTD loss. Bottom row: Run with lowest total loss (MSE + RTD).}
	\label{fig:bestrtdmod3}
\end{figure}

\clearpage

\section{Hyperparameters for High-Dimensional Point Cloud Reconstructions}
To ensure optimal reconstruction performance on the high-dimensional torsional point cloud data, we conducted a systematic hyperparameter search using the Optuna framework. The optimization targeted the minimization of the validation mean squared error (MSE) over 30 trials, each training a multi-layer perceptron autoencoder with varying architectural and training parameters.

\subsection{$d=10$}
Table~\ref{tab:combined-hyperparams-10d} summarizes the best hyperparameter configuration found during the hyperparameter optimization process for $10d$ input data and outlines the full search space explored. The optimal model used a latent dimension of 8, in line with our fixed dimensionality constraint. It employed moderately sized hidden layers and a relatively higher learning rate compared to the $13d$ setting, reaching a validation MSE of $1.19 \times 10^{-2}$ after 50 training epochs.

\begin{table}[H]
	\centering
	\caption{Optimal hyperparameters and corresponding search space for 10D torsional point cloud data}
	\label{tab:combined-hyperparams-10d}
	\begin{tabular}{@{}lll@{}}
		\toprule
		\textbf{Hyperparameter} & \textbf{Optimal Value} & \textbf{Search Space} \\
		\midrule
		Latent dimension        & 8                        & \{2, 3, 4, 5, 6, 7, 8\} \\
		Hidden layer 1 size     & 128                      & \{64, 128, 256\} \\
		Hidden layer 2 size     & 64                       & \{32, 64, 128\} \\
		Learning rate           & 0.001366                 & log-uniform $[10^{-4}, 10^{-2}]$ \\
		Weight decay            & $2.43 \times 10^{-5}$    & log-uniform $[10^{-6}, 10^{-3}]$ \\
		Batch size              & 32                       & \{32, 64, 80, 128\} \\
		\addlinespace
		\textbf{Validation MSE} & \textbf{0.011936}        & --- \\
		\bottomrule
	\end{tabular}
\end{table}

\subsection{$d=13$}

Table~\ref{tab:combined-hyperparams-13d} summarizes the best hyperparameter configuration found during the hyperparameter optimization process for $13d$ input data and outlines the full search space explored. The optimal model used a latent dimension of 32; however, we manually fixed the latent dimension to $8$ in our final setup. Additionally, the model employed relatively small hidden layers and a modest learning rate, achieving a validation MSE of $1.7 \times 10^{-5}$ after 50 training epochs.
\begin{table}[H]
	\centering
	\caption{Optimal hyperparameters and corresponding search space for 13D torsional point cloud data}
	\label{tab:combined-hyperparams-13d}
	\begin{tabular}{@{}lll@{}}
		\toprule
		\textbf{Hyperparameter} & \textbf{Optimal Value} & \textbf{Search Space} \\
		\midrule
		Latent dimension        & 32                      & \{4, 8, 16, 32\} \\
		Hidden layer 1 size     & 64                      & \{64, 128, 256\} \\
		Hidden layer 2 size     & 32                      & \{32, 64, 128\} \\
		Learning rate           & 0.00042818              & log-uniform $[10^{-4}, 10^{-2}]$ \\
		Weight decay            & $1.21 \times 10^{-5}$   & log-uniform $[10^{-6}, 10^{-3}]$ \\
		Batch size              & 128                     & \{32, 64, 80, 128\} \\
		\addlinespace
		\textbf{Validation MSE} & \textbf{0.000017}       & --- \\
		\bottomrule
	\end{tabular}
\end{table}

\section{Reconstructability of Torsion in High Dimensional Point Clouds}

The following tables report the minimum values of the MSE, topological, RTD, and total losses, along with an indication of whether torsion was detected in the output, for the 10-dimensional and 13-dimensional inputs discussed in Section~\ref{sec:highdpc}.

\subsection{Topological Autoencoder}
\begin{table}[H]
	\centering
		\begin{tabular}{|c|c|c|c|c|}
			\hline
			\textbf{Trial} & \cellcolor{orange!40}\textbf{MSE Loss} & \cellcolor{yellow!40}\textbf{Topo Loss} & \cellcolor{red!40}\textbf{Total Loss} & \textbf{Torsion} \\
			\hline
			\textbf{472} & 0.0164 & -      & -        & No Torsion \\
			\textbf{649} & 0.0164 & -      & -        & \textbf{Torsion: (2, 19485)} \\
			\textbf{626} & 0.0169 & -      & -        & No Torsion \\
			\hline
			\hline
			\textbf{681} & -      & 4.8263 & -        & No Torsion \\
			\textbf{119} & -      & 4.9249 & -        & No Torsion \\
			\textbf{478} & -      & 4.9636 & -        & No Torsion \\
			\hline
			\hline
			\textbf{472} & -      & -      & 0.0168   & No Torsion \\
			\textbf{649} & -      & -      & 0.0168   & \textbf{Torsion: (2, 19485)} \\
			\textbf{626} & -      & -      & 0.0174   & No Torsion \\
			\hline
		\end{tabular}
	\caption{Three trials with the lowest MSE, topological, and total loss values for $10d$ data $\eta = 0.000082$. Torsional trials are highlighted in bold.}
	\label{tab:topo10d}
\end{table}

\begin{table}[ht]
	\centering
		\begin{tabular}{|c|c|c|c|c|}
			\hline
			\textbf{Trial} & \cellcolor{orange!40}\textbf{MSE Loss} & \cellcolor{yellow!40}\textbf{Topo Loss} & \cellcolor{red!40}\textbf{Total Loss} & \textbf{Torsion} \\
			\hline
			\textbf{430} & 0.0325 & -      & -       & No Torsion \\
			\textbf{87}  & 0.0326 & -      & -       & No Torsion \\
			\textbf{119} & 0.0328 & -      & -       & No Torsion \\
			\hline
			\hline
			\textbf{516} & -      & 5.7202 & -       & \textbf{Torsion: (2, 18140)} \\
			\textbf{252} & -      & 6.1430 & -       & No Torsion \\
			\textbf{374} & -      & 6.2096 & -       & No Torsion \\
			\hline
			\hline
			\textbf{430} & -      & -      & 0.0331  & No Torsion \\
			\textbf{87}  & -      & -      & 0.0332  & No Torsion \\
			\textbf{119} & -      & -      & 0.0333  & No Torsion \\
			\hline
		\end{tabular}
	
	\caption{Three trials with the lowest MSE, topological, and total loss values for $13d$ data $\eta = 0.000091$.  Torsional trials are highlighted in bold.}
	\label{tab:topo13d}
\end{table}

\newpage
\subsection{RTD Autoencoder}

\begin{table}[htb]
	\centering
		\begin{tabular}{|c|c|c|c|c|}
			\hline
			\textbf{Trial} & \cellcolor{orange!40}\textbf{MSE Loss} & \cellcolor{cyan!30}\textbf{RTD Loss} & \cellcolor{red!40}\textbf{Total Loss} & \textbf{Torsion} \\
			\hline
			\textbf{478} & 0.0141 & -       & -        & No Torsion \\
			\textbf{227} & 0.0142 & -       & -        & No Torsion \\
			\textbf{794} & 0.0143 & -       & -        & No Torsion \\
			\hline
			\hline
			\textbf{794} & -      & 34.3383 & -        & No Torsion \\
			\textbf{710} & -      & 34.6614 & -        & \textbf{Torsion: (2, 18131)} \\
			\textbf{544} & -      & 34.7148 & -        & No Torsion \\
			\hline
			\hline
			\textbf{478} & -      & -       & 0.0145   & No Torsion \\
			\textbf{227} & -      & -       & 0.0146   & No Torsion \\
			\textbf{794} & -      & -       & 0.0147   & No Torsion \\
			\hline
		\end{tabular}
	
	\caption{Three trials with the lowest MSE, RTD, and total loss values for $\chi = 0.000011$. Torsional trials are highlighted in bold.}
	\label{tab:rtdloss10d}
\end{table}

\begin{table}[htb]
	\centering
		\begin{tabular}{|c|c|c|c|c|}
			\hline
			\textbf{Trial} & \cellcolor{orange!40}\textbf{MSE Loss} & \cellcolor{cyan!30}\textbf{RTD Loss} & \cellcolor{red!40}\textbf{Total Loss} & \textbf{Torsion} \\
			\hline
			\textbf{614} & 0.0271 & -       & -        & No Torsion \\
			\textbf{399} & 0.0274 & -       & -        & No Torsion \\
			\textbf{817} & 0.0279 & -       & -        & \textbf{Torsion: (2, 17840)} \\
			\hline
			\hline
			\textbf{455} & -      & 52.7213 & -        & \textbf{Torsion: (2, 15519)} \\
			\textbf{453} & -      & 53.7707 & -        & No Torsion \\
			\textbf{119} & -      & 54.7875 & -        & No Torsion \\
			\hline
			\hline
			\textbf{614} & -      & -       & 0.0282   & No Torsion \\
			\textbf{399} & -      & -       & 0.0285   & No Torsion \\
			\textbf{817} & -      & -       & 0.0289   & \textbf{Torsion: (2, 17840)} \\
			\hline
		\end{tabular}
	
	\caption{Three trials with the lowest MSE, RTD, and total loss values for $\chi = 0.000019$. Torsional trials are highlighted in bold.}
	\label{tab:rtdloss13d}
\end{table}

\end{document}